\documentclass[a4paper,12pt]{amsart}
\usepackage{amssymb,amsmath}
\let\amslrcorner\lrcorner
\usepackage{mathabx}
\usepackage{stmaryrd}
\usepackage{amscd,amsthm,amssymb}
\usepackage{enumerate}
\usepackage{color}
\usepackage[all,cmtip]{xy}
\usepackage{hyperref}
\usepackage{mathrsfs} 

\let\lrcorner\amslrcorner

\usepackage{latexsym}

\usepackage{hyperref}
\hypersetup{
    colorlinks=true, 
    linktoc=all,     
   linkcolor=blue,  
}

\def\sideremark#1{\ifvmode\leavevmode\fi\vadjust{\vbox to0pt{\vss
\hbox to 0pt{\hskip\hsize\hskip1em%
\vbox{\hsize2cm\tiny\raggedright\pretolerance10000%
\noindent {\color{red}{#1}}\hfill}\hss}\vbox to8pt{\vfil}\vss}}}%

\newcommand{\edz}[1]{\sideremark{#1}}

\def \bbR{\mathbb{R}}
\def\.{\cdot}

\def\d{{\mathrm d}}

\def\vs{\vskip .6cm}

\def\la{\langle}
\def\ra{\rangle}

\def\l{\lambda}
\def\s{\sigma}
\def\t{\tilde}
\def\beq{\begin{equation}}
\def\eeq{\end{equation}}
\def\bea{\begin{eqnarray*}}
\def\eea{\end{eqnarray*}}
\def\beaa{\begin{eqnarray}}
\def\eeaa{\end{eqnarray}}
\def\ba{\begin{array}}
\def\ea{\end{array}}

\def\o{\omega}

\def \L{\mathscr{L}}

\def \bbV{\mathbb{V}}
\def \bfi{\mathbf{i}}
\def \bfH{\mathbf{H}}
\def \bfs{\mathbf{s}}
\def \scrD{\mathscr{D}}

\def\Ric{\mathrm{Ric}}
\def\id{\mathrm{id}}
\def\be{\begin{equation}}
\def\ee{\end{equation}}
\def\tr{\mathrm{tr}}

\def\Hom{\mathrm{Hom}}
\def\Sym{\mathrm{Sym}}
\def\SL2{\mathfrak{sl}_2(\bbC)} 

\def\su{\mathfrak{su}}

\def\sp{\mathfrak{sp}}

\def\hh{\mathfrak{h}}

\def\C{\mathscr{C}}

\def\G{\mathrm{G}}
\def\H{\mathcal{H}}

\def\vol{\mathrm{vol}}
\def\Ker{\mathrm{Ker}}

\def\Sym{\mathrm{Sym}}
\def\scal{\mathrm{scal}}
\def\Id{\mathrm{id}}

\def\P{\mathrm{P}}

\def\grad{\mathrm{grad}}

\def\V{\mathcal{V}}
\def\pr{\mathrm{pr}}

\def\im{\mathrm{Im}}

\def \t5{\frac{1}{\sqrt{5}}}
\def \s5{1/\sqrt{5}}
\def \hh{h_{3,4}}
\DeclareMathOperator{\sH}{\star_{\mathcal{H}}} 
\DeclareMathOperator{\dH}{\d_{\mathcal{H}}} 
\DeclareMathOperator{\p}{p} 
\DeclareMathOperator{\PP}{P}
\DeclareMathOperator{\TT}{TT}

\DeclareMathOperator{\di}{d} 
\DeclareMathOperator{\spa}{span}

\DeclareMathOperator{\bbC}{\mathbb{C}}
\DeclareMathOperator{\bbI}{\mathbb{I}}
\DeclareMathOperator{\DD}{D}
\DeclareMathOperator{\tT}{t} 
\newtheorem{pro}{Proposition}[section]
\newtheorem{teo}[pro]{Theorem}
\newtheorem{lema}[pro]{Lemma}
\newtheorem{coro}[pro]{Corollary}
\theoremstyle{definition}

\newtheorem{rema}[pro]{Remark}

\title{The $\G_2$ geometry of $3$-Sasaki structures}
\author{Paul-Andi Nagy, Uwe Semmelmann}

\address{Paul-Andi Nagy\\
Center for Complex Geometry, Institute for Basic Science(IBS)\\
55 Expo-ro, Yuseong-gu \\
34126 Daejeon, South Korea
}
\email{paulandin@ibs.re.kr}

\address{Uwe Semmelmann\\
Institut f\"ur Geometrie und Topologie \\
Fachbereich Mathematik\\
Universit{\"a}t Stuttgart\\
Pfaffenwaldring 57 \\
70569 Stuttgart, Germany
}
\email{uwe.semmelmann@mathematik.uni-stuttgart.de}

\date{\today}

\begin{document}

\begin{abstract}
We initiate a systematic study of the deformation theory of the second Einstein metric $g_{\frac{1}{\sqrt{5}}}$ respectively the proper 
nearly $\mathrm{G}_2$ structure $\varphi_{\frac{1}{\sqrt{5}}}$ of a $3$-Sasaki manifold $(M^7,g)$. We show that infinitesimal Einstein deformations for $g_{\frac{1}{\sqrt{5}}}$ coincide with infinitesimal $\mathrm{G}_2$ deformations for $\varphi_{\frac{1}{\sqrt{5}}}$. The latter are showed to be parametrised 
by eigenfunctions of the basic Laplacian of $g$, with eigenvalue twice the Einstein constant of the $4$-dimensional base orbifold, via an explicit differential operator. In terms of this parametrisation we determine those infinitesimal $\mathrm{G}_2$ deformations which are unobstructed to second order.

\vs



\end{abstract}
\maketitle

\tableofcontents
%
\section{Introduction} \label{intro} 
\subsection{Background from $\G_2$ geometry} \label{mot}
A nearly $\G_2$ structure on an oriented compact manifold $(M^7,\vol)$ is given by a stable $3$-form $\varphi$ which is compatible with the orientation choice and additionally satisfies 
$\di\!\varphi=\tau_0 \star_{g_{\varphi}}\varphi$ for some non-zero $\tau_0 \in \bbR$, sometimes referred to as the torsion constant of the structure. Here $g_{\varphi}$ is the Riemannian metric induced 
by $\varphi$ which is necessarily Einstein with $\scal_{g_{\varphi}}=\tfrac{21}{8}\tau_0^2$. The focus in this paper is on instances when $\varphi$ is {\it{proper}} in the sense that $\mathfrak{aut}(M,g_{\varphi}) \subseteq \mathfrak{aut}(M,\varphi)$; equivalently $g_{\varphi}$ 
admits exactly one Killing spinor. In this situation the metric cone $(CM:= M \times \bbR_{+},r^2g_{\varphi}+\di\!r^2)$ has Riemannian holonomy equal to the subgroup $Spin(7) \subseteq SO(8)$.  
The homogeneous examples are the squashed $7$-sphere, the Berger space $SO(5)/SO(3)$ and the Aloff-Wallach spaces $N(k,l)$, see \cite{FKMS}. 
To the best of our knowledge the only known class of compact non-homogeneous examples occurs when $g_{\varphi}$ is obtained from the canonical variation of a $3$-Sasaki metric on $M$ by the following construction.

Consider a compact, oriented, manifold $M^7$ equipped with a $3$-Sasaki structure $(g,\xi)$ with triple of Reeb vector fields 
$\xi=(\xi_1,\xi_2,\xi_3)$. The distribution $\V:=\spa \{\xi_1,\xi_2, \xi_3\}$ is tangent to the leaves of a totally geodesic Riemannian foliation $\mathcal{F}$, referred to as the {\it{canonical}} foliation; the latter allows considering the canonical variation $g_s=s^2g_{\vert \V}+g_{\vert \H}, s>0$ of $g$ where $\H:=\V^{\perp}$. As it is well known the $3$-Sasaki metric $g$ is Einstein with $\Ric^{g}=6g$ and the second Einstein metric \cite{Besse} in the canonical variation is obtained 
for $s=\s5$, when $\Ric^{g_s}=54s^2g_s$. A remarkable feature of the Einstein metric $g_{\s5}$, due to working in dimension $7$, is to carry a proper nearly $\G_2$ structure 
determined by a canonically defined positive form $\varphi_{\s5} \in \Omega^3M$, with torsion constant $\tau_0=12/\sqrt{5}$.
See \cite{GS,FKMS} as well as the monograph \cite{BoGa} for more details. There is no scarcity of non-homogeneous $3$-Sasaki metrics on compact manifolds due to the construction in \cite{BGMann}. In this paper we initiate the programme of studying the Einstein and $\G_2$ deformation theory for the metric $g_{\s5}$.
\subsection{Background from deformation theory} \label{backgr-def}
Following \cite{AlS,NS} we review the deformation theory for proper nearly $\G_2$ structures $(M,\varphi,\vol)$ with torsion constant $\tau_0$. The infinitesimal deformation 
space is 
$$\mathcal{E}(\varphi):=\{\gamma \in \Omega^3_{27}(\varphi) : \star_{g_{\varphi}}\di\!\gamma=
-\tau_0\gamma\}$$
where we denote with $\Omega_{27}^3(\varphi)$ the space of sections of the $27$ dimensional, $\G_2$-irreducible, subbundle $\Lambda_{27}^3(\varphi) \subseteq \Lambda^3M$. 
The obstruction to deformation map $\mathbf{K}:\mathcal{E}(\varphi) \to \Lambda^1\mathcal{E}(\varphi)$ reads 
$$ \mathbf{K}(\gamma)\eta=\int_M \mathbf{P}(\gamma,\gamma) \wedge \star_{g_{\varphi}}\eta\,\vol,
$$
as introduced in our previous work \cite{NS}.
Here $\mathbf{P}:\Lambda^3_{27}(\varphi) \times \Lambda^3_{27}(\varphi) \to \Lambda^3_{27}(\varphi)$ is a bilinear bundle map which depends in an 
algebraically explicit way on the $\G_2$ form $\varphi$, see also section \ref{obs}. These objects describe the deformation theory for $\varphi$ to second order. Indeed, a small time curve $\varphi_t$ of nearly $\G_2$ structures with constant volume $\vol$ and $\varphi_0=\varphi$ satisfies 
\begin{equation*}
\gamma_1 \in \mathbf{K}^{-1}(0), \ \DD(\star_{g_{\varphi}}\gamma_2)=-\di\!\mathbf{P}(\gamma_1,\gamma_1)
\end{equation*}
where $\star_{g_{\varphi_t}}\varphi_t=\star_{g_{\varphi}}(\varphi+t\gamma_1+\frac{t^2}{2}\gamma_2)+O(t^3)$ and 
the operator $\DD:\Omega^4M \to \Omega^4M, D\alpha:=\di\!\widehat{\alpha}-\tau_0\alpha$ is essentially the linearisation of Hitchin's duality map $\alpha \in \Omega^4M \mapsto \widehat{\alpha} \in \Omega^3M$ defined in \cite{H}. In particular $\mathbf{K}^{-1}(0)$ describes infinitesimal deformations in $\mathcal{E}(\varphi)$ which are unobstructed to second order. We will use these results to see how deformation theory at second order behaves on large classes of non-homogeneous examples e.g. 
the class of proper nearly $\G_2$ structures $\varphi_{\s5}$ considered above. As far as the homogeneous examples are concerned, record that according to \cite{AlS} the squashed $7$-sphere and the Berger space do not admit non-trivial infinitesimal $\G_2$ deformations 
whereas for the Aloff-Wallach space $(N(1,1),\varphi_{\s5})$ we have, again by \cite{AlS}, $\mathcal{E}(\varphi_{\s5})\neq 0$ but the zero locus of $\mathbf{K}$ is trivial i.e. the nearly $\G_2$ structure is rigid, see \cite{NS}.

As any curve of nearly $\mathrm{G}_2$ structures induces a curve of Einstein metrics we briefly review the facts from Einstein deformation theory needed in this work. Given a compact manifold $M$ equipped with an Einstein metric $g$ with Einstein constant $E$ the space of essential infinitesimal Einstein deformations  of $g$ is, according to \cite{Ko1}, 
$$ \mathcal{E}_{\mathrm{ess}}(g)=\TT(g) \cap \ker(\Delta_L^g-2E).
$$
Here the space of TT-tensors $\TT(g):=\{h\in \Gamma(\Sym_0^2(M,g)): \delta^{g} h=0\}$ and the divergence operator 
$\delta^{g}$ respectively the Lichnerowicz Laplacian  $\Delta_L^{g}$ are computed w.r.t. the metric $g$.
See \cite{Besse} or section \ref{CF} of the paper for the definitions of these operators. Also recall the well-known fact, see e.g.\cite{Besse}[pages 131-32], that after applying a gauge transformation, any curve of Einstein metrics $g_t$, with Einstein constant $E$, satisfies $g^{-1}\frac{\di\!g_t}{\di\!t}_{\vert t=0} \in \mathcal{E}_{\mathrm{ess}}(g)$. Finally, if $(M,\varphi,\vol)$ is nearly $\mathrm{G}_2$ we have $g_{f^{\star}\varphi}=f^{\star}g_{\varphi}$ for any volume preserving diffeomorphism. As a consequence, the normalisation 
from \cite{NS} for curves of nearly $\mathrm{G}_2$ structures produces, at first order, essential infinitesimal Einstein deformations.

\subsection{Main results}Throughout the rest of this paper by a $\mathrm{G}_2$ deformation we mean a deformation of $\varphi_{1/\sqrt{5}}$ through a curve of nearly $\mathrm{G}_2$ structures with constant volume $\vol$; similarly, the infinitesimal deformations of the nearly $\mathrm{G}_2$ form $\varphi_{1/\sqrt{5}}$ will be referred to as infinitesimal $\mathrm{G}_2$ deformations.  

Our first main result is a purely analytic description of infinitesimal Einstein deformations of $g_{\s5}$ respectively infinitesimal $\G_2$ deformations. Furthermore we give a simple expression for the obstruction to deformation polynomial of the nearly $\G_2$ structure 
$\varphi_{\s5}$.
Infinitesimal Einstein deformations are assumed to be essential as explained above 
and are thus parametrised by the space  
$$\mathcal{E}_{\mathrm{ess}}(g_{\s5}):=\TT(g_{\s5}) \cap  \ker (\Delta_L^{g_{\s5}}-\tfrac{108}{5}).$$
We show that the deformation theory of $g_{\s5}$ strongly depends on the geometry of the canonical foliation $\mathcal{F}$ and turns out to be entirely governed by the spectrum of its scalar basic Laplacian   
$$\Delta_b:C^{\infty}_{b}M \to C^{\infty}_{b}M, \ \Delta_b:=\left. \Delta^g \right|_{C^{\infty}_{b}M}$$
where $C^{\infty}_{b}M:=\{f \in C^{\infty}M : \L_{\xi_a}f=0, a=1,2,3\}$ denotes the space of basic functions on $M$. The basic 
Laplacian can be alternatively computed from any metric in the canonical variation of $g$ or from the scalar sub-Laplacian $\Delta_{\H}$ 
introduced later on in the paper, see section \ref{HoO}.
\begin{teo} \label{main1}
Let $M^7$ be compact and equipped with a $3$-Sasaki structure $(g,\xi)$. 
\begin{itemize}
\item[(i)] the space $\mathcal{E}_{\mathrm{ess}}(g_{\s5})$ of infinitesimal Einstein deformations for $g_{\s5}$ is isomorphic to the infinitesimal nearly $\G_2$ deformation 
space $\mathcal{E}(\varphi_{\s5})$
\item[(ii)] the map $\varepsilon : \ker(\Delta_b-24) \to \mathcal{E}(\varphi_{\s5})$ 
given by 
\begin{equation*}
\varepsilon(f)=\tfrac{\sqrt{5}}{6}\L_{\grad f}\varphi_{\s5}\,+\,\tfrac{12}{\sqrt{5}}f(\varphi_{\s5}-\tfrac{2}{5\sqrt{5}}\xi^{123})-2 \grad f \lrcorner \vol_{\H}
\end{equation*}
is a linear isomorphism, where $\vol_{\H}$ is the horizontal volume form 
\item[(iii)] the set of infinitesimal $\G_2$ deformations which are unobstructed to second order is given by 
\begin{equation*}
\mathbf{K}^{-1}(0)=\varepsilon (\{f \in  \ker(\Delta_b-24) : f^2 \perp  \ker(\Delta_b-24) \})
\end{equation*}
where orthogonality is meant in $L^2$-sense.
\end{itemize}
\end{teo}

The identification between deformation spaces in (i) is given by the vector bundle isomorphism $\bfi: \Sym^2_0(M,g_{\s5}) \to \Lambda^3_{27}(\varphi_{\s5})$; see section \ref{LL} for definitions and details. To explain some of the numerics above record that the antiselfdual (ASD) Einstein orbifold  $(N:=M/\mathcal{F},g_N)$ satisfies $\Ric^{g_N}=12g_N$. 

A remarkable feature of the operator $\varepsilon$ is that it allows parametrising infinitesimal $\G_2$, hence Einstein deformations by (i) above, only in terms of Laplace eigenfunctions on $N$, for twice the Einstein constant, by using the foliated structure. Our operator $\varepsilon$ generalises to an embedding  of eigenfunctions of the Laplacian acting on $C^{\infty}_{b}M$ into trace and divergence free eigentensors for the Lichnerowicz  Laplacian. It should be compared with the operator $S$ from \cite{CaoHe} which maps eigenfunctions of the scalar Laplacian into divergence 
free -- but not necessarily trace free -- eigentensors for $\Delta_L$. A posteriori it follows from (ii) in Theorem 
\ref{main1} that infinitesimal Einstein deformations are $\su(2)$-invariant, that is 
invariant under the Reeb vector fields $\xi_1,\xi_2,\xi_3$. This indicates that $\G_2$ deformations by curves could be showed to be 
$\su(2)$-invariant, which is sometimes an a priori hypothesis in deformation theory, see \cite[Theorem 3.1]{coev2} as well as \cite{Coev1}. 
  
The operator $\varepsilon$ parametrising $\mathcal{E}(\varphi_{\s5})$ is second 
order in the derivatives of $f$. In this sense it is somewhat surprising to see that the obstruction polynomial involves integrating only polynomial expressions in $f$. By (iii) in Theorem \ref{main1} infinitesimal $\G_2$ deformations $\varepsilon(f)$ which are unobstructed to second order satisfy, in particular, 
$$\int_M f^3\vol=0.$$
Pausing for a short digression based on this fact, we indicate how 
the deformation theory of the nearly $\G_2$ structure $\varphi_{\s5}$ may relate to the dynamic stability, transversally understood, of the ASD Einstein orbifold $(N^4,g_N)$. 
Whilst none of the technical details of orbifold stability will be looked at in this paper we draw the picture duplicating the smooth setup.
The criterium in \cite[thm.1.7]{K1}, see also \cite{KS}, ensures that $(N^4,g_N)$ is dynamically unstable provided there exists $f\in \ker(\Delta^{g_N}-24)$ satisfying $\int_Nf^3\vol_N \neq 0$, in which case the infinitesimal $\G_2$ deformation $\varepsilon(f)$ is obstructed to second order. 

Note that on Hermitian symmetric spaces of arbitrary dimension cubic integrals for eigenfunctions of the scalar Laplacian 
with eigenvalue twice the Einstein constant, or equivalently Killing potentials, 
have been explicitly computed in \cite{HMW} by the Duistermaat-Heckmann localisation formula. Based on this we obtain a new geometric proof for the $\G_2$ rigidity of the Aloff-Wallach space, previously considered in \cite{NS,DWI}. 
\begin{rema} \label{pb1}
It is an open problem to decide if small time Einstein deformations of $g_{\s5}$ coincide with $\G_2$-deformations of $\varphi_{\s5}$. This is the case at order $1$ by part (i) in Theorem \ref{main1}. It is however unclear if even at second order the obstruction to Einstein deformation as developed in 
\cite{Ko1} is the same as the obstruction to $\G_2$ deformation given by $\mathbf{K}$. Evidence that may not be automatically true is provided by the metric $g$ which is rigid as a $3$-Sasaki metric \cite{PP}; however $g$ admits deformations through Sasaki-Einstein metrics \cite{Coev1,coev2}. This contrasts with small time Einstein deformations of K\"ahler metrics, which stay K\"ahler provided certain topological conditions are satisfied, see \cite{Koiso3}. 
In particular the Einstein rigidity of $g_{\s5}$ on the Aloff-Wallach space $N(1,1)$ remains an open problem.
\end{rema}
Recall that an Einstein metric with Einstein constant $E$ is called linearly unstable \cite{Koiso1980} if its Lichnerowicz Laplacian $\Delta_L$ acting on TT tensors admits eigenvalues smaller than $2E$. If that is the case the direct sum of the eigenspaces corresponding to such eigenvalues is called the space of destabilising directions. From general principles, see \cite{Besse}[Fig.9.72], the Einstein metric $g_{\s5}$ is linearly unstable; see \cite{WW1} for a rigorous explanation of this fact.

The techniques used to obtain part (i) in Theorem \ref{main1} generalise to precisely measure instability for the second Einstein 
metric $g_{\s5}$ built from the $3$-Sasaki structure $(g,\xi)$ on $M$ as follows.
\begin{teo} \label{main2}
Assume that $g$ does not have constant sectional curvature. The space of destabilising directions for $g_{\s5}$ is canonically isomorphic to 
$$\bbR \oplus \bfH_4^{-} \oplus \bigoplus \limits_{16 < \nu < 24}^{}\ker(\Delta_b-\nu).$$
The corresponding eigenvalues for $\Delta_L^{g_{\s5}}$ are 
$ \ \tfrac{28}{5}, \ \ \tfrac{76}{5}, \  \ \nu-\tfrac{4}{5}\sqrt{1+5\nu}+\tfrac{32}{5}$.
\end{teo}
The summand $\bbR$ is geometrically embedded 
via the tensor $\hh:=4\id_{\V}-3\id_{\H}$ which turns out to be a Killing tensor \cite{HMS}[Propn.7.2] and has been shown to provide a destabilising direction in \cite{WangW2}. In fact we show in section \ref{AW} that the whole space of unstable directions for the Aloff-Wallach space $(N(1,1),g_{\s5})$ is spanned by $\hh$. The space $\mathbf{H}_4^{-}$ consists of equivariant harmonic forms; it is equivalently described as the space of basic eigentensors, w.r.t. the canonical foliation $\mathcal{F}$, for the Lichnerowicz Laplacian of the metric 
$g_{\s5}$. At the same time $\mathbf{H}^{-}_4$ is canonically embedded in $H^{0,1}(Z,T^{0,1}Z \otimes K_Z^{-\frac{1}{2}})$, where $Z$ is the twistor space of $N=M/\mathcal{F}$ and $K_Z$ is the canonical orbibundle of the K\"ahler orbifold $Z$. The remaining function eigenspaces in Theorem \ref{main2} embed via an explicit operator, similar to $\varepsilon$, defined in Proposition \ref{embed1}. We only consider eigenvalues $\nu>16$ since $\Delta_b>16$ on non-constant basic functions by \cite{LeeR}, provided $g$ does not have constant sectional curvature. Existence of eigenvalues $\nu<24$ for the basic Laplacian on functions implies $\nu$-instability in the sense of \cite{CaoHe}[Cor.1.3] of the base orbifold $(N^4,g_N)$.
\begin{rema} \label{pb2}
It is an open problem to decide whether eigenvalues $\nu$ for the basic Laplacian satisfying $\nu<24$ do exist, with the exception of 
the Aloff-Wallach space $N(1,1)$ which has base 
$N=\overline{\bbC P}^2$. However, when the base $N$ is toric, we expect that combining techniques as those used in \cite{HM} with the local classification of toric 
selfdual Einstein metrics in \cite{CaPe} will shed light on this problem. 
\end{rema}
To conclude we observe that ordering the unstable eigenvalues in Theorem \ref{main2} yields 
\begin{coro} \label{main3}
The Lichnerowicz Laplacian of $g_{\s5}$ acting on the space $\TT(g_{\s5})$ of trace and divergence free symmetric tensors 
satisfies 
$$ \Delta_L^{g_{\s5}} \geq \tfrac{28}{5}.
$$
The eigenspace corresponding to the minimal eigenvalue $\tfrac{28}{5}$ is spanned by $\hh$.
\end{coro}
In particular $\Delta_L$ is positive on TT tensors with first eigenvalue $\lambda_1^L=\tfrac{28}{5}$. This result is an optimal improvement of the upper bound $\lambda_1^L \leq \tfrac{28}{5}$ which has been established in \cite{WangW2} by computing the Rayleigh-Ritz quotient of the tensor  $\hh$. In particular, Corollary \ref{main3} recovers stability for $g_{\s5}$ in the sense of the Freund-Rubin compactification as used in generalised black hole theory. See \cite{GiPope}[sectn. IV.C] as well as \cite{CaoHe,BaHM} for definitions and further related results. Note that in the last two references all Laplace type operators are defined to be negative. Indeed, stability in the aforementioned sense amounts to the lower bound $\lambda_1^L \geq \tfrac{27}{5}$ which is clearly satisfied by Corollary \ref{main3}.
\begin{rema} \label{rmk6} 
As already noted dynamic instability for the orbifold $(N^4,g_N)$ is related to the existence of non-integrable infinitesimal $\G_2$ deformations of $(M^7,g_{\s5})$. 
However, the dynamic stability of $(M^7,g_{\s5})$ itself is unrelated to the $\G_2$ deformation problem since 
$ \ker(\Delta^{g_{\s5}}-\frac{108}{5}) \cap C^{\infty}M= \ker(\Delta_b-\frac{108}{5})
$
as shown in the body of the paper, see Remark \ref{rmk-intro}. By (ii) in Theorem \ref{main1} the eigenvalue $\frac{108}{5} \in (16,24)$ for the basic Laplacian, if it exists, does not turn up in deformation theory but rather as a destabilising direction.
\end{rema}
\subsection{Outline of the paper} \label{outln}In section \ref{2E} we briefly review those facts from $3$-Sasaki geometry which will be used in this paper; following \cite{AlS} we explain how the study of infinitesimal Einstein and $\G_2$ deformations in the spaces $\mathcal{E}_{\mathrm{ess}}(g_{\s5})$ and $\mathcal{E}(\varphi_{\s5})$, together with that of unstable directions, translates into solving spectral problems for the $3$-form Laplacian of $g_{\s5}$ acting on 
 $\Omega^3_{27}(\varphi_{\s5})$. The first step in solving these spectral problems, performed in section \ref{a-id}, is spelling out the algebraic structure of $\Lambda^3_{27}(\varphi_s),s>0$ w.r.t. to the canonical decomposition $TM=\V \oplus \H$. In section \ref{specD} we work out, for arbitrary $s$, the block structure of $\star_{g_s}\di$ and of the form Laplacian of $g_s$ w.r.t to the canonical decomposition. Block structure results are well known essentially only for Sasaki and contact metrics, \cite{Ta,Ru} when the canonical foliation has $1$-dimensional leaves. In our setup 
$\mathcal{F}$ has $3$-dimensional leaves making that the decomposition of form spaces has more components. The generators of the Lie algebra $\su(2)$ produce more -- by comparaison to $\mathfrak{u}(1)$ actions -- invariant operators 
relevant for the block structure of the Laplacian; their algebraic structure is derived from $\su(2)$ representation theory. In section \ref{EDS} we essentially show that the spectral theory of $\star_{g_s}\di$ acting on $3$-forms reduces 
to the study of suitably defined spaces of harmonic forms and the spectral theory of perturbations of the horizontal Laplace operator 
$\Delta_{\H}$ acting on 
$\Omega^1(\H,\bbR^3)$. In section \ref{numerics} we prove lower bounds for the spectrum of $\Delta_{\H}$ acting on weighted $\su(2)$-invariant spaces of functions and horizontal $1$-forms. In section \ref{inf-defE} the representation theory of $\su(2)$ and the eigenvalue estimates for $\Delta_{\H}$ are put together to prove 
Theorems \ref{main1} and \ref{main2} with the exception of the obstruction part.  The latter is proved in section \ref{obs} by explicitly computing the polynomial $P$ on the subspace of $\Omega^3_{27}(\varphi_{\s5})$ spanned by $\varepsilon(f)$ with $f \in \ker(\Delta_b-24)$. Section \ref{TLL} contains the computation of the basic Licherowicz Laplacian w.r.t. the Riemannian foliation $\mathcal{F}$ which we use to apply Theorem \ref{main1} and Theorem \ref{main2} to the Aloff-Wallach space $N(1,1)$.

To conclude we list some directions for future research. In \cite[sectn.5.3]{KY} deformed Donaldson Thomas instantons 
have been used to define explicit deformations of co-calibrated $\G_2$ structures; furthermore the proper nearly $\G_2$ structure 
$(M^7,g_{\s5})$ supports many examples of such instantons \cite{LoOl}. We plan to understand how the deformation theory of $\varphi_{\s5}$ interacts with the study of instantons, possibly for more general principal bundles, as considered in \cite{BO} for the Aloff-Wallach spaces 
$N(k,l)$.\\
$\\$
{\bf{Acknowledgements:}} Paul-Andi Nagy was supported by the Institute for Basic Science (IBS-R032-D1).This research has also been supported by the Special Priority Program
SPP 2026 `Geometry at Infinity' funded by the DFG. It is a pleasure to thank Tommy Murphy for many useful conversations on stability.
\begin{theindex}
\item Algebraic isomorphisms
\subitem $\bfi:\Sym_0^2(M,g_s) \to \Lambda^3_{27}(\varphi_s)$, page 11
\subitem $\iota_s:V^p\H \to \Lambda^pM$ for $p=3,2$, page 14 and \eqref{iota2}
\subitem $\kappa_s:\Lambda^1(\H,\bbR^3) \oplus \Lambda^2_{sym}(\H,\bbR^3) \to \Lambda^3_{27}(\varphi_s)$, Lemma \ref{dec-3sas}
\subitem $\bfs:\Lambda^{-}(\H,\bbR^3) \to \Sym^2_0\H$, \eqref{s-iso} 

\item Bundles of horizontal forms 
\subitem $V^p\H, p=2,3$ page 21, page 14
\subitem $\Lambda^2_{sym}(\H,\bbR^3)$, Lemma \ref{dec-3sas}
\subitem $\Lambda^{+}_{sym}(\H,\bbR^3)$, page 15

\item Differential forms
\subitem $\varphi_s$ and $\star_s \varphi_s$, page 11
\subitem $\widetilde{\varphi}_s$, page 15

\item Eigenspaces 
\subitem $E_{\lambda}=\ker(\star_s \di-\lambda) \cap \Omega^3_{27}(\varphi_s)$, page 13
\subitem $F_{\lambda}=\ker(\scrD-\lambda)$, page 29
\subitem $F_{\lambda}^{\perp}$, Proposition \ref{semi2}
\subitem $\mathscr{E}(t,k)=\ker(\Delta_{\H}-t) \cap \Omega^1_{k}\H$, page 38, proof of Proposition 
\ref{ker-gen}

\item Embedding operators 
\subitem $\varepsilon^{\pm}_{\nu}$, Proposition \ref{embed1}
\subitem $\varepsilon$, \eqref{e-bis}
\item Laplacians 
\subitem basic, $\Delta_b$, page 4
\subitem horizontal, $\Delta_{\H}$, page 18
\subitem Hodge, $\Delta^g$, page 13
\subitem Lichnerowicz $\Delta_L^g$, sections \ref{LL} and \ref{CF}
\subitem basic Lichnerowicz, $\Delta_L^b$, page 46
\subitem connection Laplacian, $\overline{\Delta}$, Remark \ref{deltabar}

\item Lie algebra actions
\subitem $\su(2)$ on $\Omega^{\star}\H$, \eqref{ro}
\subitem $\sp(1)$ on $\Omega^1\H$, \eqref{spp1}

\item Obstruction to deformation map 
\subitem $\mathbf{K} :\mathcal{E}(\varphi) \to \Lambda^1\mathcal{E}(\varphi)$, page 3

\item Operators constructed from representations
\subitem $\tT :\Lambda^{\star}(\H,\bbR^3) \to \Lambda^{\star}\H$, page 14 and \eqref{tT-alt}
\subitem $\mathbb{I} : \Lambda^1\H \to \Lambda^1(\H,\bbR^3)$, \eqref{bbId}
\subitem $\P:\Omega^1(\H,\bbR^3) \to \Omega^1(\H,\bbR^3)$, Lemma \ref{blgen-1}
\subitem Casimir operator $\C$, page 16
\subitem $\L_{\xi} : \Omega^{\star}\H \to \Omega^{\star}(\H,\bbR^3)$, page 16
\subitem $C:\Omega^{\star}(\H,\bbR^3) \to \Omega^{\star}(\H,\bbR^3)$, page 16
\subitem $\p:\Omega^1\H \to \Omega^1\H$, page 16

\item Perturbations of $\Delta_{\H}$
\subitem $\mathscr{G}^{s}:\Omega^1(\H,\bbR^3) \to \Omega^1(\H,\bbR^3)$, Proposition \ref{comp1}
\subitem $\scrD=\Delta_{\H}+5\C-2\p:\Omega^1\H \to \Omega^1\H$, page 27

\item Polynomial $\mathrm{G}_2$ invariants 
\subitem $p :\Lambda^3_{27}(\varphi_s) \times \Lambda^3_{27}(\varphi_s) \to \Sym^2(TM,g_s) $, page 41 
\subitem $\mathbf{P} :\Lambda^3_{27}(\varphi_s) \times \Lambda^3_{27}(\varphi_s) \to \Lambda^3_{27}(\varphi_s)$, page 3 and page 41
\subitem $P \in \Sym^3\Lambda^3_{27}(\varphi_s)$, page 41

\item Spaces of horizontally harmonic forms
\subitem $\mathbb{H}, \mathbb{H}^{\pm}$, page 27
\subitem $\mathbb{H}_{\lambda}$, \eqref{h-spc}
\subitem $\mathbf{H}^{\pm}_{\lambda}, \mathbf{H}_{\lambda}$, \eqref{h-spc}

\item Spaces of horizontal differential forms
\subitem $\mathscr{S}=\iota_s^{-1}\Omega^3_{27}(\varphi_s) \subseteq \mathbb{V}^3\H$, \eqref{scrS}
\subitem $\Omega^1_{inv}\H$, page 30 
\subitem $\Omega^{1}_{\perp}\H$, page 30
\subitem $\Omega^1_o(\H,\bbR^3)$, Proposition \ref{semi2}

\item $\mathrm{TT}$-tensors
\subitem $\mathrm{TT}(g)$, page 4
\subitem $\mathrm{TT}_b(\mathcal{H})$, page 37

\item Weighted differential forms 
\subitem $\Omega^1_k(\H)$, \eqref{we}
\end{theindex}
\section{Preliminaries} \label{2E}
\subsection{Elements of $3$-Sasaki geometry} \label{3sas}
We only recall those facts from $3$-Sasaki geometry which will be strictly needed in what follows. For general theory and equivalent formulations see \cite{BoGa}. Let $(M^7, g)$ be a compact Riemannian manifold with a 3-Sasaki structure defined by three 
Killing vector fields $\xi_1, \xi_2, \xi_3$ satisfying $g(\xi_a,\xi_b)=\delta_{ab}$ and 
\begin{equation} \label{su2}
\begin{split}
& [\xi_1,\xi_2]=2\xi_3, \ [\xi_2,\xi_3]=2\xi_1, \ [\xi_3,\xi_1]=2\xi_2. 
\end{split}
\end{equation}
The distributions $\V:=\spa \{\xi_1,\xi_2,\xi_3\}$ respectively $\H:=\V^{\perp}$ will be referred to as the vertical respectively the horizontal distributions. The vertical distribution induces a Riemannian foliation with totally geodesic leaves, denoted with 
$\mathcal{F}$ in what follows. In addition the leaf space $N:=M\slash \mathcal{F}$ has the structure of a compact $4$-dimensional 
orbifold. The differential geometric properties of $g$ are encoded in the structure equations for the coframe $\xi^a:=g(\xi_a, \cdot), a=1,2,3$ which read 
\begin{equation} \label{str-xi}
\di\!\xi^a=-2\xi^{bc}+2\omega_a
\end{equation}
with cyclic permutations on $abc$, where $\omega_1,\omega_2,\omega_3$ belong to $\Omega^2\H$. Here $\xi^{bc}=\xi^b \wedge \xi^c$ in shorthand notation. The triple of horizontal forms $\omega_1,\omega_2,\omega_3$ satisfies the additional algebraic requirements  
\begin{equation*}
\omega_1^2=\omega_2^2=\omega_3^2\neq 0 \ \mbox{and} \ \omega_i \wedge \omega_j=0 \ \mbox{for} \ 1 \leq i \neq j \leq 3.
\end{equation*}
The distribution $\H$ is thus equipped with a canonical volume form $\vol_{\H}=\frac{1}{2}\omega_1^2$ w.r.t. which we form the horizontal Hodge star operator $\sH:\Lambda^{\star}\H \to \Lambda^{\star}\H$ computed with respect to the metric $g_{\H}:=g_{\vert \H}$ on $\H$ and the volume form $\vol_{\H}$. The convention in use here is $\alpha \wedge \sH\beta=g_{\H}(\alpha,\beta)\vol_{\H}$ for $\alpha, \beta \in \Lambda^{\star}\H$. As $\H$ has rank $4$ we can further split $\Lambda^2\H=\Lambda^{-}\H \oplus \Lambda^{+}\H$ where $ \Lambda^{\pm}\H=\ker(\sH\mp 1_{\Lambda^2\H})$. Then $\Lambda^{+}\H=\spa \{\omega_1, \omega_2, \omega_3\}$. As it is well known from conformal geometry in dimension $4$, the triple $\{\omega_a, 1 \leq a \leq 3\}$ determines a quaternion structure on $\H$ via 
$\omega_a=\omega_b(I_c \cdot, \cdot)$
with cyclic permutation on $abc$. This guarantees the algebraic quaternion relations 
$I_a \circ I_b=-I_b \circ I_a=I_c$
on $\H$ and allows recovering the metric according to 
\begin{equation} \label{met}
-\omega_a=g_{\H}(I_a \cdot, \cdot)
\end{equation}
with $1 \leq a \leq 3$. Equivalently  $g_{\H}$ is determined from 
$$ (U_1 \lrcorner \omega_1) \wedge (U_2 \lrcorner \omega_2) \wedge \omega_3=-g_{\H}(U_1,U_2)\vol_{\H}
$$ with $U_1,U_2\in TM$. To ensure validity for the structure equations \eqref{str-xi} the Ricci curvature of $g$ reads 
$$ \Ric^g=6g.$$
The Ricci curvature of the compact, Einstein ASD-orbifold $(N:=M \slash \mathcal{F},g_N)$ is then normalised to $\Ric^{g_N}=12g_N$. This follows by O'Neill's formulas for the curvature of Riemannian foliations and can equivalently be phrased in terms of the transversal geometry of $M$.
\subsection{The second Einstein metric} \label{2ndE}
Splitting 
$g=g_{\V}+g_{\H}$
according to $TM=\V \oplus \H$ enables considering the canonical variation 
\begin{equation*} 
g_s:=s^2g_{\V}+g_{\H}, s >0
\end{equation*} 
of the $3$-Sasaki metric; explicitly $g_{\V}=\sum_a\xi^a \otimes \xi_a$. In subsequent computations we will systematically use the scaled vertical vector fields  $Z_a := \frac{1}{s}\xi_a$ together with the dual forms $Z^a=g_s(Z^a, \cdot)$ which satisfy $Z^a=s\xi^a$ where $a=1,2,3$. The Hodge star operator of $g_s$ is again defined according to the convention 
$\alpha \wedge \star_s \beta=g_s(\alpha,\beta)\vol_s$ for $\alpha, \beta \in \Lambda^{\star}M$. The volume form 
$\vol_s=Z^{123} \wedge \vol_{\H}$. As $g_1=g$ we simply write $\star_1=\star$ and $\vol_1=\vol$ in what follows. With these conventions 
we have the following set of purely algebraic identities, to be used extensively in subsequent computations.
\begin{lema} \label{HoL}
Pick $\alpha \in \Lambda^{\star}\H$. We have 
\begin{equation*} 
\begin{split}
&\star_s \alpha=(-1)^{\deg(\alpha)}Z^{123} \wedge \star_{\H} \alpha\\
&\star_s(Z^{a} \wedge \alpha)=Z^{bc} \wedge \star_{\H}\alpha\\
&\star_s(Z^{ab} \wedge \alpha)=(-1)^{\deg(\alpha)}Z^c \wedge \star_{\H}\alpha\\
& \star_s(Z^{123} \wedge \alpha)=\sH\alpha
\end{split}
\end{equation*}
with cyclic permutations on $abc$.
\end{lema}
The canonical variation $g_s$ 
of the $3$-Sasaki metric $g$ has the remarkable property to admit a $\G_2$ structure with torsion \cite{GS,FKMS} given by 
\begin{equation*}
\begin{split}
\varphi_s=&Z^{123} +Z^1 \wedge \omega_1+Z^2 \wedge \omega_2+Z^3 \wedge \omega_3\\
\star_s \varphi_s=&\vol_{\H}+Z^{12} \wedge \omega_3+Z^{23} \wedge \omega_1+Z^{31} \wedge \omega_2.
\end{split}
\end{equation*}
The last equation follows from Lemma \ref{HoL}. To spell out the volume convention for $\G_2$ structures in use here, record that $(U_1 \lrcorner \varphi_s) \wedge (U_2 \lrcorner \varphi_s) \wedge \varphi_s=6g_s(U_1,U_2)\vol_s$ 
with $U_1,U_2 \in TM$ as it can be checked by a direct computation, crucially relying on \eqref{met}. This convention agrees with that 
in \cite{Bryant1} but is opposite to the one in \cite{NS}. 
 
Additional background facts we shall need are as follows. The action of $\G_2$,viewed as the stabiliser of the $3$-form $\varphi_s$, allows splitting 
\begin{equation*}
\Lambda^4M=\Lambda^4_{27}M \oplus \Lambda^4_7M \oplus \Lambda^4_1M, \ \ 
\Lambda^3M=\Lambda^3_{27}M \oplus \Lambda^3_7M \oplus \Lambda^3_1M, \ \ \Lambda^2M=\Lambda^2_{14}M \oplus \Lambda^2_7M
\end{equation*}
into irreducible representations, where the subscript indicates dimension of the factor. As this is purely algebraic 
we systematically use the notation $\Lambda^4_{27}M=\Lambda^4_{27}(\varphi_s), \Lambda^3_{27}M=\Lambda^3_{27}(\varphi_s)$ to emphasize dependence on the $\G_2$ structure. In addition we have a canonical isomorphism $\bfi: \Sym^2_0(M,g_s) \rightarrow \Lambda^3_{27}(\varphi_s)$ which acts on decomposable tensors as the restriction of   
the mapping $a \otimes a \mapsto a \wedge (a \lrcorner \varphi_s)$ for $a \in \mathrm T M$. This isomorphism differs by a factor of $\tfrac12$ from the definition given in \cite{Bryant1}, to which we refer the reader for further information.

To explain the torsion type of the $\G_2$-structure $\varphi_s$ we record a few consequences of the structure equations. Firstly, the frame $Z^a$ satisfies 
\begin{equation} \label{dza}
\begin{split}
&\di\! Z^a=2s \,\omega_a \,-\, \tfrac{2}{s} \ \, Z^{bc}\\
&\di Z^{ab}=2s(\omega_a \wedge Z^b-\omega_b \wedge Z^a)\\
&  \di Z^{123}=2s \mathfrak{S}_{abc}Z^{ab} \wedge \omega_c
\end{split}
\end{equation}
where $\mathfrak{S}_{abc}$ indicates the cyclic sum on $abc$. Secondly, differentiating in \eqref{str-xi} yields 
\begin{equation} \label{str-o}
\di\!\omega_a=2(\omega_b \wedge \xi^c-\omega_c \wedge \xi^b)=\tfrac{2}{s} \,(\omega_b \wedge Z^c \,-\, \omega_c \wedge Z^b).
\end{equation}
These equations reveal that the choice $s = 1/\sqrt{5}$ plays a distinguished r\^ole; in particular this value of $s$ picks up the second Einstein metric 
in the canonical variation of the Einstein metric $g$ as the following shows.
\begin{teo}\cite{FKMS,GS} \label{recap}
The form $\varphi_s$ defines a nearly $\G_2$ structure if and only if $s = 1/\sqrt{5}$.
\end{teo}
With $s =\s5$ we explicitly have 
$
\di\!\varphi_s=\tfrac{12}{\sqrt{5}} \,\star_s \varphi_s
$. 
As mentioned in the introduction the nearly $\G_2$ structure $\varphi_{\s5}$ has the remarkable 
property to be proper, equivalently
the Einstein metric $g_{\s5}$ does not admit a compatible Sasaki structure. See \cite{FKMS} for more details.
To end this section we derive further properties of the horizontal Hodge star operator.
Direct computation based on \eqref{met} leads to 
\begin{equation} \label{Ho3}
\begin{split}
&\sH \alpha=I_a \alpha \wedge \omega_a, \ \ \star_{\H}(\alpha \wedge \omega_a)=I_a\alpha
\end{split}
\end{equation}
for $1 \leq a \leq 3$ and $\alpha \in \Lambda^1\H$. Here the endomorphisms $I_a$ act on $1$-forms $\alpha \in \Lambda^1\H$ by composition, $I_a\alpha:=\alpha \circ I_a$.
In particular \eqref{Ho3} entails the comparaison formulas 
\begin{equation} \label{Ho4}
I_1 \alpha \wedge \omega_1=I_2 \alpha \wedge \omega_2=I_3 \alpha \wedge \omega_3
\end{equation}
as well as 
\begin{equation} \label{Ho44}
I_a\alpha \wedge \omega_b=-I_b\alpha \wedge \omega_a=\alpha \wedge \omega_c
\end{equation}
with $\alpha \in \Lambda^1\H$ and cyclic permutations on $abc$. These will be frequently used in the following sections.
\subsection{The Lichnerowicz Laplacian} \label{LL}
We review a few facts about the spectrum of the Lichnerowicz Laplacian $\Delta_L^{g_s}$ acting on the space $\TT(g_s)$ of TT-tensors. For the precise definition of this operator, which is not needed at this stage, see \cite{Besse} or section \ref{CF} of the paper. We let 
$s=1/\sqrt{5}$ in what follows and recall how the $\G_2$ structure $\varphi_s$ can be used to identify $\Delta_L^{g_s}$ with an operator acting 
on $\Omega^3M$. According to \cite{AlS}
\begin{equation} \label{TT-int}
\begin{split}
\bfi (\TT(g_s))=&\{\gamma \in \Omega^3_{27}(\varphi_s) : (\di\!\gamma)_{\Lambda^4_7}=0\}\\
=&\{\gamma \in \Omega^3_{27}(\varphi_s) :(\di\!^{\star}\gamma)_{\Lambda^2_7}=0\}= \{\gamma \in \Omega^3_{27}(\varphi_s) : \di\! \gamma \in \Omega^4_{27}(\varphi_s)\}
\end{split}
\end{equation}
where the last two equalities follow essentially by type considerations w.r.t. the $\G_2$ invariant splitting of $\Lambda^{\star}M$.

On the space 
$\{\gamma \in \Omega^3_{27}(\varphi_s) :(\di\!\gamma)_{\Lambda^4_7}=0 \}$ the comparaison formula relating $\Delta_L^{g_s}$ to the form Laplacian $\Delta^{g_s}:\Omega^3M \to \Omega^3M$ from \cite[Prop. 6.1]{AlS} reads 
\begin{equation}\label{DeltaL}
\bfi \circ \, \Delta_L^{g_s} \, \circ \bfi^{-1} \;=\; \Delta^{g_s}  \;+\; 6s \star_s \di  \;+\; 36 s^2.
\end{equation}
As the operator on the r.h.s. of \eqref{DeltaL} can be rewritten as $(\star_s\di+3s)^2+\di\di^{\star_s}+27s^2$ we obtain the estimate 
$$\Delta_L^{g_s} \geq 27s^2$$ 
on $\TT(g_s)$. In our setup this recovers, with a simple proof, the lower bound for the first Lichnerowicz eigenvalue for metrics with Killing spinors in \cite{GiPope} used as a criterion for generalised black hole stability in the Freund-Rubin compactification. 

Throughout this paper we are interested in eigenvalues $\tau$ for $\Delta_L^{g_s}:\TT(g_s) \to \TT(g_s)$ with $\tau \le 2E_s$, where we recall that the Einstein constant of the metric $g_s$ is explicitely given by 
$E_s=54s^2$. The eigenspace for $\tau = 2E_s$ is precisely the space of infinitesimal Einstein deformations of $g_s$, which contains 
infini\-tesimal $\G_2$ deformations as a subspace. The latter 
correspond to $E_{-12s}$ where the notation 
$$E_{\lambda}:=\ker(\star_s\di-\lambda) \cap  \Omega^3_{27}(\varphi_s)$$ 
for $\lambda \in \bbR$ will be used in the rest of the paper. Eigenvalues $\tau<2E_s$ will be called {\it{unstable}} and the corresponding eigentensors form the space of destabilising directions \cite{Koiso1980}. Arguments entirely similar to those used in the proof of Theorem 6.2 in \cite{AlS} show that 
\begin{pro} \label{lich}
The eigenspace $\ker(\Delta_L^{g_s}-\tau)$ of the Lichnerowicz Laplacian $\Delta_L^{g_s}$
acting on $\TT(g_s)$
is isomorphic to the direct sum 
$$ E_{\lambda^{+}} \oplus E_{\lambda^{-}} \oplus \{ \gamma \in \Omega^3_{27}(\varphi_s) : \di\!\di^{\star_s}\gamma=\mu \gamma\}
$$
where $\lambda^\pm = -3s \pm \sqrt{\tau - 27s^2}$ and $\mu=\tau - 36s^2 \neq 0$. In case $\tau\leq 2E_s=108s^2$ we must have 
$$\ \lambda^+ (\lambda^+ + 2s) \le \tfrac{48}{5}, \ \lambda^- (\lambda^- + 2s) \le 24, \, \ 0 \neq \mu \leq \tfrac{72}{5}.
$$ 
\end{pro}
\begin{proof}
We split the finite dimensional space $\bfi(\ker ( \Delta_L^{g_s} - \tau))$ into eigenspaces for the operator $\star_s \di$. To outline how this process works, record that $\star_s \di : \Omega^3M \to \Omega^3M$ is self-adjoint, commutes with the operator on the r.h.s. of \eqref{DeltaL} and at the same time preserves the condition $(\di\!\gamma)_{\Lambda^4_{7}}= 0$. 
Hence, for the eigenspace  $\ker (\star_s \di - \lambda) $ we either have $\lambda=0$, or $\lambda$ is determined from the quadratic equation 
$ \lambda^2 + 6s\lambda +36s^2-\tau= 0$
with solutions 
$\lambda^\pm = -3s \pm \sqrt{\tau - 27s^2}$. The square root is well defined due to the lower bound for $\Delta_L^{g_s}$ given above. For $\lambda=0$ it follows  that $\gamma \in \ker (\di\!\di^{\star_s} - \mu)$ with $\mu=\tau - 36s^2$. The instance $\mu=0$ cannot occur 
since it forces $\di\!\di^{\star_s}\gamma=0$; as $\star_s\di\gamma=0$ by hypothesis it follows that $\gamma$ is harmonic. Because the de Rham 
cohomology $H^3_{dR}M=0$ for $3$-Sasaki manifolds (see \cite{GS}) it follows that $\gamma=0$. Thus, assuming $\tau\leq 2E_s$ forces $\mu \le 72s^2$ as well as $\lambda^+ \le 6s$ and $|\lambda^-| \le 12 s$. 
A simple calculation then shows $\lambda^- (\lambda^- + 2s) \le 24$ and  $\lambda^+ (\lambda^+ + 2s) \le \tfrac{48}{5}$.
\end{proof}
In all eigenvalue estimates from Proposition \ref{lich} equality corresponds precisely to having $\tau = 2E_s$, i.e. to infinitesimal Einstein deformations. 
\section{$\G_2$ and $\su(2)$-representation spaces} \label{a-id}
\subsection{$\G_2$-modules} \label{G2}
We determine, for arbitrary values of $s>0$, the algebraic structure of the $\G_2$-module $\Lambda^3_{27}(\varphi_s)\subseteq \Lambda^3M$ w.r.t. the splitting $TM=\V \oplus \H$. As the latter ensures that
\begin{equation} \label{type3d} 
\Lambda^3M=\Lambda^3\V \oplus (\Lambda^2\V \wedge \Lambda^1\H) \oplus (\Lambda^1\V \wedge \Lambda^2\H) \oplus \Lambda^3\H
\end{equation}
we obtain an isomorphism $\iota_s: V^3\H \to \Lambda^3M$ given by 
\begin{equation*}
\iota_s\left( \begin{array}{c}
F\\
\alpha \\
\sigma \\
\beta \end{array} \right):=FZ^{123}+\mathfrak{S}_{abc}Z^{ab} \wedge \alpha_c+\sum_a Z^a \wedge \sigma_a+\beta
\end{equation*}
where $ V^3\H:=\Lambda^0\H \oplus \Lambda^1(\H,\bbR^3) \oplus \Lambda^2(\H,\bbR^3) \oplus \Lambda^3\H.$

The map $\iota_s$ is an isometry 
when $\Lambda^3M$ is equipped with the metric induced by $g_s$ and the bundle $V^3\H$
is equipped with the direct product metric induced by $g_{\H}$. Unless otherwise indicated sections of the latter bundle will be systematically viewed as column vectors, in order to enable multiplication by matrix valued differential operators. Relating the isomorphism $\iota_s$ to $\Lambda^3_{27}(\varphi_s)$ turns out to hinge on the purely algebraic contraction maps 
\begin{equation*}
\begin{split}
& \tT:\Lambda^{\star}(\H,\mathbb{R}^3) \to \Lambda^{\star}\H, \ \ \tT(\sigma):=\sH \sum_a \sigma_a \wedge \omega_a \\
& L_{\omega}:\Lambda^{\star}(\H,\bbR^3) \to \Lambda^{\star+2}(\H,\bbR^3), \ \ (L_{\omega}\sigma)_a:=\sigma_b \wedge \omega_c-\sigma_c \wedge \omega_b
\end{split}
\end{equation*}
with cyclic permutations on the indices $abc$. Indeed
\begin{lema} \label{dec-3sas}
The map $\kappa_s:\Lambda^1(\H,\bbR^3) \oplus \Lambda^2_{sym}(\H,\bbR^3) \to \Lambda^3_{27}(\varphi_s)$ given by 
\begin{equation*} 
\kappa_s(\alpha,\sigma):=\iota_s(-\tT(\sigma),\alpha,\sigma,\sH \tT(\alpha))
\end{equation*} 
where 
$\Lambda^2_{sym}(\H,\bbR^3):=\ker(L_{\omega}:\Lambda^2(\H,\bbR^3)\to \Lambda^4(\H,\bbR^3))$
is a bundle isomorphism.
\end{lema}
\begin{proof} 
Pick $\gamma=\iota_s(F,\alpha,\sigma,\beta)^T \in \Lambda^3M$ where $(F,\alpha,\sigma,\beta) \in V^3\H$. Direct algebraic computation, only using the vanishing of $\Lambda^q\H=0$ for 
$q \geq 5$, that of $\Lambda^4\V$  as well as the identity $\sH^2=(-1)^p$ on $\Lambda^p\H$ shows that 
\begin{equation*}
\begin{split}
\gamma \wedge \varphi_s&=Z^{123}\wedge (\sH\tT(\alpha)-\beta)+\mathfrak{S}_{abc}Z^{ab}\wedge (L_{\omega}\sigma)_c\\
\gamma \wedge \star_s \varphi_s&=(F+\tT(\sigma))\vol_s.
\end{split}
\end{equation*}
Recalling that $\Lambda_{27}^3(\varphi_s)=\{\gamma \in \Lambda^3M : \gamma \wedge \varphi_s=0, \gamma \wedge \star_s\varphi_s=0\}$ the claim follows by projection onto the component factors of \eqref{type3d}.
\end{proof}
The splitting of $\Lambda^3_{27}(\varphi_s)$ provided by the isomorphism above can be further refined by taking into account the following observations. As $L_{\omega}$ vanishes on $\Lambda^{-}(\H,\bbR^3)$ we have 
$$\Lambda^2_{sym}(\H,\bbR^3)=\Lambda^{+}_{sym}(\H,\bbR^3) \oplus \Lambda^{-}(\H,\bbR^3)$$
where $\Lambda^{+}_{sym}(\H,\bbR^3):=\Lambda^2_{sym}(\H,\bbR^3) \cap \Lambda^{+}(\H,\bbR^3)$.  Consider the element 
$\omega:=(\omega_1,\omega_2,\omega_3)^T$
in $\Lambda^{+}_{sym}(\H,\bbR^3)$. Since 
the map 
$\Lambda^0(\H,\Sym^2(\mathbb{R}^3)) \to \Lambda^{+}_{sym}(\H, \mathbb{R}^3)$ given by matrix multiplication, $a\mapsto a \omega$, is a bundle isomorphism we can split 
\begin{equation*}\label{G22}
\Lambda^{+}_{sym}(\H,\bbR^3)=\ker \tT\oplus \,\bbR\omega
\end{equation*}
according to $\Sym^2\mathbb{R}^3=\Sym^2_0\mathbb{R}^3\oplus \mathbb{R}$. Consequently we obtain a distinguished line in $\Lambda^3_{27}(\varphi_s)$ spanned by 
\begin{equation*}
\widetilde{\varphi}_s:=\kappa_s(0,\omega)=\varphi_s-7Z^{123}
\end{equation*}
where the last equality follows from $\tT(\omega)=6$. As already mentioned in the introduction this plays a significant r\^ole when looking at unstable eigenvalues.
\begin{rema} \label{dual-rmk}
Having the forms $\omega_a$ self-dual makes that 
\begin{equation*}
\begin{split}
& (L_{\omega}^{\star}\sigma)_a=g_{\H}(\omega_b,\sigma_c)-g_{\H}(\omega_c,\sigma_b)
\end{split}
\end{equation*}
whenever $\sigma \in \Lambda^2(\H,\bbR^3)$. In particular 
$\Lambda^2_{sym}(\H,\bbR^3)=\ker(L_{\omega}^{\star}:\Lambda^2(\H,\bbR^3)\to \Lambda^0(\H,\bbR^3))$.
\end{rema}
We conclude by describing alternative algebraic expressions for the operator $\tT$ acting on $\Lambda^1(\H,\bbR^3)$.
Indeed \eqref{Ho3} makes that 
\begin{equation} \label{tT-alt}
\tT(\alpha)=\sum_a I_a\alpha_a 
\end{equation}
 when $\alpha \in \Lambda^1(\H,\bbR^3)$. Equivalently, 
\begin{equation*}
\tT=-\bbI^{\star} \ \mbox{on} \ \Lambda^1(\H,\bbR^3)
\end{equation*}
where the operator 
\begin{equation} \label{bbId}
\bbI:\Lambda^1\H \to 
\Lambda^1(\H,\bbR^3)
\ \mathrm{is \ defined \ according \ to} \  (\mathbb{I}\alpha)_a:=I_a\alpha.
 \end{equation}

\subsection{Geometry of the $\su(2)$-action} \label{sl2C}
Consider the representation of $\su(2)$ on $\Omega^{\star}M$ given by $A_a \mapsto \L_{\xi_a}$
for the basis choice 
$$A_1=\left( \begin{array}{ccc} 0& 0&0\\
0&0&-2\\
0&2&0
\end{array} \right), \ A_2=\left( \begin{array}{ccc} 0& 0&2\\
0&0&0\\
-2&0&0
\end{array} \right), \ A_3=\left( \begin{array}{ccc} 0& -2&0\\
2&0&0\\
0&0&0
\end{array} \right)
$$
in $\su(2)$. Since $\xi_a$ are Killing vector fields preserving $\H$ we have $\L_{\xi_a}^{\star}=-\L_{\xi_a}$ on $\Omega^{\star}M$ respectively $\Omega^{\star}\H$. Therefore the $\su(2)$-representation on $\Omega^{\star}M$ is orthogonal w.r.t the $L^2$-inner product induced by $g_s$ and preserves $\Omega^{\star}\H$ as well as the $\G_2$-invariant spaces 
$\Omega^3_{27}(\varphi_s)$ due to 
$$\spa \{\xi_1,\xi_2,\xi_3\} \subseteq \mathfrak{aut}(M,\varphi_s).$$ 
The last inclusion is a direct consequence of the structure equations \eqref{str-xi} and \eqref{str-o}. We indicate with 
\begin{equation} \label{ro}
\rho:\su(2) \times \Omega^{\star}\H \to \Omega^{\star}\H
\end{equation} the induced representation and let $\pi^1$ be the representation of $\su(2)$ on $\bbR^3$ by matrix multiplication. The Casimir operator of $\rho$ (or vertical Laplacian) thus reads 
$$\C:=-\sum_a \L^2_{\xi_a} :\Omega^{\star}\H \to \Omega^{\star}\H.$$ 
This differs by a factor of $\tfrac{1}{8}$ from the usual Lie theoretic definition involving the Killing form of $\su(2)$. The operator $\C$ is self-adjoint, non-negative and $\su(2)$-invariant. 

From the structure equations of the frame $\omega_a, a=1,2,3$ in \eqref{str-o} together with Cartan's formula we obtain
\begin{equation} \label{Lie-os}
\L_{\xi_a}\omega_b=-\L_{\xi_b}\omega_a=2\omega_c
\end{equation}  
which clearly entail
\begin{equation} \label{Lie-osI}
\L_{\xi_a}I_b=-\L_{\xi_b}I_a=2I_c
\end{equation}
on $\Omega^1\H$.
Direct computation based on these facts shows that the action of $\su(2)$ on $\Omega^3M$ by Lie derivatives breaks down via the isomorphism 
$\iota_s$ into 
\begin{itemize}
\item the direct sum representation $\rho \oplus \rho$ on $\Omega^1\H \oplus \Omega^3\H$
\item the tensor product representation $\rho \otimes \pi^1$ on $\Omega^{1}(\H,\bbR^3)$ respectively $\Omega^{2}(\H,\bbR^3)$.
\end{itemize}
The representation $\rho \otimes \pi^1$ acts according to $A_a\mapsto \L_{\xi_a}+A_a$ where the Lie derivative $\L_{\xi_a}$ is extended to act on each component of  
elements in $\Omega^{\star}(\H,\bbR^3)$. To determine the main invariants of the tensor product representation $\rho \otimes \pi^1$ we let 
$$\L_{\xi}:\Omega^{\star}\H \to \Omega^{\star}(\H,\bbR^3), \ \ \L_{\xi}:=(\L_{\xi_1}, \L_{\xi_2}, \L_{\xi_3})^T.$$ 
Its formal adjoint reads  
$ \L^{\star}_{\xi}\sigma=-\sum_a \L_{\xi_a}\sigma_a$
for $\sigma \in \Omega^{\star}(\H,\bbR^3)$. In addition consider 
\begin{equation*}
\begin{split}
&C:=\left( \begin{array}{ccc}
0 & -\L_{\xi_3} & \L_{\xi_2}\\
\L_{\xi_3}  & 0  & -\L_{\xi_1} \\
-\L_{\xi_2}  & \L_{\xi_1}  & 0
\end{array} \right  ):\Omega^{\star}(\H, \mathbb{R}^3) \rightarrow \Omega^{\star}(\H,\mathbb{R}^3)\\
& \p:=\tT \circ \L_{\xi}:\Omega^1\H \to \Omega^1\H.
\end{split}
\end{equation*}
An equivalent way of computing $\p$, derived from \eqref{tT-alt}, is according to $\p=\sum_a I_a \circ \mathscr{L}_{\xi_a}$. 

The operators $\tT,\L_{\xi},\p$ and $C$ feature in the block structure, w.r.t. to the splitting \eqref{type3d}, 
of various differential operators of interest in this paper, as we will see in the next section. Therefore it is useful to record here those of their properties which follow directly from basic representation theory.
\begin{lema} \label{inv-tL}
The operators
$\tT:\Omega^1(\H,\bbR^3) \to \Omega^1\H$ and $\L_{\xi}:\Omega^1\H \to \Omega^1(\H,\bbR^3)$, 
as well as $\p:\Omega^1\H \to \Omega^1\H$, 
are $\su(2)$ invariant.
\end{lema}
\begin{proof}
Letting $\alpha \in \Omega^1(\H,\bbR^3)$ we get 
\begin{equation*}
\L_{\xi_1}\tT(\alpha)=\L_{\xi_1}\sum_a I_a\alpha_a=I_a( \sum_a \L_{\xi_1}\alpha_a)+2(I_3\alpha_2-I_2\alpha_3)=\tT(\L_{\xi_1}\alpha)+
2(I_3\alpha_2-I_2\alpha_3)
\end{equation*}
after using \eqref{Lie-osI}. At the same $\tT(A_1\alpha)=2(-I_2\alpha_3+I_3\alpha_2)$ and invariance for $\tT$ is proved. Similarly, with 
$\alpha \in \Omega^1\H$
\begin{equation*}
\begin{split}
\L_{\xi_1}\L_{\xi}\alpha&=(\L_{\xi_1}^2\alpha, \L_{\xi_1}\L_{\xi_2}\alpha,\L_{\xi_1}\L_{\xi_3}\alpha)^T=\L_{\xi}(\L_{\xi_1}\alpha)+
2(0,\L_{\xi_3}\alpha,-\L_{\xi_2}\alpha)^T\\
&=\L_{\xi}(\L_{\xi_1}\alpha)-A_1\L_{\xi}\alpha
\end{split}
\end{equation*}
from the $\su(2)$-bracket relations in \eqref{su2}. This proves invariance for $\L_{\xi}$ and thus also for $\p=\tT \circ \L_{\xi}$.
\end{proof}
Additionally, we consider the Casimir operator of the tensor product representation of 
$\su(2)$ on $\Omega^{\star}(\H,\bbR^3)$, which is defined according to $\C_{\rho \otimes \pi^1}=-\sum_a (\L_{\xi_a}+A_a)^2$. A short matrix multiplication calculation shows that 
$\sum_a A_a\L_{\xi_a}=\sum_a \L_{\xi_a}A_a=2C$ and $\sum_aA_a^2=-8$ on $\Omega^{\star}(\H,\bbR^3)$, hence after expanding the squares in the definition of $\C_{\rho \otimes \pi^1}$ we obtain 
\begin{equation} \label{cas01}
\C_{\rho \otimes \pi^1}=\C-4C+8.
\end{equation}
In particular $C:\Omega^{\star}(\H,\bbR^3) \to \Omega^{\star}(\H,\bbR^3)$ is $\su(2)$-invariant and self-adjoint, $C^{\star}=C$. Below we also compute its characteristic polynomial. 
\begin{lema} \label{sq} 
The following hold on  $\Omega^{\star}(\H,\bbR^3)$
\begin{eqnarray}
&&C^2=2C+\C-\L_{\xi}\L_{\xi}^{\star} \label{sqC}\\
&&\L^{\star}_{\xi} \circ (C-2)=0. \label{sqC00}
\end{eqnarray}
\end{lema}
\begin{proof}
Recall that $\C=-\sum_a \L_{\xi_a}^2$. After multiplying the operator valued matrices which occur in the definitions of $C,\L_{\xi}$ and $\L_{\xi}^{\star}$ it follows that $C^2$ respectively $\L_{\xi}\L_{\xi}^{\star}$ are given by 
\begin{equation*}
C^2_{ij}=\delta_{ij}\C+\L_{\xi_j}\L_{\xi_i} \ \mathrm{respectively} \ -(\L_{\xi}\L_{\xi}^{\star})_{ij}=\L_{\xi_i}\L_{\xi_j}.
\end{equation*}
The claim in \eqref{sqC} follows now from the $\su(2)$-bracket relations in \eqref{su2}. Similarly, an elementary computation only based on \eqref{su2} shows that $(C-2)\circ \L_{\xi}=0$. As $C$ is self-adjoint the claim in \eqref{sqC00} follows from the latter relation by duality.
%
\end{proof}
\section{Operator block structure} 
\label{specD}
The primary aim is to determine the block structure of the operators $\star_s\di$ and $\Delta^{g_s}$ w.r.t the splitting induced by the isomorphism $\iota_s:
\mathbb{V}^3\H \to \Omega^3M$. Here $\bbV^3\H$ denotes the space of sections of the vector bundle $V^3\H$, explicitely 
$$\mathbb{V}^3\H=\Omega^0\H \oplus \Omega^1(\H,\bbR^3) \oplus \Omega^2(\H,\bbR^3) \oplus \Omega^3\H.$$
This is one the main technical steps in this paper, needed to determined the structure of various eigenspaces of Laplace type operators. Throughout this section the parameter $s$ will be arbitrary.  
\subsection{Horizontal operators} \label{HoO}
The first of these operators is the horizontal exterior derivative 
$ \dH:\Omega^{\star}\H \to \Omega^{\star+1}\H, \ \alpha \mapsto (\di\alpha)_{\H}$ 
where the subscript indicates projection onto 
$\Omega^{\star}\H$ w.r.t. the splitting $\Omega^{\star}M=\oplus_{i=0}^3 \Omega^{3-i}\V \wedge \Omega^{\star+i-3}\H$. Cartan's formula shows that $\dH$ is related to the ordinary exterior differential via 
\begin{equation} \label{Cart}
\di=\dH+\sum_{a} \xi^a \wedge \L_{\xi_a}.
\end{equation}
Note that the operators $\L_{\xi_a}$ preserve $\Omega^{\star}\H$ as $\V$ is totally geodesic. Further properties of the horizontal exterior derivative include its $\su(2)$-invariance 
\begin{equation} \label{com1}
[\dH,\L_{\xi_a}]=0.
\end{equation}
This is a consequence of \eqref{Cart} and is checked by using that $[\di,\L_{\xi_a}]=0$ together with having
$\L_{\xi_b}\xi^a \in \Omega^1\V$ as granted by the structure equations of the frame ${\xi^1,\xi^2,\xi^3}$. Secondly, with the aid of \eqref{Cart} and $(\di \xi^a)_{\H}=2\omega_a$ we see that the projection of the identity $\di^2=0$ onto $\Omega^{\star}\H$ reads 
\begin{equation} \label{sqH}
\di_{\H}^2+2\sum_a \omega_a \wedge \L_{\xi_a}=0.
\end{equation}
In particular the $\su(2)$-invariant operator $\p$ acting on $\Omega^1\H$ can be recovered from 
\begin{equation} \label{def-p}
\sH \di^2_{\H}=-2\!\p.
\end{equation}

The formal adjoint $\di_{\H}^{\star}:\Omega^{\star}\H \to \Omega^{\star-1}\H$ of $\dH$, computed w.r.t. the metric induced by $g_{\H}$, 
is also $\su(2)$-invariant i.e.
$[\di_{\H}^{\star},\L_{\xi_a}]=0$.
It allows building the horizontal Laplacian 
$$\Delta_{\H}:=\dH\di_{\H}^{\star}+\dH \di_{\H}^{\star} : \Omega^{\star}\H \to \Omega^{\star}\H$$
which together with the Casimir operator of the representation $\rho$ 
enters the following set of comparaison formulas involving the codifferential $\di^{\star_s}$ respectively the Laplacian 
$\Delta^{g_s}$ of the canonical variation $g_s,s>0$. 
\begin{lema} \label{comp-cd}
We have 
\begin{itemize}
\item[(i)] $\di_{\H}^{\star}=-\star_{\H} \dH \star_{\H}$ on $\Omega^{\star}\H$ as well as 
 $\di^{\star_s}=\di_{\H}^{\star}$ on $\Omega^0\H \oplus \Omega^1\H$
\item[(ii)] the horizontal component of $\Delta^{g_s}\alpha$ with $\alpha \in \Omega^1 \H$ satisfies
\begin{equation*}
(\Delta^{g_s}\alpha)_{\H}=(\Delta_{\H}+\frac{1}{s^2}\C)\alpha.
\end{equation*}
\end{itemize}
\end{lema}
\begin{proof}
The claims in (i) are proved at the same time. Since $M$ has dimension $7$ we have $\di^{\star_s}=(-1)^p \star_s \di \star_s$ on $\Omega^pM$. Pick $\alpha \in \Omega^p\H$; using successively Lemma \ref{HoL} and \eqref{Cart} we obtain 
\begin{equation*}
\begin{split}
(-1)^p\di\star_s \alpha=&\di(Z^{123} \wedge \sH\alpha)=\di\!Z^{123} \wedge \sH \alpha-Z^{123} \wedge \dH \sH \alpha.
\end{split}
\end{equation*} 
As $\di\!Z^{123}=2s\mathfrak{S}_{abc}Z^{ab}\wedge \omega_c$ we find 
\begin{equation*}
\di^{\star_s}\alpha=(-1)^p\star_s\di(\star_s \alpha)=-\sH\dH\sH\alpha+2s(-1)^p\sum_a Z^a \wedge \sH(\omega_a \wedge \sH \alpha)
\end{equation*}
by taking once again into account the structure of $\star_s$ in Lemma \ref{HoL}. In particular the projection of $\di^{\star_s}$ onto $\Omega^{\star}\H$ equals $-\sH\dH\sH$ thus $\di_{\H}^{\star}=-\star_{\H} \dH \star_{\H}$ by $L^2$-orthogonality. To finish the proof it is enough to notice that $\omega_a \wedge \sH \alpha=0$ when $\alpha \in \Omega^0\H\oplus\Omega^1\H$.\\
(ii) follows by an $L^2$-orthogonality argument. First, we compute with the aid of \eqref{Cart} the $L^2$-product 
\begin{equation*}
\begin{split}
(\di\! \alpha, \di\! \beta)_s=&(\dH\alpha+\sum_a Z^a \wedge \L_{Z_a}\alpha, \dH\beta+\sum_b Z^b \wedge \L_{Z_b}\beta)_s\\
=&(\dH\alpha,\dH \beta)+\sum_a (\L_{Z_a}\alpha, \L_{Z_a}\beta)=((\di_{\H}^{\star}\!\dH+\frac{1}{s^2}\C)\alpha, \beta).
\end{split}
\end{equation*}
Here the round bracket denotes the $L^2$-product w.r.t. $g_s$ respectively $g$. The claim follows from having $\di^{\star_s}=\di_{\H}^{\star}$ on $\Omega^1\H$, as granted by (i).
\end{proof}
An entirely similar argument also shows that the scalar Laplacian 
$$\Delta^{g_s}=\Delta_{\H}+\frac{1}{s^2}\C \ \mathrm{on} \ C^{\infty}M.$$
Yet another action which is relevant for our purposes is 
\begin{equation} \label{spp1}
\sp(1) \times \Omega^1\H \to \Omega^1\H, \ i_a \mapsto I_a
\end{equation} 
where $\sp(1)$ is generated by $\{i_1,i_2,i_3\}$ with Lie bracket determined from $[i_a,i_b]=-2i_c$.
To finish this section we identify, for later use, the piece in the horizontal Laplacian $\Delta_{\H}$ which is $\sp(1)$-invariant, that is invariant under the complex structures $\{I_1,I_2,I_3\}$. 
The most computationally 
efficient way towards this end is to use the Riemannian cone $(CM:=M \times \bbR_{+},g_c:=r^2g+(\di\!r)^2)$ of $M$. This is hyperk\"ahler w.r.t. 
the triple of complex structures determined from 
\begin{equation*}
\begin{split}
&J_a\partial_r=-r^{-1}\xi_a, \ J_a\xi_b=\xi_c, \ J_a=I_a \ \mbox{on} \ \H 
\end{split}
\end{equation*}
with cyclic permutations on $abc$. The corresponding symplectic forms are $\omega_{J_a}=-\frac{1}{2}\di(r^2\xi^a)$ and satisfy $g_c^{-1}\omega_{J_a}=J_a$. In fact an equivalent definition of a $3$-Sasaki metric is to require its metric cone be hyperk\"ahler.
\begin{lema} \label{comI} 
We have $[\Delta_{\H}+\C,I_a]=0$ on $\Omega^1\H$.
\end{lema}
\begin{proof}
Indicating with $\Delta^c$ the Laplacian of the cone metric we derive 
\begin{equation*} 
\Delta^c=r^{-2}\Delta^g+\di r^{-2} \wedge \di^{\star}
\end{equation*}
on $\Omega^{\star}M \subseteq \Omega^{\star}CM$, after a short computation. Pick $\alpha \in \Omega^1\H$, so that $J_1\alpha=I_1\alpha$. As $(g_c,J_1)$ is K\"ahler $\Delta^c J_1=J_1 \Delta^c$ hence the comparaison  formula for the Laplacians above makes that 
\begin{equation*} 
\begin{split}
r^{-2}\Delta^g(I_1 \alpha)+\di r^{-2} \wedge \di^{\star}(I_1\alpha)=&J_1(r^{-2}\Delta^g\alpha+f\di r^{-2})
=r^{-2}J_1(\Delta^g \alpha)-2r^{-2}f\xi^1
\end{split}
\end{equation*}
where $f=\di^{\star}\!\alpha$.
Projecting onto $\Omega^1\H$ we find $(\Delta^g(I_1 \alpha))_{\H}=I_1(\Delta^g \alpha)_{\H}$ and the claim follows from Lemma \ref{comp-cd},(ii).
\end{proof}
\begin{coro} \label{dhp}
We have $[\Delta_{\H},\p]=0$ on $\Omega^1\H$.
\end{coro}
\begin{proof}
As $\Delta_{\H}+\C$ is $\su(2)$ invariant and $\p=\sum_a \L_{\xi_a}I_a$ we get $[\Delta_{\H}+\C,\p]=0$ by Lemma \ref{comI}. At the same time 
$\p$ is $\su(2)$ invariant by Lemma \ref{inv-tL}, hence $[\C,\p]=0$ and the claim follows.
\end{proof}
\subsection{Block structure for $\star_s \di$} \label{dsys}
We make this explicit with the aid of the vertical operators $C,\L_{\xi},\p$ and their algebraic structure as described in Section \ref{a-id}. For notational convenience, we also consider the operator $\alpha \in \Lambda^{\star}\H \mapsto \alpha \wedge \omega \in 
\Lambda^{\star+2}(\H,\bbR^3)$ which acts according to $(\alpha \wedge \omega)_a:=\alpha \wedge \omega_a$. Thus prepared we first establish the following
\begin{lema} \label{blgen-1} The operator $\star_s \di:\Omega^3M \to \Omega^3M$ satisfies 
\begin{equation*}
\iota_s^{-1}\star_s \di \iota_s \left( \begin{array}{c}
F\\
\alpha \\
\sigma \\
\beta \end{array} \right)=\left( \begin{array}{ll}
2s\tT(\sigma)+\sH\dH\!\beta\\
-2s\PP\!\alpha-\sH\dH\!\sigma+\frac{1}{s}\sH\L_{\xi}\beta\\
2sF\omega+\sH\dH\!\alpha+\frac{1}{s}\sH(C-2)\sigma\\
-\sH(\frac{1}{s} \L_{\xi}^{\star}\alpha+\dH\!F)
\end{array} \right )
\end{equation*}
where $\PP:\Omega^1(\H,\bbR^3) \to \Omega^1(\H,\bbR^3)$ is given by $\PP=1+\bbI \circ \tT$ and $\bbI$ is defined according to \eqref{bbId}. 
\end{lema}
\begin{proof}
In the following computations we systematically take into account the structure equations of the frame $Z^a$ and their 
direct consequences, as listed in \eqref{dza}. A short computation based on the expansion of the exterior derivative $\di$ according to \eqref{Cart} and on 
the structure of the Hodge star operator $\star_s$ as described in Lemma \ref{HoL} thus leads to 
\begin{equation*}
\begin{split}
\star_s \di(FZ^{123})=&2sF \sum_a Z^a \wedge \omega_a-\star_{\H}\dH\!F\\
\star_s \di (\mathfrak{S}_{abc}Z^{ab} \wedge \alpha_c)=&2s \mathfrak{S}_{abc}Z^{ab} \wedge \sH(L_{\omega}\alpha)_c+
\sum_{a}Z^{a} \wedge \sH \dH \alpha_a-\tfrac{1}{s}\sH \L_{\xi}^{\star}\alpha \\
\star_s\di\!\sum_a Z^a \wedge \sigma_a=&2s \tT(\sigma) Z^{123}-\mathfrak{S}_{abc} Z^{ab} \wedge \sH \dH\sigma_a+\tfrac{1}{s} \sum_{a} Z^{a} \wedge \sH (C \sigma-2\sigma)_a\\
\star_s\di\beta=&Z^{123} \wedge \sH \dH \beta+\tfrac{1}{s}\mathfrak{S}_{abc} Z^{ab} \wedge \sH \L_{\xi_c}\beta
\end{split}
\end{equation*}
for $(F,\alpha,\sigma,\beta)\in \mathbb{V}^3\H$.
The claim follows now 
by gathering terms and using the purely algebraic identity 
$-\sH \circ L_{\omega}=\PP \ \mbox{on} \ \Lambda^1(\H,\bbR^3).$
\end{proof}

As a direct consequence the forms $\varphi_s=\iota_s(1,0,\omega,0)$ and $\widetilde{\varphi}_s=\iota_s(-6,0,\omega,0)$ satisfy 
\begin{equation*}
\begin{split}
&\tfrac{7}{2s}\star_s\di\!\varphi_s=6(2+\tfrac{1}{s^2})\varphi_s+(\tfrac{1}{s^2}-5)\tilde{\varphi}_s, \ \ \ 
\tfrac{7}{2s}\star_s\di\!\tilde{\varphi}_s=6(\tfrac{1}{s^2}-5)\varphi_s+(\tfrac{1}{s^2}-12)\tilde{\varphi}_s
\end{split}
\end{equation*}
by taking into account that $C\omega=4\omega$ and $\tT(\omega)=6$. In particular when $s=\t5$ it follows that 
\begin{equation} \label{str5}
\begin{split}
&\star_s\di\!\varphi_s=12s\varphi_s\\
&\star_s\di\!\widetilde{\varphi}_s=-2s\tilde{\varphi}_s
\end{split}
\end{equation}
as previously claimed in section \ref{2ndE}.

To deal with the block structure of $\di$ acting on two forms we consider, in analogy with section \ref{G2}, the isometry 
$\iota_s : \mathbb{V}^2\H:=\Omega^0(\H,\bbR^3) \oplus \Omega^1(\H,\bbR^3) \oplus \Omega^2\H \to \Omega^2M$ given by 
\begin{equation} \label{iota2}
\iota_s \left(\begin{array}{c}
f\\
\alpha\\
\sigma \end{array}\right):=\mathfrak{S}_{abc}f_cZ^{ab}+\sum_a Z^a \wedge \alpha_a+\sigma.
\end{equation}
A calculation entirely similar to that in the proof of Lemma \ref{blgen-1} shows that the exterior differential $\iota_s^{-1}\di
\iota_s:\mathbb{V}^2\H \to \mathbb{V}^3\H$ reads  
\begin{equation} \label{blgen-2}
\iota_s^{-1}\di \iota_s \left(\begin{array}{c}
f\\
\alpha\\
\sigma \end{array}\right)=\left(\begin{array}{ll}
&-\frac{1}{s}\L_{\xi}^{\star}f\\ 
&\frac{1}{s}(C-2)\alpha+\dH\!f\\  
&-\dH\!\alpha-2s L_{\omega}f+\frac{1}{s}\L_{\xi} \sigma\\ 
&\dH\!\sigma-2s\sH \tT(\alpha)
\end{array} \right).
\end{equation}%
This allows proving the following 
\begin{lema} \label{cod-gen}
The codifferential $\di^{\star_s}:\Omega^3M \to \Omega^2M$ reads 
\begin{equation*}
\iota_s^{-1}\di^{\star_s} \iota_s \left (\begin{array}{c}
F\\
\alpha\\
\sigma \\
\beta
\end{array} \right)=
\left( \begin{array}{ll}
-\tfrac{1}{s}\L_{\xi}F+\di_{\H}^{\star}\!\alpha-2sL_{\omega}^{\star}\sigma\\
\tfrac{1}{s}(C-2)\alpha-\di_{\H}^{\star}\!\sigma-2s\bbI\sH\beta \\[.4ex]
\tfrac{1}{s}\L_{\xi}^{\star}\sigma+\di_{\H}^{\star}\!\beta 
\end{array} \right).
\end{equation*}
\end{lema}
\begin{proof}
Since the operator 
$C:\Omega^{\star}(\H,\bbR^3) \to \Omega^{\star}(\H,\bbR^3)$ is self adjoint (see section \ref{sl2C})
and the maps $\iota_s$ are isometric the claim follows from \eqref{blgen-2} by $L^2$-orthogonality.
\end{proof}
To prepare the ground for the computations in the next section we list below those identities pertaining to the operators $C,\tT$ and $\p$ which are needed to determine the block structure of the half Laplacians $\di\!\di^{\star_s}$ and 
$\di^{\star_s}\!\di$. 
\begin{lema} \label{Id13H}The following hold on $\Omega^1(\H,\bbR^3)$
\begin{eqnarray}
(C-2)\circ \PP+\p & = & -\mathbb{I} \circ \L_{\xi}^{\star} \label{CP}\\
-L_{\omega}^{\star}\circ \dH & = & \di_{\H}^{\star}\circ \PP  \label{dHsym}\\
\tT \circ \sH \dH & = & \di_{\H}^{\star}\circ \tT. \label{idnew0}
\end{eqnarray}
\end{lema}
\begin{proof}
Pick $\alpha \in \Omega^1(\H,\bbR^3)$ and observe that the first two identities can be proved at the same time as follows. Evaluate the 
identity $\di^{\star_s}\star_s\di=0$ on $(0,\alpha,0,0)^T$ and project onto $\Omega^0(\H,\bbR^3) \oplus \Omega^2(\H,\bbR^3)$. After a short computation using the block form for $\di^{\star_s}$ respectively 
$\star_s\di$ in Lemma \ref{cod-gen} respectively Lemma \ref{blgen-1} we obtain 
\begin{equation*}
\begin{split}
&\di_{\H}^{\star}(\P\alpha)+L_{\omega}^{\star}(\sH \dH\!\alpha)=0\\
&2(C-2)\PP\alpha+\d_{\H}^{\star}\sH\dH\!\alpha+2\bbI \L^{\star}_{\xi}\alpha=0.
\end{split}
\end{equation*}
Equation \eqref{dHsym} is thus proved since $L^{\star}_{\omega}\sH=L_{\omega}^{\star}$. Using that $\di_{\H}^{\star}\sH=-\sH\dH$ on $\Omega^2\H$ together with \eqref{def-p} in the second displayed equation above proves \eqref{CP}. To prove \eqref{idnew0}, observe that direct computation based on the definition of the map $\tT$ ensures that 
\begin{equation*}
\tT(\star \dH\! \alpha)=\sH \sum_a \star \dH\! \alpha \wedge \omega_a=\sH \sum_a \dH\! \alpha \wedge \omega_a=\sH \sum_a \dH (\alpha \wedge \omega_a)=-\sH\dH\sH \tT(\alpha)=\di_{\H}^{\star}\tT(\alpha).
\end{equation*}
\end{proof}
\subsection{The components of the Laplacian $\Delta^{g_s}$} \label{blockLap}
The aim in this section is to investigate the block structure of the Laplacian $\iota_s^{-1}\Delta^{g_s}\iota_s$ acting on
$$\mathbb{V}^3\H=\Omega^0\H \oplus \Omega^1(\H,\bbR^3) \oplus \Omega^2(\H,\bbR^3) \oplus \Omega^3\H.
$$
The projections from the latter space onto each summand will be denoted with $\pr_k$ where $0 \leq k \leq 3$ indicates form degree.
By a slight abuse of notation we identify in what follows the operators $\iota_s^{-1}\Delta^{g_s}\iota_s$ and $\Delta^{g_s}$ as well as  $\iota_s^{-1}(\star_s\di)\iota_s$ and 
$\star_s\di$. We indicate now a quick way of computing the component $\pr_1 \Delta^{g_s}$ which essentially relies on formally multiplying the operator matrices for $\di$ and $\di^{\star}$ found in section \ref{dsys}.
Explicitly we first use the matrix form for $\di^{\star_s}$ in Lemma \ref{cod-gen} and the matrix form for $\di$ in \eqref{blgen-2} to arrive, after composition, at 
\begin{equation*}
\di\!\di^{\star_s}\!\left( \begin{array}{ll}
0\\
\alpha\\
0\\
0
\end{array} \right)
=\left( \begin{array}{ll}
& -\frac{1}{s}\L^{\star}_{\xi}\di_{\H}^{\star}\alpha\\
& \frac{1}{s^2}(C-2)^2\alpha+\dH\!\di_{\H}^{\star}\alpha\\
&-\frac{1}{s}\dH(C-2)\alpha-2sL_{\omega}(\di_{\H}^{\star}\alpha)\\
&-2 \sH\tT(C-2)\alpha
\end{array} \right)
\end{equation*} 
where $\alpha \in \Omega^1(\H,\bbR^3)$. Similarly 
\begin{equation*} \di^{\star_s}\!\di\!\left( \begin{array}{ll}
0\\
\alpha\\
0\\
0
\end{array} \right)
=\left( \begin{array}{ll}
&2s \tT(\sH \dH\!\alpha)+\frac{1}{s}\di_{\H}^{\star}\L_{\xi}^{\star}\alpha\\
& 4s^2\PP^2\!\alpha+\di_{\H}^{\star}\dH\!\alpha+\frac{1}{s^2}\L_{\xi}\L^{\star}_{\xi}\alpha\\
&\frac{1}{s}(C-2)\dH\!\alpha-2s\sH\dH\!\PP\!\alpha\\
&2\sH\L^{\star}_{\xi}\PP\!\alpha
\end{array} \right)
\end{equation*} 
by Lemma \ref{cod-gen} after taking into account that $\di^{\star_s}\di=(\star_s\di)^2$ on $\Omega^3M$. At the same time, using again the block form for $\star_s\di$ 
shows that 
\begin{equation*}
\star_s\di \left( \begin{array}{ll}
0\\
\PP\!\alpha\\
0\\
\sH \tT(\alpha)
\end{array} \right)
=\left( \begin{array}{ll}
-2s\di_{\H}^{\star} \tT(\alpha) \vspace{1mm}\\
-2s\PP^2\!\alpha-\frac{1}{s}\L_{\xi}\tT(\alpha)\\
\sH\dH\!\PP\!\alpha\\
-\frac{1}{s}\sH\L_{\xi}^{\star}\alpha
\end{array} \right)
\end{equation*}
since $\di_{\H}^{\star}=-\sH\dH\sH$ and $\sH\L_{\xi}\sH=-\L_{\xi}$ on $\Omega^1\H$. 
At this stage, in order to simplify these expressions, we start using the identities from the previous sections.
As the operators $\dH$ and $\di_{\H}^{\star}$ are both $\su(2)$-invariant we have $[\dH,C]=[\di_{\H}^{\star},\L^{\star}_{\xi}]=0$.
Thus putting the two half Laplacians above together whilst using \eqref{idnew0} shows that 
\begin{equation} \label{id-nr}
\Delta^{g_s} \left(
\begin{array}{ll}
0\\
\alpha\\
0\\
0
\end{array} \right )
+2s\star_s\di
\left( \begin{array}{ll}
0\\
\PP\!\alpha\\
0\\
\sH\tT(\alpha)
\end{array} \right )
=\left( \begin{array}{ll}
0\\
\Delta_{\H}\alpha+\frac{1}{s^2}((C-2)^2+\L_{\xi}\L^{\star}_{\xi})\alpha-2\L_{\xi}\tT(\alpha)\\
-2sL_{\omega}\di_{\H}^{\star}\alpha\\
-2\sH\tT(C-2)\alpha
\end{array} \right).
\end{equation}
This observation allows computing $\pr_1\Delta^{g_s}$ on 
the subspace 
\begin{equation} \label{scrS}
 \mathscr{S}:=\{(F,\alpha,\sigma, \beta)^T: \sigma \in \Omega^2_{sym}(\H,\bbR^3), \ F=-\tT(\sigma), \ \beta=\sH \tT(\alpha)\}
\end{equation}
of $\mathbb{V}^3\H$ which corresponds to $\Omega^{3}_{27}(\varphi_s)$ via $\iota_s$.
\begin{pro} \label{comp1}
We have 
\begin{equation*} 
\pr_1 \Delta^{g_s}=-2s\PP\pr_1(\star_s \di)-2s \bbI\sH \pr_3(\star_s \di)+\mathscr{G}^s\pr_1
\end{equation*}
on $\mathscr{S}$ where the second order differential operator 
$\mathscr{G}^s:\Omega^1(\H,\bbR^3) \to \Omega^1(\H,\bbR^3)$ is given by $$\mathscr{G}^s:=\Delta_{\H}+\frac{1}{s^2}\C-2\!\p-2(1+\frac{1}{s^2})(C-2).$$
\end{pro}
\begin{proof}First we list the adjoints for all the operators appearing in \eqref{id-nr}. The Laplacian $\Delta_{\H}$ together with 
$\L_{\xi}\L^{\star}_{\xi},\P$ and $C$ are self-dual. The duals of $\tT:\Omega^1(\H,\bbR^3) \to \Omega^1\H$ respectively 
$\sH\tT:\Omega^1(\H,\bbR^3) \to \Omega^3\H$ are given by $-\bbI$ respectively $\bbI \sH$. Now we consider the adjoint 
of the identity \eqref{id-nr}, as follows. Take the $L^2$ scalar product of \eqref{id-nr}
with an arbitrary element $(F_1,\alpha_1,\sigma_1,\beta_1)^T \in \bbV^3\H$ and take the adjoints for all the operators involved. In this way 
we see that the adjoint of the l.h.s of \eqref{id-nr} is $\pr_1 \Delta^{g_s}+2s\PP\pr_1(\star_s \di)+2s \bbI\sH \pr_3(\star_s \di)$ while
that of its r.h.s. acts on  $(F_1,\alpha_1,\sigma_1,\beta_1)^T$ according to
$$ \Delta_{\H}\alpha_1+\frac{1}{s^2}((C-2)^2+\L_{\xi}\L_{\xi}^{\star})\alpha_1+2\bbI \L_{\xi}^{\star}\alpha_1-2s\dH L^{\star}_{\omega}\sigma_1-2(C-2)\bbI \sH\beta_1.
$$
Assuming now that $(F_1,\alpha_1,\sigma_1,\beta_1)^T \in \mathscr{S}$, so that $ L^{\star}_{\omega}\sigma_1=0$ since $\sigma_1 \in 
\Omega^2_{sym}(\H,\bbR^3)$ (see Remark \ref{dual-rmk}) and $\sH\beta_1=-\tT(\alpha_1)=
\bbI^{\star}\alpha_1$ it follows that proving the claim amounts to computing the operator 
\begin{equation*}
\tfrac{1}{s^2}((C-2)^2+\L_{\xi}\L_{\xi}^{\star})+2\bbI \L_{\xi}^{\star}+2(C-2)\bbI \tT
\end{equation*}
acting on $\Omega^1(\H,\bbR^3)$. With the aid of the characteristic polynomial for $C$ in \eqref{sqC} this reads 
\begin{equation*}
\begin{split}
& \tfrac{1}{s^2}(\C-2(C-2))+2\bbI \L_{\xi}^{\star}+2(C-2)(\PP-1)
=\tfrac{1}{s^2}\C-2\!\p-2(1+\tfrac{1}{s^2})(C-2)
\end{split}
\end{equation*}
after re-arranging terms and using \eqref{CP}. The proof of the claim is thus complete.
\end{proof}
\section{Spectral theory for $\star_s\di$ and embedding operators} \label{EDS}
The aim in this section is two-folded. The first objective is to determine in an explicit way the dependence of the eigenspaces $E_{\lambda}$ on the parameter $s$ as these do not relate in a direct way to eigenspaces for $\Delta^g$. 
The second is to examine how $E_{\lambda}$ relates to a subspace of $\Omega^1(\H,\bbR^3) \oplus \Omega^2_{sym}(\H,\bbR^3)$ via the isomophism $\kappa_s$. To carry out this programme several technical ingredients and clarifications are needed as follows.
\subsection{Properties of the $\su(2)\oplus \sp(1)$ action on $\Omega^{1}\H$} \label{vops}
Key to understanding the structure of the eigenspaces of $\Delta^{g_s}$ is producing the full set of algebraic relations satisfied by the operators $\C,\p,\L_{\xi},\bbI$ acting on $\Omega^1\H$ or on $\Omega^1(\H,\bbR^3)$. In more abstract terms we look at the action of $\su(2) \oplus \sp(1)$ on $\Omega^1\H$ induced by \eqref{ro} and \eqref{sp1}. Furthermore we need a good description of the action 
of $C$ on the latter space. Firstly, we observe that 
\begin{lema} \label{pvsC}
The operators $\p$ and $\C$ satisfy 
\begin{eqnarray}
(\p-2) \circ I_a+I_a \circ (\p-2) & = & -2\mathscr{L}_{\xi_a} \label{pI}\\
\p^2-2\!\p & = & \C \label{PI2} 
\end{eqnarray}
as well as  $[\C-2\p,I_a]= 0$ on $\Omega^1\H$.
\end{lema}
\begin{proof}
Pick $\alpha \in \Omega^1\H$; we compute 
\begin{equation*}
\begin{split}
\p(I_1\alpha)=&\sum_a (I_a \L_{\xi_a})I_1\alpha=-\L_{\xi_1}\alpha+I_2((\L_{\xi_2}I_1)\alpha+I_1 \L_{\xi_2}\alpha)+I_3((\L_{\xi_3}I_1)\alpha+I_1 \L_{\xi_3}\alpha)\\
=&-\L_{\xi_1}\alpha-I_1(I_2 \L_{\xi_2}+I_3 \L_{\xi_3})\alpha+(I_2 (\L_{\xi_2}I_1)+I_3(\L_{\xi_3}I_1))\alpha.
\end{split}
\end{equation*}
As $I_2 \L_{\xi_2}+I_3 \L_{\xi_3}=\p-I_1 \L_{\xi_1}$ and 
\begin{equation*}
(I_2 (\L_{\xi_2}I_1)+I_3(\L_{\xi_3}I_1))\alpha=-2I_2I_3\alpha+2I_3I_2 \alpha=4I_1\alpha
\end{equation*}
by \eqref{Lie-osI}, the claim in \eqref{pI} is proved for $a=1$. The relation between $\mathscr{C}$ and $\p$ in \eqref{PI2} follows from 
\eqref{pI} by taking into account that $\p$ is $\su(2)$-invariant; indeed this leads to 
$(\p-2)\L_{\xi_a}I_a+\L_{\xi_a}I_a(\p-2)=-2\L^2_{\xi_a}$ which grants the desired relation after summation over $a$. Finally, and again by using \eqref{pI}, the operators 
$\p-2$ and $(\p-2)I_a+I_a(\p-2)$ commute thus so do $(\p-2)^2=\mathscr{C}-2\!\p+4$ and $I_a$.
\end{proof}
Secondly, and as a direct consequence of Lemma \ref{pvsC}, we prove that 
\begin{coro}The following hold on $\Omega^1(\H,\bbR^3)$
\begin{eqnarray} 
C &=& 2-\p+\bbI \circ (-\L_{\xi}^{\star}+\p\circ \tT-2\tT)+\L_{\xi}\circ \tT \label{up-C-temp} \\
\tT \circ \p & = & (4-\p) \circ \tT+2\L_{\xi}^{\star}  \label{TT1} \\
\tT \circ \C &=& (\C+8-4\p) \circ \tT+4\L_{\xi}^{\star} \label{TT2}\\
\tT \circ C &=& (4-\p)\circ \tT+\L_{\xi}^{\star}. \label{TT3}
\end{eqnarray}
\end{coro}
\begin{proof}
Since the operator $\PP=1+\bbI \circ \tT$ is invertible with $\PP^{-1}=\frac{1}{2}(\PP+1)$ we derive that $C-2+\frac{1}{2}\p\circ(\PP+1)=-\frac{1}{2}\bbI \circ \L_{\xi}^{\star}\circ (\PP+1)$ with the aid of \eqref{CP}. Because $-\L_{\xi}^{\star}\circ \bbI=\p$ we get 
$\L_{\xi}^{\star}\circ \PP=-\L_{\xi}^{\star}+\p\circ \tT$, fact which leads to 
$$ C+\p-2=-\bbI \circ \L_{\xi}^{\star}+\frac{1}{2}(\bbI \circ \p-\p\circ \bbI)\circ \tT.
$$
The first displayed identity follows now from \eqref{pI}.

Identity \eqref{TT1} follows directly from \eqref{pI}. As the operator $\C-2\p$ is invariant under $\{L_{\xi_1},L_{\xi_2},
L_{\xi_3},I_1,I_2,I_3\}$
we have $\tT \circ (\C-2\p)=
(\C-2\p)\circ \tT$ thus \eqref{TT2} follows from \eqref{TT1}. Finally, acting with $\tT$ on the left hand side of \eqref{up-C-temp} shows that 
$$ \tT \circ C=2\tT-\tT \circ \p-3(-\L_{\xi}^{\star}+\p \circ \tT-2\tT)+\p\circ \tT=2(4-\p)\circ \tT-\tT \circ P+3\L_{\xi}^{\star}
$$
after taking into account that $\tT \circ \bbI=-3$ and $\tT \circ \L_{\xi}=\p$. The last identity in the claim follows now from 
\eqref{TT1}.
\end{proof}
Therefore the operator 
$C$ acting on $\Omega^1(\H,\bbR^3)$ is entirely determined by $\p,\mathbb{I}$ together with the contracted Lie derivative 
and the algebraic trace map $\tT$. In the next section we will crucially rely on this observation to determine eigenspaces of type
$\ker(\star_s\di-\lambda) \cap \Omega^3_{27}M$.
Next we establish additional commutation relations for the differential operators $\dH$ and $\Delta_{\H}$. These will be needed to determine how the Hodge decomposition 
of $\Omega^1\H$ behaves w.r.t. the $\sp(1)$-action.
\begin{lema} \label{van1}
We have $\Delta_{\H} \circ \dH=\dH \circ \Delta_{\H}+2\p \circ \dH \ \mbox{on} \ C^{\infty}M.$ 
\end{lema}
\begin{proof}Pick $F \in C^{\infty}M$ and observe that $(\Delta_{\H} \circ \dH-\dH \circ \Delta_{\H})F=\di_{\H}^{\star}\di_{\H}^2F$ from the definitions. Then $\di_{\H}^2F=-2\sum_a (\mathscr{L}_{\xi_a}F) \omega_a$, according to \eqref{sqH}. Since the forms $\omega_a$ are selfdual with $\dH\omega_a=0$ we get, by also using \eqref{Ho3}
\begin{equation*}
\di_{\H}^{\star}\di_{\H}^2F=2\sH \sum_a (\dH \mathscr{L}_{\xi_a}F) \wedge \omega_a=2\sum_a I_a \dH  \mathscr{L}_{\xi_a}F=2
\sum_a I_a \mathscr{L}_{\xi_a}\dH\! F=2\!\p(\dH\!F).
\end{equation*}
\end{proof}
Since the operator $\p$ is symmetric we also have the dual identity 
\begin{equation} \label{com22}
\Delta_{\H} \circ \di_{\H}^{\star}= \di_{\H}^{\star} \circ \Delta_{\H}-2 \di_{\H}^{\star} \circ \p
\end{equation}
on $\Omega^1\H$. The following set of identities will be systematically used in this paper.
\begin{lema} \label{batch1}The following hold for  $f \in C^{\infty}M$
\begin{itemize}
\item[(i)] $\di_{\H}^{\star}\mathbb{I}\dH\!f=-4\L_{\xi}f$
\item[(ii)] $\di_{\H}^{+}\mathbb{I}\dH\!f=-\frac{1}{2}(\Delta_{\H}f+16f)\omega+2C(f\omega)$
\item[(iii)] \label{codi-p} $\di_{\H}^{\star} \p \di_{\H}f=4\C f$.
\end{itemize}
\end{lema}
\begin{proof}
(i) with the aid of \eqref{Ho3} and \eqref{sqH} we see that 
\begin{equation*}
\di_{\H}^{\star}I_a\dH\!f=\sH \dH(\dH\!f \wedge \omega_a)=\sH(\di_{\H}^2f \wedge \omega_a)=-2\L_{\xi_a}f \sH \omega_a^2=-4\L_{\xi_a}f
\end{equation*}
which proves the claim.\\
(ii) the diagonal terms in $\di_{\H}^{+}\bbI\dH\!f$, w.r.t. the basis $\{\omega_a, 1 \leq a \leq 3\}$ in $\Lambda^{+}\H$, are determined from 
$\dH I_a \dH\!f \wedge \omega_a=\dH (I_a \dH\!f \wedge \omega_a)=\dH \sH \dH\!f=-\Delta_{\H}f\vol_{\H}$. To compute the remaining terms we 
start from the identity $I_a\dH\!f \wedge \omega_b=-I_b\dH\!f \wedge \omega_a=\dH\!f \wedge \omega_c$, as entailed by \eqref{Ho44}, with cyclic permutation on $abc$. Since $\dH\omega_a=\dH\omega_b=0$ it follows that $ \dH I_a \dH\!f \wedge \omega_b=-
\dH I_b \dH\!f \wedge \omega_a$. At the same time, by also using \eqref{sqH}
$$ \dH I_a \dH\!f \wedge \omega_b=\dH (I_a \dH\!f \wedge \omega_b)=\di_{\H}^2f \wedge \omega_c=-2(\L_{\xi_c}f) \omega_c^2.
$$
The claimed expression for $\di_{\H}^{+}\bbI\dH\!f$ follows from $\omega_1^2=\omega_2^2=\omega_3^2=2\vol_{\H}$ and \eqref{Lie-os}.\\
(iii) follows from (i) and $\p=-\L_{\xi}^{\star} \circ \bbI$ on $\Omega^1\H$ since $\di_{\H}$ is $\su(2)$-invariant.
\end{proof}
To end this section we consider the operator $C^{\infty}_{b}M \to \Omega^{-}(\H,\bbR^3), f \mapsto \di_{\H}^{-}\bbI\dH\!f$ which 
will be needed for the embedding result in section \ref{alg} and to establish eigenvalue estimates 
in section \ref{HoEe}. Note that $\di=\dH$ on invariant functions. We prove that 
\begin{coro} \label{co-r1}
Whenever $f\in \ker(\Delta_{b}-\nu)$ we have 
\begin{itemize}
\item[(i)] $\sH\dH(\di_{\H}^{-}\bbI\dH\!f)=\frac{\nu-16}{2}\bbI\dH\!f$
\item[(ii)] $\int_M\vert \di_{\H}^{-}\bbI\dH\!f\vert^2\vol=\frac{3(\nu-16)}{2}\int_M \vert \dH\!f\vert^2\vol$.
\end{itemize}
 
\end{coro}
\begin{proof}
According to part (ii) in Lemma \ref{batch1} we have $\di_{\H}^{-}\bbI\dH\!f=\dH\bbI\dH\!f+\frac{\nu}{2}f\omega$ thus 
\begin{equation*}
\sH\dH(\di_{\H}^{-}\bbI\dH\!f)=\sH\di_{\H}^2(\bbI\dH\! f)+\tfrac{\nu}{2}\sH(\dH\!f \wedge \omega)=-2\p(\bbI\dH\! f)+\tfrac{\nu}{2}\bbI\dH\! f
\end{equation*}
by using \eqref{def-p} and \eqref{Ho3}. Since $f$ is $\su(2)$-invariant we have $\p(\bbI \di\!f)=4\bbI \di\!f$ by \eqref{pI} and the claim 
in (i) follows. Part (ii) follows from (i) by integration using that $\sH\dH=\di_{\H}^{\star}$ on $\Omega^{-}(\H,\bbR^3)$.
\end{proof}
\subsection{Eigenspace properties} \label{sectn52}
In this section we work exclusively with the value $s=\t5$. The aim is to combine the $\su(2)$ splitting of $\Omega^3_{27}(\varphi_s)$ from section \ref{G2} and the block structure of the Laplacian $\Delta^{g_s}$ in Proposition \ref{comp1} to study pairs $(\alpha,\sigma) \in 
\Omega^1(\H,\bbR^3) \oplus \Omega^2_{sym}(\H,\bbR^3)$ such that $\kappa_s(\alpha,\sigma) \in E_{\lambda}$ with 
$\lambda \in \mathbb{R}$. It will be sometimes useful to record that this requirement on $(\alpha,\sigma)$ corresponds to the first order exterior differential system 
\begin{equation} \label{sys1}
\begin{split}
&\di_{\H}^{\star}\!\tT(\alpha)=(\lambda+2s)\tT(\sigma)\\
&\star_{\H}\dH\!\sigma+\tfrac{1}{s}\L_{\xi}\tT(\alpha)+2s\mathbb{I}\tT(\alpha)=-(\lambda+2s) \alpha\\
&(C-2)\sigma+s\dH\!\alpha-2s^2\tT(\sigma)\omega=s\lambda \star_{\H} \sigma\\
&\L^{\star}_{\xi}\alpha=-s\lambda \tT(\alpha)+s\dH\!\tT(\sigma).
\end{split}
\end{equation}
This follows from the block structure of $\star_s\di$ in Lemma \ref{blgen-1}, with $F=-\tT(\sigma)$ and $\beta=\sH\tT(\alpha)$. 

We will derive differential constraints pertaining only on $\alpha$ and on its scalar valued invariants 
$\L_{\xi}^{\star}\alpha$ and $\tT(\alpha)$. To carry out this programme consider the second order 
differential operator $\scrD:\Omega^1\H \to \Omega^1\H$ given by 
\begin{equation*}
\scrD:=\Delta_{\H}+5\C-2\p
\end{equation*} 
which enters the following preliminary 
\begin{lema} \label{tG}
We have 
$\tT \circ\, \mathscr{G}^{\t5}=\scrD \circ \tT$
on $\Omega^1(\H,\bbR^3)$.
\end{lema}
\begin{proof}
Since $\Delta_{\H}+\C$ is $\sp(1)$-invariant by Lemma \ref{comI}, it commutes with the trace map $\tT$. We compute, by succesively using \eqref{TT2},\eqref{TT1} as well as \eqref{TT3}
\begin{equation*}
\begin{split}
\tT\circ \mathscr{G}^{\t5}=&(\Delta_{\H}+\C)\circ \tT+\ 4\tT\circ \C-2\tT\circ \p-12\tT \circ (C-2)\\
=&(\Delta_{\H}+\C)\tT+4((\C+8-4\p)\circ \tT+4\L_{\xi}^{\star})-2((4-\p)\circ \tT+2\L_{\xi}^{\star})-12((2-\p)\circ \tT+\L_{\xi}^{\star}).
\end{split}
\end{equation*}
The claim follows by gathering terms.
\end{proof}
\begin{rema}\label{deltabar}
Perhaps not accidentally the operator $\scrD$ acting on $\Omega^1 \H $ can be viewed as a Laplace-type operator defined with the aid 
of the canonical connection $\overline{\nabla}$ of the nearly $\G_2$ structure $\varphi_{\s5}$. This connection can be characterised as the unique metric connection with torsion proportional to $\varphi_s$. The associated Laplace-type operator $\bar \Delta$ acting on 
$\Omega^{\star}M$ is defined 
according to $\bar \Delta = \overline{\nabla}^{\star} \overline{\nabla} + q(\bar R)$, where $q(\bar R)$ is a curvature term, linear in the curvature $\bar R$ of $\overline{\nabla}$
(see \cite{AlS} for details). Then the comparaison formula from \cite[Prop. 5.1]{AlS} yields after a short calculation 
$
\bar \Delta \alpha  = \Delta^{g_{s}} \alpha    +   \tfrac{2}{\sqrt{5}} \mathrm{pr}_{\Lambda^1} (\di\! \alpha)
$
for $\alpha \in \Omega^1\H$. Here  $ \mathrm{pr}_{\Lambda^1}$ denotes the projection given by $ \mathrm{pr}_{\Lambda^1}(A \wedge B)  = B \lrcorner\, A\lrcorner \,\varphi_s$
for tangent vectors $A,B \in TM$. Since $\H$ is a co-associative $4$-plane we have $ \mathrm{pr}_{\Lambda^1} (\Lambda^2 \H) \in \V$ as well as $ (\mathrm{pr}_{\Lambda^1} (\di\!\alpha))_\H = -\sqrt{5} \p(\alpha)$, making that 
\bea
(\bar \Delta \alpha)_\H &=& (\Delta^{g_s} \alpha)_\H  \; +\;  \tfrac{2}{\sqrt{5}} (\mathrm{pr}_{\Lambda^1} (\di\! \alpha))_\H
\;=\;
(\Delta_\H + 5 \C ) \alpha \;-\; 2 \p(\alpha)  \;  = \; \scrD \alpha \ .
\eea
Since $\bar \Delta$ preserves the distribution $\H$ it follows that $\bar \Delta=\scrD$ on $\Omega^1\H$.
\end{rema}

To be able to state our first structure results we introduce several spaces of harmonic forms starting with 
$$\mathbb{H}:=\{\sigma \in \Omega^2\H : \dH\sigma=\di_{\H}^{\star}\sigma=0\}$$
which splits as $\mathbb{H}=\mathbb{H}^{-} \oplus \mathbb{H}^{+}$ according to 
$\Lambda^2\H=\Lambda^{-}\H \oplus \Lambda^{+}\H$. In addition, let 
\begin{equation} \label{h-spc}
\begin{split}
&\bfH_{\lambda}^{-}:=(\mathbb{H}^{-} \otimes \bbR^3) \cap \ker(C-\lambda) \cap 
\ker \L_{\xi}^{\star}\\
&\mathbf{H}_{\lambda}^{+}:=(\mathbb{H}^{+} \otimes \bbR^3) \cap \ker(C-\lambda) \cap 
\ker (\L_{\xi}^{\star} \oplus \tT) \cap \Omega^{+}_{sym}(\H,\bbR^3)\edz{find phrasing!!!}\\
& \mathbb{H}_{\lambda}:=\mathbb{H} \cap \ker(\C-\lambda)
\end{split}
\end{equation}
for $\lambda \in \bbR$, where we recall that $\L_{\xi}^{\star} \oplus \tT:\Omega^{+}(\H,\bbR^3) \to \Omega^1\H \oplus \Omega^1\H$ is the direct sum map. Spaces of type $\bfH^{\pm}_{\lambda}$ are, as \eqref{sqC} shows, contained in $(\mathbb{H} \otimes \bbR^3) \cap \ker(\C-\lambda(\lambda-2))$ thus they are finite dimensional and $\su(2)$-invariant. As the Casimir operator 
of a finite dimensional irreducible, possibly with multiplicity, $\su(2)$-representation is an integer, of the form $m(m+2), m \in \mathbb{N}$ we conclude that   
\begin{equation} \label{van11}
(\mathbb{H}^{\pm} \otimes \bbR^3) \cap \ker(C-\lambda)=0 \ \mbox{for} \ \lambda \in \bbR \backslash \mathbb{Z}.
\end{equation}
In what follows we call a sequence $0 \to V_1 \stackrel{f_1}{\to}V_2 \stackrel{f_2}{\to}V_3$ of vector spaces and linear maps semi-exact provided that $\ker(f_1)=0$ and $\ker(f_2)=\im(f_1)$. When $V_1$ occurs as a subspace in $V_2$ and hence $f_1$ is the inclusion map  we simply use the notation $0 \to V_1 \hookrightarrow V_2 \stackrel{f_2}{\to} V_3$, when semi-exact amounts to $\ker(f_2)=V_1$.

These preparations allow relating the eigenspaces of the Laplacian on co-closed forms in $\Omega^3_{27}(\varphi_s)$, in other words 
spaces of type $E_{\lambda}$, to eigenspaces of 
the operator $\mathscr{G}^{\t5}$. Based on the identification $\Omega^3_{27}(\varphi_s)$ with the subspace $\mathscr{S} \subseteq \bbV^3\H$ (see \eqref{scrS}) we prove the following
\begin{pro} \label{int-l}
Assume that $\lambda(\lambda+2s) \neq 0$. We have a semi-exact sequence 
\begin{equation*}
0 \to \bfH_{2-s\lambda}^{-} \oplus \bfH^{+}_{2+s\lambda} \stackrel{\kappa_s(0,\cdot)}{\to} \ker(\star_s\di-\lambda) \cap \Omega^3_{27}(\varphi_s) \stackrel{\pr_1 }{\to} \ker(\mathscr{G}^{\t5}-\lambda(\lambda+2s))
\end{equation*}
with $\pr_1:\Omega^3_{27}(\varphi_s) \to \Omega^1(\H,\bbR^3)$ as defined in section \ref{blockLap}. If $\lambda=-2s$ we have a semi-exact sequence
$$0 \to \bbR\widetilde{\varphi}_s \hookrightarrow \ker(\star_s\di+2s) \cap \Omega^3_{27}(\varphi_s) \stackrel{\pr_1 }{\to} \ker \mathscr{G}^{\t5}.$$
\end{pro}
\begin{proof}
Let $\gamma=\kappa_s(\alpha,\sigma) \in \ker(\star_s\di-\lambda) \cap \Omega^3_{27}(\varphi_s)$. Since $\lambda \neq 0$ it follows that 
$\di^{\star_s}\gamma=0$ hence $\Delta^{g_s}\gamma=\lambda^2\gamma$. As $\pr_1(\gamma)=\alpha$ and $\pr_3(\gamma)=\sH\tT(\alpha)$ the projections 
of $\star_s\di$ satisfy 
$ \pr_1(\star_s\di)=\lambda \pr_1$ and $\pr_3(\star_s\di)=\lambda \sH\tT \circ \,\pr_1
$
on $\gamma$. Proposition \ref{comp1} thus yields 
$$\mathscr{G}^{\t5}\alpha=\lambda^2\alpha+2s\lambda\PP\!\alpha-2s\lambda\bbI\tT(\alpha)=\lambda(\lambda+2s)\alpha $$
since $\sH^2=-1$ on $\Omega^1\H$ and $\P=1+\bbI \circ \tT$. In other words the last arrow in the statement is well defined.\\
Now assume, in addition, that $\alpha=0$, that is $\gamma \in \ker \pr_1$. By \eqref{sys1} the requirement $\star_s\di\gamma=\lambda\gamma$ then reduces to 
\begin{equation*}
\begin{split}
&(\lambda+2s)\tT(\sigma)=0, \ \dH\tT(\sigma)=0, \ \dH \sigma=0, \ (C-2)\sigma=s\lambda\sH \sigma+2s^2\tT(\sigma)\omega.
\end{split}
\end{equation*}
There are two cases to distinguish as follows.
\begin{itemize}
\item[(i)]$\lambda(\lambda+2s)\neq 0$.\\
Here we must have $\tT(\sigma)=0$ which makes that $(C-2)\sigma=s\lambda \sH\sigma$ after updating the last equation above. This forces $\dH\sH \sigma=0$ since $[\dH,C]=0$ as well as $\L_{\xi}^{\star}\sigma=0$ after taking into the identity \eqref{sqC00}. Furthermore, projection onto $\Lambda^2\H=\Lambda^{-}\H \oplus \Lambda^{+}\H$ leads to $C\sigma^{\pm}=(2\pm s\lambda)\sigma^{\pm}$ which shows that 
$\sigma^{-} \in \bfH_{2-s\lambda}^{-}$. Since $\tT$ vanishes on $\Omega^{-}(\H,\bbR^3)$ we see that $\sigma^{+}$ satisfies $\tT(\sigma^{+})=0$. As $\sigma^{+} \in \Omega^{+}_{sym}(\H,\bbR^3)$ by assumption we have showed that $\sigma^{+} \in \bfH_{2+s\lambda}^{+}$. Therefore 
the statement on $\ker \pr_1$ is proved.
\item[(ii)] $\lambda+2s=0$.\\
Having the function $\tT(\sigma) \in \ker\dH$ entails that $\tT(\sigma)$ is constant, since the distribution $\H$ is bracket generating. As before 
$\sigma^{-} \in \bfH^{-}_{2-s\lambda}$. In addition, $\rho:=\sigma^{+}-\frac{\tT(\sigma)}{6}\omega$ satisfies $\tT(\rho)=0$ and 
$(C-2)\rho=s\lambda \rho$, hence $\rho \in \bfH_{2+s\lambda}^{+}$. 
As $s\lambda=-\frac{2}{5} \in \mathbb{Q}\backslash \mathbb{Z}$ both $\rho$ and $\sigma^{-}$ vanish by \eqref{van1}, hence 
$\sigma \in \spa\{\omega\}$. The claim of having the second sequence in the statement semi-exact follows since $\widetilde{\varphi}_s=\kappa_s(0,\omega)$. 
\end{itemize}
\end{proof}
For closed eigenforms of the Laplacian an analogous, though slightly different, argument shows that
\begin{pro} \label{int-2n}
If $\mu \neq 0$ we have a semi-exact sequence  
\begin{equation*}
0 \to \mathbb{H}_{s^2\mu} \stackrel{\L_{\xi}}{\to} \ker(\di\di^{\star_s}-\mu) \cap \Omega^3_{27}(\varphi_{\t5}) \stackrel{\pr_1 }{\to} \ker(\mathscr{G}^{\t5}-\mu).
\end{equation*}
\end{pro}
\begin{proof}
Let $\gamma=\kappa_s(\alpha,\sigma)$ belong to $\ker(\di\di^{\star_s}-\mu)$. As $\di\gamma=0$ the projected operators $ \pr_1(\star_s\di)$ and  $\pr_3(\star_s\di)$ both vanish on $\gamma$. Hence $\alpha$ belongs to $\ker(\mathscr{G}^{\t5}-\mu)$ by using again Proposition \ref{comp1}. To determine the kernel of the projection map $\pr_1$ assume now that $\alpha=0$. Closure for $\gamma=\iota_s(-\tT(\sigma),0,\sigma,0)$ is then equivalent to 
\begin{equation} \label{dds}
\tT(\sigma)=\dH\!\sigma=(C-2)\sigma=0
\end{equation} 
by Lemma \ref{blgen-1}. At the same time, the eigenvalue equation $\di\di^{\star_s}\gamma=72s^2\gamma$ becomes 
\begin{equation} \label{dds1}
\begin{split}
&(C-2)\di_{\H}^{\star}\sigma=0\\
&\dH\di_{\H}^{\star}\sigma+\tfrac{1}{s^2}\L_{\xi}\L^{\star}_{\xi}\sigma=\mu\sigma\\
&-\dH\L^{\star}_{\xi}\sigma=2s^2\sH \tT(\di_{\H}^{\star}\sigma)
\end{split}
\end{equation}
after a short computation based on \eqref{blgen-2} and Lemma \ref{cod-gen}. As $\dH\!\L^{\star}_{\xi}\sigma=0$ by using \eqref{dds} it follows that 
$$\di_{\H}^{\star}\sigma \in 
\{\alpha \in \Omega^1(\H,\bbR^3) : (C-2)\alpha=0, \tT(\alpha)=0\}.$$
Applying $\sH \dH$ in the second equation of \eqref{dds1} further yields $\p(\di_{\H}^{\star}\sigma)=0$ by means of \eqref{def-p}. It follows that $\di_{\H}^{\star}\sigma=0$ by using \eqref{up-C-temp}. Due to $(C-2)\sigma=0$ we get $\C\sigma=\L_{\xi}\L_{\xi}^{\star}\sigma$ by \eqref{sqC}, thus the second 
equation in \eqref{dds1} makes that $\C\sigma=\mu s^2\sigma$. In other words $\L_{\xi}^{\star}\sigma \in \mathbb{H}_{s^2\mu}$ whence the claim. 
\end{proof} 
To gain further insight into the structure of both types of form eigenspaces which occur in Proposition \ref{int-l} and Proposition 
\ref{int-2n} additional information on the eigenspaces of the operator $\mathscr{G}^{\t5}$ is needed. To that aim record that the operator $\scrD$ is elliptic and self-adjoint hences its eigenspaces 
$$F_{\lambda}:=\ker(\scrD-\lambda) \subseteq \Omega^1\H$$
where $ \lambda \in \bbR$ are finite dimensional. Moreover 
$\p(F_{\lambda}) \subseteq F_{\lambda}$ since $[\Delta_{\H},\p]=0$ by Lemma \ref{comI}.
Indicating the space of basic one forms on $M$ with 
$$\Omega^1_{inv}\H:=\{\alpha \in \Omega^1\H : \L_{\xi_a}\alpha=0\}$$ we let $\Omega^1_{\perp}\H$ be its $L^2$-orthogonal complement within 
$\Omega^1\H$ and observe that 
\begin{pro} \label{semi2} We have a semi-exact sequence 
\begin{equation*}
0 \to \ker(\Delta_{\H}+5\!\p^2-\lambda) \cap \Omega^1_{o}(\H,\bbR^3) \hookrightarrow \ker(\mathscr{G}^{\frac{1}{\sqrt{5}}}-\lambda) \stackrel{\L_{\xi}^{\star}\oplus \tT}{\to} F_{\lambda}^{\perp} \oplus F_{\lambda}
\end{equation*}
where $\Omega^1_{o}(\H,\bbR^3):=\Omega^1(\H,\bbR^3) \cap \ker(\L_{\xi}^{\star} \oplus \tT)$ and $F^{\perp}_{\lambda}:=F_{\lambda} \cap \Omega^1_{\perp}\H$. 
\end{pro} 
\begin{proof}
According to Lemma \ref{inv-tL} the operator $\p$ is $\su(2)$-invariant. As this is also the case for $\Delta_{\H}$ and the Casimir operator $\C$, all these $3$ operators commute with $\L_{\xi}^{\star}$, in particular 
$\L_{\xi}^{\star} \circ \scrD=\scrD \circ \L_{\xi}^{\star}$. As $\mathscr{G}^{\frac{1}{\sqrt{5}}}=\scrD-12(C-2)$ and $\L_{\xi}^{\star} \circ (C-2)=0$ by \eqref{sqC00} we obtain the identity 
$$\L_{\xi}^{\star} \circ \mathscr{G}^{\t5}=\scrD \circ \L_{\xi}^{\star}.$$ 
Now pick $\alpha \in \ker(\mathscr{G}^{\frac{1}{\sqrt{5}}}-\lambda)$ and observe that the above identity 
forces  
$\L_{\xi}^{\star}\alpha \in F_{\lambda}$. Since $\L_{\xi}^{\star}\alpha$ is $L^2$-orthogonal to $\Omega^1_{inv}\H$ we thus have  $\L_{\xi}^{\star}\alpha \in F_{\lambda}^{\perp}$. That $\tT(\alpha)$ belongs to $F_{\l}$ follows from Lemma \ref{tG} so the last arrow in the sequence in the statement is well defined. To prove semi-exactness for that sequence, assume, in addition that 
$\alpha \in \ker(\L_{\xi}^{\star} \oplus \tT)$, so that $\alpha$ belongs to $\Omega^1_{o}(\H,\bbR^3)$. It is now enough to observe that $C-2=-\p$ on $\Omega^1_{o}(\H,\bbR^3)$ by \eqref{up-C-temp} and hence $\mathscr{G}^{\frac{1}{\sqrt{5}}}=\Delta_{\H}+5\!\p^2$ on the latter space.
\end{proof}
\subsection{The embedding of $C^{\infty}_{b}M$ into $\Omega^3_{27}(\varphi_{\frac{1}{\sqrt5}})$} \label{alg}
The aim here is to give an explicit embedding of eigenspaces for the scalar basic Laplacian $\Delta_b$ into eigenspaces of type $E_{\lambda}$. For convenience we write $s=\s5$ throughout this section instead of using explicit numerics. We also assume that $g$ does not have constant 
sectional curvature; accordingly $\Delta_{b}>16$ on non-constant invariant functions as we shall see in Proposition \ref{est-inv} in the next section. In particular the embedding operators below are well defined.
\begin{pro} \label{embed1}
The map given by 
$$f \mapsto \varepsilon_{\nu}^{\pm}(f):=-\frac{1}{3}\kappa_{s}(-\bbI\di\!f,\frac{s}{2+s\lambda_{\pm}}\di_{\H}^{-}\bbI\di\!f+\frac{\lambda_{\pm}}{2}f\omega)$$
where $\lambda_{\pm}=-s\pm\sqrt{\nu+s^2}$ defines an embedding of  $\ker(\Delta_b-\nu)$ into $E_{\lambda_{\pm}}$.
\end{pro}
\begin{proof}
To explain how the embedding above has been found we make the following Ansatz. Consider the forms $\alpha=\frac{1}{3}\bbI\di\!f \in \Omega^1(\H,\bbR^3)$ and $\sigma=t_1\di_{\H}^{-}\bbI\di\!f+t_2f\omega \in \Omega^2_{sym}(\H,\bbR^3)$ where $t_1,t_2 \in \bbR$. We search for $\lambda \in \bbR$ 
such that $\gamma:=\kappa_s(\alpha,\sigma) \in \ker(\star_s\di-\lambda)$. In the process this requirement will also determine $t_1$ and $t_2$.

Since $f$ is invariant $C(f\omega)=4f\omega$. As $\di_{\H}^{\star}(f\omega)=-\bbI\d f$ and $C$ commutes with the operators $\di_{\H}^{\star}$ respectively $\dH$ it follows 
that $\bbI\d f$ and hence $\dH\bbI\d f$ as well as $\sigma$ belong to $\ker(C-4)$. Further on we have $\tT(\sigma)=6t_2f$ from the definition of $\sigma$ and $\dH \alpha=\frac{1}{3}(\di_{\H}^{-}\bbI\di f-\frac{\nu}{2}f\omega)$ by part (ii) in Lemma \ref{batch1}.

Based on Lemma \ref{blgen-1} with $F=-\tT(\sigma), \beta=\sH\tT(\alpha)$ these facts allow computing directly the components of the eigenvalue equation $(\star_s\di-\lambda)
\gamma=0$, starting with 
\begin{equation*}
\begin{split}
\pr_2(\star_s \di-\lambda)\gamma=&\frac{1}{s}\sH(C-2)\sigma+\sH\dH\alpha-2s\tT(\sigma)\omega-\lambda\sigma\\
=&-\frac{1}{s}(t_1(2+s\lambda)+\frac{s}{3})\di_{\H}^{-}\bbI\di f-((\lambda+2s)t_2+\frac{\nu}{6})f\omega.
\end{split}
\end{equation*}
The eigenvalue equation is thus satisfied when $t_1,t_2$ are determined from 
\begin{equation} \label{t1t2}
t_1(2+s\lambda)+\frac{s}{3}=(\lambda+2s)t_2+\frac{\nu}{6}=0.
\end{equation}
Since $\L^{\star}_{\xi}\alpha=0$ we have 
$\pr_3(\star_s\di-\lambda)\gamma=\sH(\dH\!\tT(\sigma)-\lambda \tT(\alpha))=(6t_2+\lambda)\sH\di\!f$
by taking into account that $\tT(\alpha)=-\di\!f$. Thus $6t_2+\lambda=0$, which plugged into the second equation of \eqref{t1t2} reveals 
that 
\begin{equation}\label{ln}
\lambda(\lambda+2s)=\nu.
\end{equation}
Record that \eqref{t1t2} can be solved for $t_1$ only if $\lambda \neq -\frac{2}{s}$; equivalently $\nu \neq 16$ which is granted by 
the general assumption in this section. To compute the projection of the eigenvalue equation on $\Omega^1(\H,\bbR^3)$ we first observe that using 
part (i) in Corollary \ref{co-r1} yields 
\begin{equation*}
\sH \dH\!\sigma=\frac{1}{2}(t_1(\nu-16)+2t_2)\bbI\di\!f.
\end{equation*}
Thus, after taking into account that $\PP\!\alpha=-\frac{2}{3}\bbI\di\!f$ and again $\tT(\alpha)=-\di\!f$ we get 
\begin{equation*}
\begin{split}
\pr_1(\star_s \di-\lambda)\gamma=&-2s\PP\alpha-\sH\dH\sigma-\tfrac{1}{s}\L_{\xi}\tT(\alpha)-\lambda\alpha
=-(\tfrac{\lambda-4s}{3}+\tfrac{1}{2}t_1(\nu-16)+t_2)\bbI\di\!f.
\end{split}
\end{equation*}
A short computation shows this vanishes 
when $\lambda(\lambda+2s)=\nu$ and $t_1,t_2$ satisfy \eqref{t1t2}. Finally the vanishing of 
\begin{equation*}
\pr_0(\star_s\di-\lambda)\gamma=(2s+\lambda)\tT(\sigma)-\di_{\H}^{\star}\tT(\alpha)=6t_2(2s+\lambda)f+\Delta_{\H}f=
(6t_2(2s+\lambda)+\nu)f
\end{equation*}
does not provide new information, as it coincides with the second equation in \eqref{t1t2}. Solving \eqref{ln} for $\lambda$, then 
expressing $t_1,t_2$ according to \eqref{t1t2} thus proves the claim. 
\end{proof}
For the pair $(\nu,\lambda)=(24,-12s)$ we obtain a linear injective map 
\begin{equation} \label{e-bis}
\begin{split}
&\varepsilon: \ker(\Delta_{b}-24) \to \Omega^3_{27}(\varphi_{s}), \ \ f \mapsto \tfrac{1}{3}\kappa_{s}(\bbI\dH\!f,\tfrac{1}{2s}\di_{\H}^{-}\bbI\dH\!f+6sf\omega) .
\end{split}
\end{equation}
Next we show that the operator $\varepsilon$ just defined can be alternatively described as stated in part (ii) of Theorem \ref{main1}.
\begin{pro} \label{alg-str} For any  $f\in \ker(\Delta_{b}-24)$ we have 
\begin{equation*}
\begin{split}
& \varepsilon(f)=\tfrac{\sqrt{5}}{6}\L_{\grad f}\varphi_s+\tfrac{12}{\sqrt{5}}f(\varphi_s  - 2 Z^{123}) - 2 \grad f \lrcorner \vol_{\H}  \ .\\
\end{split}
\end{equation*}
\end{pro}
\begin{proof}
This essentially amounts to the computation of $\iota_s^{-1}\L_{\grad f}\varphi_s$ which is outlined below, since the rest of terms 
in the r.h.s.of the statement are algebraic in $f$ and $\grad f$. 
Since $\grad f $ is horizontal and $\grad f \lrcorner \omega_a=I_a \dH f$ we have 
$\grad f \lrcorner \varphi_s=\iota_s
(0,-\bbI\dH\!f,0)^T \in \Omega^2M$ according to \eqref{iota2}. As seen before $f$ satisfies $(C-2)\bbI\dH\!f=2\bbI\dH\!f$ and  $\tT(\bbI\dH\!f)=-3\dH\!f$ thus with the aid of \eqref{blgen-2} we obtain
$
\di(\grad f \lrcorner \varphi_s)= \iota_s(0,-\frac{2}{s}\bbI\dH\!f,\dH\bbI\dH\!f,-6s\sH\dH\!f)^T.$
At the same time 
$\L_{\grad f} \varphi_s =\di(\grad f \lrcorner \varphi_s)+\grad f\lrcorner \di\varphi_s=
\di(\grad f \lrcorner \varphi_s)+12s\, \grad f \lrcorner \ast_s \varphi_s$, by Cartan's formula. As 
$ \grad f \lrcorner \star_s \varphi_s=\iota_s(0,\bbI\dH\!f,0,\grad f \lrcorner \vol_{\H})^T
$
and $\sH\di\!f=\grad f\lrcorner \vol_{\H}$ we find 
$$\L_{\grad f} \varphi_s=\iota_s(0,2s\bbI\dH\!f,\dH\bbI\dH\!f,6s\,\grad f\lrcorner \vol_{\H})^T.$$
Taking into account that $f(\varphi_s-2Z^{123})=\iota_s(-f,0,f\omega,0)^T$ the claim follows now easily. Notice that the final step 
here uses $\tT(\frac{\sqrt{5}}{6} \dH\! \mathbb{I}\dH\!f+\frac{12}{\sqrt{5}}f\omega)=12sf$, as established during the proof of Proposition \ref{embed1}.
\end{proof}
\section{Numerical eigenvalues} \label{numerics}
Recall that to determine infinitesimal Einstein deformations we need to describe eigenspaces of the type 
$\ker(\star_s\di-\lambda) \cap 
\Omega^3_{27}(\varphi_{\t5})$ for the numerical eigenvalues $\lambda=-\frac{12}{\sqrt{5}}$ and $\frac{6}{\sqrt{5}}$ as well as 
$\ker(\di\di^{\star_s}-\mu) \cap \Omega^3_{27}(\varphi_{\t5})$ for $\mu=\frac{72}{5}$. In addition, such eigenspaces with $\lambda(\lambda+2s) \leq 24$ respectively $\mu \leq 16$ turn up when looking at unstable directions for $g_{\t5}$. As we have seen in Proposition \ref{int-l} and Proposition \ref{int-2n} these problems reduce to the study of eigenspaces of perturbations of $\Delta_{\H}$ acting on subspaces of $\Omega^1(\H,\bbR^3)$. In this section we will develop eigenvalue estimates which will eventually lead to a complete description of the $\su(2)$ representation on spaces of this type 
and will also provide vanishing results.
\subsection{Weighted invariant spaces} \label{sp1}
Whenever $k \in \mathbb{Z}$ we consider the $\su(2)$ invariant spaces 
\begin{equation} \label{we}
\Omega^1_k{\H}:=\Omega^1\H \cap \ker(\p-k).
\end{equation}
According to Corollary \ref{dhp} these weighted spaces are preserved by the horizontal Laplacian $\Delta_{\H}$. A positivity argument based on \eqref{PI2} shows that $\Omega^1_0\H$ coincides with the space of invariant horizontal $1$-forms $\Omega^1_{inv}\H$. With respect to the foliation $\mathcal{F}$ those correspond to basic differential $1$-forms. The weighted spaces $\Omega^1_k\H$ are acted on by the Lie algebra $\sp(1)$ 
in the following way.
\begin{lema} \label{fourier}Assuming that $m \in \mathbb{N}$ the following hold 
\begin{itemize}
\item[(i)] the direct sum 
$ \Omega^1_{-m}\H \oplus \Omega_{m+4}^1\H $
is $\sp(1)$ invariant, that is invariant under the complex structures $I_a$
\item[(ii)] for $\alpha \in \Omega^1_{-m}\H $ the projection of $I_a\alpha$ onto $\Omega_{m+4}^1\H $ reads 
$(I_a\alpha)_{m+4}=I_a \alpha-\frac{1}{m+2}\L_{\xi_a}\alpha$ 
\item[(iii)] the map 
$ \Omega_{-m}^1\H \to \Omega^1_{m+4}\H, \ \alpha \mapsto (I_a \alpha)_{m+4}
$
is injective for each $a \in \{1,2,3\}$
\item[(iv)] we have $\L_{\xi}=-\bbI$ on $\Omega^1_3\H$.
\end{itemize}
\end{lema}
\begin{proof}
(i)\&(ii) are proved at the same time. Let $\alpha \in \Omega^1_{-m}\H$; from \eqref{pI} we get 
$(\p-(m+4))I_a\alpha=-2\L_{\xi_a}\alpha.$
As $\p$ is $\su(2)$-invariant we have $\L_{\xi_a}\alpha \in \Omega^1_{-m}\H$ thus $(\p+m)(\p-(m+4))I_a\alpha=0$. It follows that 
$I_a\alpha \in  \Omega^1_{-m}\H \oplus \Omega_{m+4}^1\H $ and moreover 
$(m+2)(I_a\alpha)_{-m}=\L_{\xi_a}\alpha$ by projection onto $\Omega^1_{-m}\H$.
Similarly, if $\alpha \in \Omega^1_{m+4}\H$ we have $(\p+m)I_a\alpha=-2\L_{\xi_a}\alpha$ hence 
$I_a\alpha \in  \Omega^1_{-m}\H \oplus \Omega_{m+4}^1\H $ and $(m+2)(I_a\alpha)_{m+4}=-\L_{\xi_a}\alpha$.\\
(iii) having $\alpha \in \Omega^1_{-m}\H$ satisfy $(I_1\alpha)_{m+4}=0$ is equivalent to 
$\L_{\xi_1}\alpha=(m+2)I_1\alpha$. It follows that $-\L_{\xi_1}^2 \alpha=(m+2)^2\alpha$. As $\C\alpha=m(m+2)\alpha$ this leads to  
$-(\L_{\xi_2}^2+\L_{\xi_3}^2)\alpha=-2(m+2)\alpha$. Hence $\alpha=0$ since the operator $-(\L_{\xi_2}^2+\L_{\xi_3}^2)$ is non-negative.\\ 
(iv) pick $\alpha \in \Omega_3^1\H$; since $\p\!\alpha=3\alpha$ we get $(\p-1)I_a\alpha=-2\L_{\xi_a}\alpha$, with the aid of \eqref{pI}. As 
$\p$ is $\su(2)$-invariant, it follows that $(\p-3)(\p-1)I_a \alpha=0$. Since $\C=\p^2-2\p=-1$ on $\ker(\p-1)$ and the operator $\C$ is non-negative 
it follows that $\ker(\p-1)=0$. Thus $(\p-3)I_a\alpha=0$ and the claim is proved by comparaison with $(\p-1)I_a\alpha=-2\L_{\xi_a}\alpha$.
\end{proof}

\subsection{Eigenvalue estimates for the horizontal Laplacian} \label{HoEe}
%
%
Based on the previous material we obtain  eigenvalue estimates for $\Delta_{\H}$ acting on 
$\Omega^1\H$ and $C^\infty M$. 
These estimates will play a crucial  r\^ole in describing infinitesimal Einstein deformations in the next section. We first record the available estimates in the invariant case where $\Delta_{\H}$ acting 
on $\Omega^{\star}_{inv}\H$ coincides with the basic Laplacian of the foliation $\mathcal{F}$.
If $(N^4,g_N)$ is an Einstein manifold with $\Ric^{g_N}=12g_N$ the classical results 
of Lichnerowicz and Obata provide that the first non-zero eigenvalue $\lambda_1$ of the Laplacian acting on functions respectively co-closed $1$-forms satisfies $\lambda_1 \geq 16$ respectively $\lambda_1 \geq 24$. Equality holds if $g_N$ has constant sectional curvature, respectively on the space of Killing vector fields. Clearly these estimates lift into estimates for the basic Laplacian on the total space of a Riemannian submersion with base $N$. On $C^{\infty}_{b}M$ this is sharper than the Lichnerowicz-Obata estimate for $g$ which asserts  that $\Delta^g\geq 7$ on $C^{\infty}M$; this is also sharper than the restriction to  $C^{\infty}_{b}M$ of the estimate $\Delta_{\H} \geq 4$ on $C^{\infty}M$
proved in \cite{I}. In our case $N=M\slash \mathcal{F}$ is in general not smooth, however the estimates carry through for Riemannian 
foliations, by work in \cite{LeeR}, which adapts to our situation as follows.
\begin{pro} \label{est-inv}
The scalar basic Laplacian acting on $C^{\infty}_{b} M \cap \{f : \int_M f\vol=0\}$ satisfies 
\begin{equation*}
\Delta_b \geq 16 \ \mathrm{and} \ \Delta_b >16 \ \mathrm{if} \ g \ \mathrm{does \ not \ have \ constant \ sectional \ curvature.}
\end{equation*}
\end{pro}
\begin{proof}
Viewing $\H$ as the normal bundle of the Riemannian foliation $\V$ the normal connection $\nabla^{\perp}$ in $\H$ is given by 
$\nabla^{\perp}_XY:=(\nabla^g_XY)_{\H}$ for $X,Y \in \Gamma(\H)$. Its curvature tensor $R^{\perp}$ is defined (see e.g. \cite{BoGa}) according to  
$(X,Y) \mapsto \nabla^{\perp}_{[X,Y]_{\H}}-[\nabla^{\perp}_X,\nabla^{\perp}_Y]$ and has Ricci contraction denoted by $\Ric^{\perp}$. In our case by using O'Neil's formulas we see that $\Ric^{\perp}=12g_{\H}$. Since $\V$ has codimension $4$ \cite[Theorem 4.4]{LeeR} ensures that the first non-zero eigenvalue of the basic Laplacian $\Delta_{b}$ is $\geq 16$. Note that this estimate also follows directly from Corollary \ref{co-r1},(ii). If equality holds $M$ is transversally isometric to 
$S^4\slash G$ by \cite[Theorem 5.1]{LeeR}, for some discrete subgroup $G\subseteq O(4)$. At tensorial level this entails 
$R^{\perp}(X,Y)=4X \wedge Y$; taking into account the O'Neill's formulas for $3$-Sasaki structures in dimension $7$ (see e.g.\cite{BoGa}) leads easily to having $g$ of constant sectional curvature. 
\end{proof}
In a very similar way the estimate 
\begin{equation} \label{basic-vf}
\Delta_{\H} \geq 24 \ \mbox{on} \ \Omega^1_{inv}\H \cap \ker\di_{\H}^{\star}
\end{equation}
follows from the Bochner formula on basic $1$-forms on $M$, see e.g. \cite[Theorem 2.2]{DJRi}. The limiting eigenspace 
consists of (basic) transversal Killing fields, again according to \cite{DJRi}. 
Using the extra input coming from the $3$-Sasaki structure this can be improved to 
\begin{equation*}
\ker(\Delta_{\H}-24) \cap \Omega^1_{inv}\H=\dH \{f \in C^{\infty}_{b}M : \Delta_{\H}f=24f \} \oplus \{X_{\H} : 
X \in \mathfrak{g}\}
\end{equation*}
where $\mathfrak{g}:=\{X \in \Gamma(TM): \L_{X}\xi^a=0 \}$ is the Lie algebra of automorphisms of the $3$-Sasaki structure.
However the second component space above does not embed in $E_{-\frac{12}{\sqrt{5}}}$ as we shall see during the proof of (i) in Theorem \ref{main1}, so this point will not be further developed. 

Combining the estimates in Proposition \ref{est-inv} and \eqref{basic-vf} shows 
\begin{equation} \label{forms-e}
\Delta_{\H}>16 \ \mbox{on} \ \Omega^1_{inv}\H
\end{equation} 
as $\Delta_{\H}$ and $\di_{\H}^{\star}$ commute on $\Omega^1_{inv}\H$.
Next we derive  lower bounds for the spectrum of $\Delta_\H$ restricted to the subspaces $\Omega^1_{-m} \H$ of $\Omega^1\H$ where $m \in  \mathbb{N}$, which generalise \eqref{basic-vf}. 
\begin{lema}\label{estimate}
We have 
$\Delta_{\H} >4(m+2) \ \mbox{on} \ \Omega^1_{-m} \H$
for $m \in  \mathbb{N}^{\times}$.
\end{lema}
\begin{proof}
For  $\alpha \in \Omega^1_{-m} \H \cap \ker (\Delta_\H - \lambda) $ we have $\C \alpha = (\p^2-2\!\p)\alpha = m(m+2) \alpha$ and thus
$(\Delta_\H + \C)\alpha = (\lambda +m(m+2))\alpha$. By the $\sp(1)$-invariance of  $\Delta_\H + \C$  the same equation holds with
$\alpha$ replaced by $\bbI \alpha \in \Omega^1_{-m} (\H,\bbR^3) \oplus \Omega^1_{m+4}(\H,\bbR^3)$.  Moreover, since $\p$ commutes with $\Delta_\H$ and
$\C$ we can project this eigenvalue equation onto $\Omega^1_{m+4}(\H,\bbR^3)$ where $\C$ acts by multiplication with $(m+2)(m+4)$.
Note that $(\bbI \alpha)_{m+4}\neq 0$ for $\alpha \neq 0$ due to part (iii) in Lemma \ref{fourier}. Then 
\begin{equation} \label{Ee2}
\Delta_{\H}(\bbI \alpha)_{m+4}=(\lambda-4(m+2))(\bbI \alpha)_{m+4}.
\end{equation}
The desired estimate follows from $\Delta_\H \ge 0$ and $\ker \Delta_{\H} \cap \Omega^1_{m+4}\H=0$, which is a consequence of e.g.  \eqref{def-p}.
\end{proof}

As this estimate is not sufficiently sharp for some of the numerical eigenvalues in the next section, we provide below a refinement of the estimate
in Lemma \ref{estimate} for $\Delta_{\H}$ acting on $\Omega_{-m}^1\H \cap \ker\di_{\H}^{\star}$.
Writing $C^{\infty}_m M:=C^{\infty}M \cap \ker(\C-m(m+2))$ for $m \in \mathbb{N}$, so that $C^{\infty}_0 M = C^{\infty}_{b} M$, we observe that
\begin{pro}\label{Ho-m} The following hold for $m \in \mathbb{N}$
\begin{itemize}
\item[(i)] the map $\Omega_{-m}^1\H \cap \ker\di_{\H}^{\star} \to C^{\infty}_{m+2}(M,\bbR^{3})$ given by 
$\alpha \mapsto \di_{\H}^{\star}(\bbI\alpha)$ is injective
\item[(ii)] we have 
$\Delta_{\H}>6m+16 \ \mbox{on} \ \Omega_{-m}^1\H \cap \ker\di_{\H}^{\star}.$
\end{itemize}
\end{pro}
\begin{proof}
(i) letting $\alpha \in \Omega^1_{-m}\H \cap \ker \di_{\H}^{\star}$ we have $\di_{\H}^{\star}(\bbI\alpha)_{m+4}=\di_{\H}^{\star}(\bbI\alpha)$
by part (ii) in Lemma \ref{fourier}, since $\di^{\star}_\H$ commutes with $\L_{\xi}$ and $\di^{\star}_\H \alpha = 0$. As $\di^{\star}_\H$ commutes with  $\C$ it follows that  $\di_{\H}^{\star}(\bbI \alpha)_{m+4} \in C^{\infty}_{m+2}(M,\bbR^3)$
showing that the map under consideration is well defined. Assuming that $\di_{\H}^{\star}(\bbI\alpha)=0$ yields 
$\dH \alpha \wedge \omega=\dH(\alpha \wedge \omega)=-\dH \sH \bbI \alpha=\di_{\H}^{\star}(\bbI\alpha) \vol_{\H}=0.$
Equivalently $\star_{\H}\dH\alpha=-\dH\alpha$ which by  \eqref{def-p}  implies that 
\begin{equation*}
\Delta_{\H}\alpha=\di_{\H}^{\star}\dH\!\alpha=\star_{\H}\di_{\H}^2\alpha=-2\!\p(\alpha)=2m\alpha.
\end{equation*}
Hence $\alpha$ has to vanish due to the estimate in Lemma \ref{estimate}.\\
(ii) if $\alpha \in \Omega^1_{-m}\H \cap \ker \di_{\H}^{\star}$ satisfies $\Delta_{\H}\alpha=\lambda \alpha$ we apply $\di_{\H}^{\star}$ in \eqref{Ee2} 
to obtain, after using the commutation formula \eqref{com22}, that $\di_{\H}^{\star}(\bbI\alpha)_{m+4}=\di_{\H}^{\star}(\bbI\alpha) \in 
\ker(\Delta_{\H}-(\lambda-6m-16))$. The claim follows from  $\Delta_{\H}>0$ on $\{f \in C^{\infty}_{m+2}(M,\bbR^3) : \int_Mf\vol=0$\},  
by also taking into account that the map in (i) is injective.
\end{proof}
Arguments within the same circle also show that 
\begin{pro} \label{15+7}
If $g$ does not have constant sectional curvature we have 
$\Delta_{\H}>20$ on $\Omega^1_3\H $.
\end{pro}
\begin{proof} Let $\alpha \in \Omega^1_3\H$ satisfy $\Delta_{\H}\alpha=\lambda \alpha$. Since $\bbI=-\L_{\xi}$ on $\Omega^1_3\H$ using 
\eqref{com22} ensures that 
\begin{equation*}
f:=(\di_{\H}^{\star}\!\alpha, \di_{\H}^{\star}\bbI\alpha) \in C^{\infty}_{1}(M,\bbR^4) \cap \ker(\Delta_{\H}-(\lambda-6)).
\end{equation*}
We now differentiate this eigenvalue equation, w.r.t. $\dH$ and with the aid of the commutator identity in Lemma \ref{van1}.  As $\ker(\C-3) \cap \Omega^1\H=\Omega^1_{-1}\H\oplus \Omega^1_{3}\H$ splitting $\dH\!f=(\dH\!f)_{-1}+(\dH\!f)_3$ thus leads to 
$\Delta_{\H}(\dH\!f)_{-1}=(\lambda-8)(\dH\!f)_{-1}$. If $(\dH\!f)_{-1} \neq 0$ we get $\lambda-8>12$ by Lemma \ref{estimate} hence the claim is proved. If $(\dH\!f)_{-1}=0$, or equivalently $\p(\dH\!f)=3\dH\!f$ applying $\d_{\H}^{\star}$ and taking into account \eqref{codi-p} yields 
$\Delta_{\H}f=4f$. It follows that $\Delta^g f=\Delta_{\H}f+\C f=7f$ which forces $f=0$, by Obata's theorem. Hence $\alpha$ vanishes as well,
by Proposition \ref{Ho-m}, (i) and the claim is proved.
\end{proof}
As we believe some of these results may be of independent interest we have worked here in slightly more generality than strictly needed 
in the next section where only estimates on the weighted spaces $\Omega^1_k \H$ for the weights $k= -1,-2,3$ will be used. 
\section{Infinitesimal Einstein and $\G_2$ deformations} \label{inf-defE}
In this section we will refine the structure results on the spaces $E_{\lambda}$ obtained so far. 
The numerical pairs of relevance in this section  are $(s,\lambda)=(\frac{1}{\sqrt{5}},-\frac{12}{\sqrt{5}})$ respectively $(s,\lambda)=(\frac{1}{\sqrt{5}},\frac{6}{\sqrt{5}})$; recall that the first corresponds to infinitesimal $\G_2$ deformations. According to Propositions \ref{int-2n} and \ref{semi2} pairs $(\alpha,\sigma) \in E_{\lambda}$ then satisfy
$$\tT(\alpha)\in F_{\nu}$$ 
where $\nu=\lambda(\lambda+2s)$. Thus the first priority is to study the spaces $F_{\nu}$ with $\nu$ bounded from above as directed by deformation theory, see section \ref{LL}.

The breakdown of our future strategy is as follows. The Lie derivatives $\mathscr{L}_{\xi_a}$ make the spaces $F_{\nu}$ into $\su(2)$-representations. Decomposing those into irreducible pieces makes it possible to understand in a geometric way the action of $\su(2)$ on $\Omega^{\star}\H$. 
For numerically explicit eigenvalues $\lambda$ we can effectively count which irreducible $\su(2)$-representations (with multiplicities) can occur in $F_{\lambda}$. This is due to the estimate 
\begin{equation} \label{est}
5\C-2\!\p\leq \nu
\end{equation}
on $F_{\nu}=\ker(\mathscr{D}-\nu)$ with $\mathscr{D}=\Delta_{\H}+5\C-2\!\p$, which descends from having $\Delta_{\H} \geq 0$. This observation makes it possible to prove the following key result.
\begin{pro} \label{num1}
Assuming that $\nu \leq 24$ we have 
\begin{equation*}
F_{\nu}=\ker(\Delta_{\H}-\nu) \cap  \Omega^1_{inv}\H.
\end{equation*}
In addition, if $g$ does not have constant sectional curvature $F_{\nu}=0$ for $\nu\leq 16$.
\end{pro}
\begin{proof}
Recall that the (real) irreducible finite dimensional representations of the Lie algebra $\su(2)$ are entirely determined by their 
dimension and come in two series 
\begin{itemize}
\item $U_n$ with $\dim_{\bbR}U_n=2n+1$ where $n \in \mathbb{N}, n \geq 1$
\item $V_{n}$ with $\dim_{\bbR}V_n=4n+4$ where $n \in \mathbb{N}$.
\end{itemize}
Their explicit realisation is not needed here, we only record that the Casimir operator $\C$ acts on 
$U_n$ respectively $V_n$ as $m(m+2)$ with $m=2n$ respectively $m=2n+1$.

Split $F_{\nu}=W_0 \oplus W_1 \oplus \ldots \oplus  W_d$ into isotypical components w.r.t. the $\su(2)$ action, where $W_0$ is the trivial representation. As 
$\p$ preserves $F_{\nu}$ and is $\su(2)$ invariant it follows that $\p(W_i) \subseteq W_i$. Here we haven taken into account that 
$\Hom_{\su(2)}(W_i,W_j)=0$ for $1 \leq i \neq j \leq d$. 
Consequently we need only examine the constraint 
\eqref{est} on $W_i, i \neq 0$ where $\C=m(m+2)$ for some $m \in \mathbb{N}, m \geq 1$. Thus $(\p+m)(\p-m-2)=0$ on $W_i$ by Lemma \ref{pvsC}. Assume that $m \geq 2$; if $-m$ is an eigenvalue for $\p$ the estimate \eqref{est} reads $5m^2+12m \leq \nu \leq 24$ which has no solution. Similarly, assuming $m+2$ is an eigenvalue for $\p$ we get 
$5m^2+8m-4 \leq \nu \leq 24$, a contradiction. We have showed that $m=1$, which allows splitting $W_i=\ker(\p+1) \oplus \ker(\p-3)$.
From the construction 
of $F_{\nu}$ these pieces correspond to the eigenspaces  
$$\ker(\Delta_{\H}-(\nu-17)) \cap \Omega_{-1}^1\H \ \mbox{respectively} \ \ker(\Delta_{\H}-(\nu-9)) \cap \Omega_{3}^1\H$$ which both vanish 
by Lemma \ref{estimate} respectively Proposition \ref{15+7} since $\nu-17\leq 12$ respectively $\nu-9 < 20$. Summarising, 
$\su(2)$ acts trivially on $F_{\nu}$, so $\mathscr{D}=\Delta_{\H}$ on that space. The vanishing of $F_{\nu}$ for $ \nu \leq 16$ is hence granted by the estimate in \eqref{forms-e}.
\end{proof}
This yields a full description of the eigenspace $E_{\lambda}$ for the unstable eigenvalue $\lambda=-2s$.
\begin{coro} \label{-2s}
We have $\ker(\star_s\di+2s)\cap \Omega^3_{27}(\varphi_s)=\bbR\widetilde{\varphi}$. 
\end{coro}
\begin{proof}
A positivity argument shows that $\ker(\Delta_{\H}+5\!\p^2)=0$ on $\Omega^1(\H,\bbR^3)=0$. As $F_0$ vanishes, Proposition \ref{semi2} allows 
concluding that $\ker \mathscr{G}^{\t5}=0$. The claim follows now from the second semi-exact sequence in  Proposition \ref{int-l}.
\end{proof}
\begin{rema} \label{rmk-intro}
Techniques similar to the proof of Proposition \ref{num1} also allow proving the statement from
Remark \ref{rmk6} in the introduction, i.e. showing that any $\Delta^{g_{\s5}}$-eigenfunction for the eigenvalue $2E= 108/5$ is automatically basic.
Indeed, any such eigenfunction $f$ satisfies $\Delta_\H f = (108/5 - 5 \C)f$, and in particular the estimate $\C f \le  108/25 f < 5  f$. 
By $\su(2)$ representation theory we get $f \in C^\infty_{b} M \oplus C^\infty_1 M$ and there remains to exclude the second summand. Assuming $f \in C^\infty_1 M$ we must have $\Delta_\H f = 33/5 f $. 
However, arguments similar to those in section \ref{HoEe} show that $\Delta_\H > 14 > 33/5$ on $C^\infty_1 M \cap \Ker(\Delta_\H - 4)^\perp$. Consequently any eigenfunction for the eigenvalue $2E$ has to be $\su(2)$ invariant, i.e. basic as stated. Note that 
$\ker(\Delta_\H - 4)=0$ if $g$ does not have constant sectional curvature, see \cite{I}.
\end{rema}
At this stage additional insight into the structure of the harmonic form spaces defined in \eqref{h-spc} is required. We consider the bundle map
\begin{equation} \label{s-iso}
\bfs:\Lambda^{-}(\H,\bbR^3) \to \Sym_0^2\H, \ \ \bfs(\sigma):=\sum_a \sigma_a^{\sharp} \circ I_a
\end{equation}
where the skew endomorphisms $\sigma_a^{\sharp}$ acting on $\H$ satisfy $g_{\H}(\sigma_a^{\sharp} \cdot, \cdot)=\sigma_a$. This is an isomorphism since it is an injective map between spaces of the same dimension. Furthermore, let $\Gamma_{b}(\Sym_0^2\H)$ be the space of basic trace free symmetric tensors on $M$, in other words the space of 
$\su(2)$-invariant sections of 
$\Sym_0^2\H$. Basic TT tensors are then defined according to 
\begin{equation*}
\TT_b(\H):=\Gamma_{b}(\Sym_0^2\H) \cap \ker \delta^{g_s}.
\end{equation*}
As $\V$ is totally geodesic w.r.t. any of the metrics $g_s,s>0$ this definition does not depend on the choice of the parameter $s$. 
\begin{pro} \label{harm1}
The spaces $\bfH_0^{+},\bfH^{\pm}_1$ and $\bfH^{\pm}_3$ vanish. In addition the map 
\begin{equation*}
\bfs: \bfH_4^{-} \to \ker (\Delta_L^{g_s}-\tfrac{76}{5}) \cap \TT_b(\H)
\end{equation*}
is injective.
\end{pro}
\begin{proof}
Let $\sigma \in \Omega^2(\H,\bbR^3)$ satisfy $C\sigma=\lambda \sigma$ and $\L^{\star}_{\xi}\sigma=0$. Thus $\C\sigma=\lambda(\lambda-2)\sigma$  by \eqref{sqC}. When $\lambda=0$ it follows that $\sigma$ is $\su(2)$-invariant. Under the additional requirement 
that $\sigma \in \Omega^{+}_{sym}(\H,\bbR^3)$ this leads to $\sigma=0$ after a short argument based on \eqref{Lie-os}, 
hence $\bfH^{+}_0=0$. \\
Since the operator $\C$ is non-negative we get that $\sigma=0$ for $\lambda=1$, so $\bfH^{\pm}_1=0$. Further on the algebraic constraints on $\sigma$ lead to $\C_{\rho \otimes \pi^1}\sigma=(\lambda-4)(\lambda-2)\sigma$ according to \eqref{cas01}. Since $\C_{\rho \otimes \pi^1}$ is non-negative $\sigma=0$ for $\lambda=3$ and $\sigma$ is $\su(2)$-invariant w.r.t. the tensor product representation 
when $\lambda=4$. In expanded form this reads 
$\L_{\xi_a}\sigma_a=0, \ \L_{\xi_a}\sigma_b=-\L_{\xi_b}\sigma_a=2\sigma_c$
with cyclic permutations on $abc$. Equivalently, the tensor $\bfs(\sigma)$ is $\su(2)$-invariant by \eqref{Lie-osI}, hence basic.

In order to finish the argument we argue as follows. A glance at the proof of Proposition \ref{int-l} reveals that $\kappa_s(0,\sigma)$ belongs to $E_{-\frac{2}{s}}$.  Next, according 
to part (i) in the purely algebraic Lemma \ref{L80} proved in the next section we have $\bfs(\sigma)=2\bfi^{-1}\kappa_s(0,\sigma)$. Now use Proposition \ref{lich} with $\tau=\frac{76}{5}$ (and hence $\lambda^{-}=-\frac{2}{s}$) to see that $\bfi^{-1}(E_{-\frac{2}{s}}) \subseteq \ker (\Delta_L^{g_s}-\frac{76}{5}) \cap \TT(g_s)$ hence
$\bfs(\sigma)$ lives in the latter space. As we have already checked that $\bfs(\sigma)$ is basic, it follows that it actually belongs to $\TT_b(\H)$ and the claim is proved.
\end{proof}
The proof of Theorem \ref{main2} given in section \ref{proofs1} will show that $\bfs$ is actually an isomorphism between the spaces featuring in the proposition above; thus we obtain a characterisation of the space of equivariant harmonic forms $\mathbf{H}_4^{-}$ as the unique basic eigenspace of the 
Lichnerowicz Laplacian acting on TT tensors.
\begin{rema} \label{tw}
Denoting with $\mathcal{F}_1$ the foliation tangent to $\spa\{\xi_1\}$ consider the twistor space $Z:=M\slash \mathcal{F}_1$ which is a compact K\"ahler orbifold (see \cite{BoGa}). Its complex structure is the projection of $J\xi_2:=\xi_3, J_{\vert \H}:=I_1$ onto $Z$. We have a natural embedding $\bfH^{-}_4 \to H^{1,0}(Z,T^{0,1}Z\otimes \mathbf{L})$ 
coming from the projection of $\sigma \mapsto \sigma_2(I_2 \cdot, \cdot)+i\sigma_3(I_3 \cdot, \cdot)$ onto 
$\Omega^{1,0}(Z,T^{01}Z \otimes \mathbf{L})$, where $\mathbf{L}=\mathbf{K}_Z^{-\frac{1}{2}}$. This suggests that the algebraic geometry 
of $(Z,J)$, rather than the spectral theory of $\Delta_L^{g_s}$, could alternatively be used to describe $\mathbf{H}_4^{-}$.
\end{rema}
Combining the representation theory arguments used in the proof of Proposition \ref{num1} with the eigenvalue estimates in section \ref{HoEe} leads to the following structure result.
\begin{pro} \label{ker-gen}
Assume that $0 \neq \nu:=\lambda(\lambda+2s) \leq 24$ and that $g$ does not have constant sectional curvature. The map 
$$E_{\lambda}=\ker(\star_{s}\di-\lambda) \cap \Omega^3_{27}(\varphi_{s}) \stackrel{\tT \circ \pr_1}{\to}\ F_{\nu}=\ker(\Delta_{\H}-\nu) \cap \Omega^1_{inv}\H
$$
is injective for $\nu>16$. 
If $\nu=16$ we have  $s\lambda=-2$ and $E_{\lambda}
=\kappa_s(0,\bfH^{-}_{4})$. In addition the space $E_{\lambda}$ vanishes when $\nu<16$.
\end{pro}
\begin{proof}
Pick $\gamma=\kappa_s(\alpha,\sigma) \in \ker(\star_s\di-\lambda) \cap \Omega^3_{27}M$ such that $\tT(\alpha)=0$.
Combining Proposition \ref{int-l} and \ref{semi2} shows that $\L_{\xi}^{\star}\alpha \in F_{\nu}$; as this is contained 
in $\Omega^1_{inv}\H$ by Proposition \ref{num1} and $\L_{\xi}^{\star}\alpha$ is $L^2$-orthogonal to $\Omega^1_{inv}\H$ it follows that 
 $\L_{\xi}^{\star}\alpha=0$. Thus $\alpha \in \ker(\L_{\xi}^{\star}\oplus \tT)$ hence further 
$$\alpha \in \ker(\Delta_{\H}+5\p^2-\nu) \cap \Omega^1_{o}(\H,\bbR^3)$$
by Proposition \ref{semi2}. Consider the finite dimensional, $\su(2)$-invariant space $\ker(\Delta_{\H}+5\!\p^2-\nu)$. 
From the estimate $5\!\p^2\leq \nu \leq 24$ on this space, by arguments entirely similar to Proposition \ref{num1}
$$\ker(\Delta_{\H}+5\!\p^2-\nu)\cap \Omega^1\H=\mathscr{E}(\nu,0)\oplus  \mathscr{E}(\nu-5,-1) \oplus \mathscr{E}(\nu-20,-2) .$$
Here 
$$\mathscr{E}(t,k):=\ker(\Delta_{\H}-t) \cap \Omega^1_{k}\H$$ 
for $(t,k) \in \bbR \times \mathbb{Z}$ in shorthand notation. This allows splitting $\alpha=\alpha_0+\alpha_{-1}+\alpha_{-2}$ where 
$$ \alpha_0 \in \mathscr{E}(\nu,0)\otimes \bbR^3, \ \alpha_{-1} \in  \mathscr{E}(\nu-5,-1)\otimes \bbR^3, \ \alpha_{-2} \in  \mathscr{E}(\nu-20,-2)\otimes \bbR^3.
$$
Next we argue that $\alpha$ is coclosed. Indeed, since $\lambda+2s\neq 0$, the last equation in \eqref{sys1} shows that $\tT(\sigma)=0$.
Since $\lambda \neq 0$ we know that $\di^{\star_s}\gamma=0$. The projection of this onto $\Omega^0(\H,\bbR^3)$ then yields 
$\di_{\H}^{\star}\alpha=0$
according to Lemma \ref{cod-gen}. In expanded form 
$$\di_{\H}^{\star}\alpha_0+\di_{\H}^{\star}\alpha_{-1}+\di_{\H}^{\star}\alpha_{-2}=0.$$
Because $\C$ commutes with $\di_{\H}^{\star}$ and $\C\alpha_0=0, \C\alpha_{-1}=3\alpha_{-1}, \C\alpha_{-2}=8\alpha_{-2}$
the latter equation leads, after succesive application of $\C$ respectively $\C^2$ as well as solving the corresponding Vandermonde system, to $\di_{\H}^{\star}\alpha_0=\di_{\H}^{\star}\alpha_{-1}=\di_{\H}^{\star}\alpha_{-2}=0$. Since $\nu-5\leq 19 <22$ and $\nu-20\leq 4<28$ (as $\nu\leq 24$ by assumption), the eigenvalue estimate in Proposition \ref{Ho-m},(ii) for $m=1,2$ leads to 
$\alpha_{-1}=\alpha_{-2}=0.$ In other words 
$$ \alpha \in \Omega^1_{inv}(\H,\bbR^3) \cap \Omega^1_{o}(\H,\bbR^3).$$  
The latter space vanishes as it can be checked using the identity \eqref{up-C-temp}, hence $\alpha=0$.\\
By Proposition \ref{int-l} the form $\sigma$ belongs then to $\bfH^{-}_{2-s\lambda} \oplus \bfH^{+}_{2+s\lambda}$. As these spaces vanish 
when $s\lambda \notin \mathbb{Z}$ (see \eqref{van11}) in order to prove the claim there only remains to examine instances with $s\lambda=n$ with $n \in \mathbb{Z}^{\times}$. The bound on $\lambda$ in the assumptions reads $n(n+\frac{2}{5\sqrt{5}}) \leq \frac{24}{5}=4.8$ forcing 
$n \in \{-2,-1,1\}$, thus $\nu=n(5n+2) \in \{16,3,7\}$. This proves injectivity for $\tT \circ \pr_1$ when $\nu>16$. For $\nu\leq 16$ 
the target space $F_{\nu}$ of the map $\tT \circ \pr_1$ vanishes by Proposition \ref{num1} hence 
$E_{\lambda}=\bfH^{-}_{2-n} \oplus \bfH^{+}_{2+n}$
for $s\lambda \in \{-2,-1,1\}$ or $E_{\lambda}=0$ otherwise.
The claim follows from the vanishing results in Proposition \ref{harm1}.
\end{proof}
We can now fully describe the eigenspaces $E_{\lambda}$ with $16<\lambda(\lambda+2s)\leq 24$ in terms of eigenspaces of the basic Laplacian.
\begin{teo} \label{solve1}
Assume that $g$ does not have constant sectional curvature and that $\lambda$ satisfies $16< \nu=\lambda(\lambda+2s) \leq 24$. The maps 
$$\varepsilon^{\pm}_{\nu}:\ker(\Delta_{b}-\nu) \to \ker(\star_s\di-\lambda_{\pm}) \cap \Omega^3_{27}(\varphi_s)$$ 
from Proposition \ref{embed1} are linear isomorphisms.
\end{teo}
\begin{proof}
As the maps $\varepsilon_{\nu}^{\pm}$ are clearly injective there remains to prove their surjectivity.
Pick $\gamma=\kappa_s(\alpha,\sigma) \in E_{\lambda}$ 
and proceed as follows. First we show that $\tT(\alpha)$ is $\dH$-exact. Indeed, combining Propositions \ref{int-l} and \ref{semi2} shows that 
$$(\L^{\star}_{\xi}\alpha,\tT(\alpha))
\in F_{\nu}^{\perp} \oplus F_{\nu}.$$
As $\nu \leq 24$ the space $F_{\nu}$ consists of $\su(2)$-invariant forms by Proposition 
\ref{num1} hence $\L^{\star}_{\xi}\alpha=0$. Consequently the first and last equations in \eqref{sys1} update to $\di_{\H}^{\star}\tT(\alpha)=(\lambda+2s)\tT(\sigma)$ and 
$\lambda\tT(\alpha)=\dH\!\tT(\sigma)$. Put together, these equations ensure $\su(2)$-invariance for  
$\tT(\sigma)$ and allow 
writing $\tT(\alpha)=-\dH\!f$ with $f=-\frac{1}{\lambda}\tT(\sigma) \in \ker(\Delta_{b}-\nu)$.\\
Next we show that $\gamma$ is fully determined by $f$.
The form $\gamma-\varepsilon_{\nu}^{\pm}(f)=\kappa_s(\beta,\rho)$ belongs to $E_{\lambda}$ and satisfies $\tT(\beta)=0$ since $\beta=\alpha-\tfrac{1}{3}\bbI\di\!f$. 
Hence the pair $(\beta,\rho)$ vanishes by Proposition \ref{ker-gen}.
In other words  $\gamma=\varepsilon^{\pm}_{\nu}(f)$ 
and the claim is proved.
\end{proof}
Following the same line of reasoning, with slightly different numerics based this time on Proposition \ref{int-2n}, 
we can also deal with eigenspaces of the Laplacian on closed forms, where we prove the following vanishing result.
\begin{teo} \label{thmII}
The space 
$\{\gamma \in \Omega^3_{27}(\varphi_s) : \di\!\di^{\star_s}\!\gamma=\mu \gamma\}$ where $0 \neq \mu \leq 72s^2$ vanishes.

\end{teo} 
\begin{proof}
Let $\gamma=\kappa_s(\alpha,\sigma)$ belong to the space above.
Combining Propositions \ref{int-2n} and \ref{semi2} shows that 
$(\L^{\star}_{\xi}\alpha,\tT(\alpha))
\in F_{\mu}^{\perp} \oplus F_{\mu}.$
Since $\mu \leq \frac{72}{5}<16$ we obtain, by using Proposition \ref{num1}, that $F_{\mu}=0$. In other words 
$\alpha \in \ker(\L_{\xi}^{\star} \oplus \tT)$ which yields  
$\alpha \in  \ker(\Delta_{\H}+5\p^2-\mu) \cap \Omega^1_{o}(\H,\bbR^3)$ by means of Proposition \ref{semi2}. 
As in the proof of Proposition \ref{ker-gen} the estimate $5\p^2\leq \mu \leq \frac{72}{5}$ on the latter space first shows that 
$$\ker(\Delta_{\H}+5\p^2-\mu) \cap \Omega^1(\H,\bbR^3)=\mathcal{E}(\mu,0) \otimes \bbR^3 \oplus  \mathcal{E}(\mu-5,-1) \otimes \bbR^3$$
where we use the same notation as in the proof of Proposition \ref{ker-gen}. Because $\mu-5\leq \frac{47}{5}<12$ the last component space vanishes by the eigenvalue estimate in Lemma \ref{estimate}. 
By Proposition \ref{num1} we get $\mathcal{E}(\mu,0)=0$ since $\mu<16$ hence we have showed that $\alpha=0$. It follows,
by Proposition \ref{int-2n}, that $\sigma=\L_{\xi}\sigma_0$ with 
$\sigma_0 \in \mathbb{H}_{\tfrac{\mu}{5}}$. By assumption $\tfrac{\mu}{5}\leq \tfrac{72}{25}=2.88$. As the Casimir operator 
$\C$ of the induced $\su(2)$ representation can have only integer eigenvalues, of the form $m(m+2), m \in \mathbb{N}$, it follows that $\mathbb{H}_{\tfrac{\mu}{5}}=0$, thus $\sigma=0$ and the claim is fully proved.
\end{proof}
\subsection{Proofs of Theorem \ref{main1} (i), (ii) and of Theorem \ref{main2} } \label{proofs1}
%
%
Proving these claims amounts to describing $\ker(\Delta_L-\tau)$ with $\tau \leq 2E_s=108s^2$ and $s= \t5$. Based on Proposition \ref{lich} there are three cases to consider
corresponding to the three summands in $\ker(\Delta_L^{g_s}-\tau)$. We will systematically use the relation between eigenvalues $\tau$ for $\Delta_L$ and eigenvalues $\lambda^{\pm}$ for $\star_s\di$ respectively eigenvalues $\mu$ for $\di\di^{\star_s}$ given in that proposition.
\begin{enumerate}
\item[(a)] $E_{\lambda^{+}}$ with $\lambda^{+}(\lambda^{+}+2s) \leq \tfrac{48}{5}$. \\
As $\tfrac{48}{5}<16$ we get $E_{\lambda^{+}}=0$ by Proposition \ref{ker-gen}, 
provided that $\lambda^{+}\neq -2s$. When  $\lambda^{+}=-2s$ we have $E_{\lambda^{+}}=\bbR\widetilde{\varphi}$ by Corollary \ref{-2s} with 
Lichnerowicz eigenvalue $\tau=28s^2$. \\
\item[(b)] $E_{\lambda^{-}}$ with $\lambda^{-}(\lambda^{-}+2s) \leq 24$. \\
If  $\lambda^{-}(\lambda^{-}+2s)<16$ then $E_{\lambda^{-}}=0$ by Proposition \ref{ker-gen}, 
since $\lambda^{-}= - 2s$ cannot occur, as $\lambda^- = -3s-\sqrt{\tau-27s^2}$. By the same proposition, having $\lambda^{-}(\lambda^{-}+2s)=16$ corresponds to 
$\lambda^{-}=-\frac{2}{s}$ and $E_{\lambda^{-}}=\bfH^{-}_4$ with $\tau=76s^2$. In the last remaining case we have $16 <\nu=\lambda^{-}(\lambda^{-}+2s) \leq 24$.
Expressing $\lambda^{-}=-s-\sqrt{\nu+s^2}$ in terms of $\nu$ and noting that $\lambda^- <0$ we see that Theorem \ref{solve1} provides a linear isomorphism
$\varepsilon^{-}_{\nu} : \ker(\Delta_{b}-\nu) \to E_{\lambda^{-}}$. In this case the eigenvalue $\tau$ for 
$\Delta_L^{g_s}$ reads 
$\tau = \nu - 4s\sqrt{\nu +s^2} +32s^2$.
 \\
\item[(c)]$\{\gamma \in \Omega^3_{27}(\varphi_s) : \di\!\di^{\star_s}\!\gamma=\mu \gamma\}$ with $\mu \leq 72s^2$. \\
As we know that $\mu \neq 0$ this space has to vanish by Theorem \ref{thmII}.
\end{enumerate}
Summarising,  the space  of infinitesimal Einstein deformations, for $\tau = 108s^2$, coincides with the space  $E_{\lambda^{-}}= E_{-12s}$ of infinitesimal $\G_2$ 
deformations which in turn  is isomorphic to the eigenspace $\ker(\Delta_{b}-24)$ via $\varepsilon=\varepsilon_{24}^{-}$. This proves
Theorem \ref{main1},(i),(ii). Moreover, the space of unstable directions has the components $\bbR\widetilde{\varphi},\, \bfH^{-}_4$ and 
$\ker(\Delta_{b}-\nu) $ with $16 < \nu < 24$. The corresponding eigenvalues $\tau$ are given above thus proving Theorem \ref{main2}.

\section{Computation of the obstruction polynomial} \label{obs}
The aim in this section is to calculate, on the space $\mathcal{E}(\varphi_s), s=1/\sqrt{5}$ of infinitesimal $\G_2$ deformations, the obstruction to integrability polynomial $\mathbf{K}:
\mathcal{E}(\varphi_s) \to \Lambda^1\mathcal{E}(\varphi_s)$ as introduced in our previous work \cite{NS} according to which we first need to examine the following algebraic invariants. 
\begin{itemize}
\item the symmetric bilinear form  $p:\Lambda^3_{27}(\varphi_s) \times \Lambda^3_{27}(\varphi_s) \rightarrow \Sym^2(TM,g_s)$ determined  from\\ $
p(\gamma, \gamma)(U, V) = g_s( U \lrcorner  \gamma, V \lrcorner \gamma)
$
\item the linear isomorphism  $\bfi^{-1} : \Lambda^3_{27}(\varphi_s) \rightarrow  \Sym_0^2(TM,g_s)$ as defined in section \ref{2ndE}
\item the trilinear map $P(\gamma_1,\gamma_2,\gamma_3) :=\la p(\gamma_1, \gamma_2), \bfi^{-1}\gamma_3 \ra $ with $\gamma_k \in \Lambda^3_{27}(\varphi_s), k=1,2,3$ where the scalar product on $\Sym^2(TM,g_s)$ is given by $\la S_1,S_2 \ra=\tr(S_1 \circ S_2)$.
\end{itemize}
Note that the map $\mathbf{P}:\Lambda^3_{27}(\varphi_s) \times \Lambda^3_{27}(\varphi_s) \to \Lambda^3_{27}(\varphi_s)$ which appears in section \ref{backgr-def} is then the metric dual of $P$, that is $g_s(\mathbf{P}(\gamma_1,\gamma_2),\gamma_3)=P(\gamma_1,\gamma_2,\gamma_3).$

Since $\mathcal{E}(\varphi_s)=\varepsilon (\ker(\Delta_b-24))$ we explicitly have 
$$\mathbf{K}(\varepsilon(f))\varepsilon(h)=\int_M P(\varepsilon(f),\varepsilon(f),\varepsilon(h))\vol$$
and the set of infinitesimal $\G_2$ deformations which are unobstructed to second order is given by the zero locus 
$\mathbf{K}^{-1}(0)$, by \cite[Thm.1.1]{NS}. 

To carry out the programme of computing $\mathbf{K}$ let $f \in \ker(\Delta_b-24)$ and split, according 
to \eqref{e-bis},
\begin{equation*}
\begin{split}
\varepsilon(f)=\,& \kappa_s(\bbI\dH\!f,t_1f\omega+t_2\di_{\H}^{-}\bbI\dH\!f)
=t_1f \widetilde{\varphi}+t_2\gamma_1+\gamma_2
\end{split}
\end{equation*}
where the factor $1/3$ has been dropped for convenience, with the constants $t_1=6s$ and $t_2=1/2s$. Here we recall that $\widetilde{\varphi}_s=\kappa_s(0,\omega)$ and use the notation 
$$ \gamma_1=\,\kappa_s(0, \di_{\H}^{-}\bbI\dH\!f) , \ \gamma_2=\kappa_s(\bbI\dH\!f,0).
$$
In the following computations we will frequently use that $\bbI\dH\!f=X \lrcorner \omega$ 
where $X:=\grad f$ together with the expanded algebraic expression 
$ \gamma_2=\mathfrak{S}_{abc}Z^{ab} \wedge (X \lrcorner \omega_c)-3X \lrcorner \vol_{\H}.
$
These observations on the algebraic structure of $\varepsilon(f)$ show that we only need determine $p$ and $\bfi^{-1}$ on the subbundle 
$$\kappa_s(\Lambda^1\H \oplus \spa \{\omega\} \oplus \Lambda^-(\H,\bbR^3)) \subseteq \Lambda^3_{27}(\varphi_s)$$ where 
$\Lambda^1\H$ is embedded into $\Lambda^1(\H,\bbR^3)$ via $\alpha \mapsto \bbI \alpha$. To determine the action of $\bfi^{-1}$ on this subbundle we mainly rely on the algebraic isomorphism $\bfs:\Lambda^{-}(\H,\bbR^3) \to \Sym_0^2\H$ defined in section \ref{inf-defE}.
\begin{lema} \label{L80}
Assume that $\gamma_1=\kappa_s(0,\sigma)$ with $\sigma \in \Lambda^{-}(\H,\bbR^3)$ and that $\gamma_2=\kappa_s(X \lrcorner \omega,0)$ with 
$X \in \H$. We have
\begin{itemize}
\item[(i)]
$ 
\bfi^{-1}\gamma_1= \tfrac12 \bfs(\sigma)
$ \vspace{1mm}
\item[(ii)]
$
\mathbf{i}^{-1}\gamma_2= \sum_a Z^a \otimes I_a X + (I_a X)^{\flat} \otimes Z_a
$\vspace{1mm}
\item[(iii)] $\bfi^{-1}\widetilde{\varphi}_s=-\frac{1}{2}(4\,\id_{\V}-3\,\id_{\H})$.
\end{itemize}
\end{lema}
\begin{proof}
All claims follow directly from the general formula $\bfi(S)=\sum Se^i \wedge e_i \lrcorner \varphi_s$ whenever $S \in \Sym^2_0(M,g_s)$, where $\{e_i, 1 \leq i \leq 7\}$ is some local orthonormal basis in $TM$. This follows from the action of $\bfi$ on decomposable elements outlined in section \ref{2ndE}.\\
(i) Let $S:=\bfs(\sigma)$ and consider a $g_{\H}$ orthonormal basis $\{x_i,1 \leq i \leq 4\}$ in $\H$. As $S$ only acts on $\H$ we have 
\begin{equation*}
\bfi(S)=\sum_i Sx^i \wedge x_i \lrcorner \varphi_s=\sum_a Z^a \wedge \sum_i Sx^i \wedge (x_i \lrcorner \omega_a).
\end{equation*}
At the same time direct calculation shows that $\sum_i Sx^i \wedge (x_i \lrcorner \omega_a)=-g_{\H}(SI_a+I_aS \cdot, \cdot)$. 
The definition of $\bfs$ entails 
$SI_a+SI_a=-2\sigma_a^{\sharp}$, since the endomorphisms $\sigma_a^{\sharp}$ commute with $I_1,I_2,I_3$. Gathering these facts yields 
$\bfi(S)=\sum_a Z^a \wedge \sigma_a=\kappa_a(0,\sigma)$ whence the claim.\\
\noindent
(ii) With 
$S=\sum_a Z^a \otimes I_a X + (I_a X)^{\flat} \otimes Z_a$ we have $SZ_a=I_aX$ and $Sx=\sum_a g(I_aX,x)Z_a$ whenever $x \in \H$ therefore $\bfi(S)=\sum_a (I_aX)^{\flat} \wedge Z_a \lrcorner \varphi_s+Z^a 
\wedge (I_aX \lrcorner \varphi_s)$. After using the quaternionic relations $I_1I_2=-I_2I_1=I_3$ we find that the second summand equals 
$$\sum_a Z^a \wedge (I_aX \lrcorner \varphi_s)=-\sum_{a,b}Z^{ab} \wedge (I_aX \lrcorner \o_b)=2\mathfrak{S}_{abc}Z^{ab} \wedge (I_cX)^{\flat}.$$ 
To conclude we take into account that $Z_a \lrcorner \varphi_s=Z^{bc}+\o_a$ and also that $(I_aX)^{\flat} \wedge \o_a=-X \lrcorner \vol_{\H}$, according to the convention in \eqref{met}.\\
(iii) Letting $S=4\id_{\V}-3\id_{\H}$ we have $\bfi(S)=4\sum_a Z^a \wedge (Z_a \lrcorner \varphi_s)-3 \sum_i
x^i \wedge (x_i \lrcorner \varphi_s)$ for some $g_{\H}$ orthonormal basis $\{x_i\}$ in $\H$. As $\varphi_s$ is a $3$-form $\sum_a Z^a \wedge (Z_a \lrcorner \varphi_s)+ \sum_i
x^i \wedge (x_i \lrcorner \varphi_s)=3\varphi_s$ hence we further obtain $\bfi(S)=7\sum_a Z^a \wedge (Z_a \lrcorner \varphi_s)-9\varphi_s=-2(\varphi_s-7Z^{123})$ after observing that $\sum_a Z^a \wedge (Z_a \lrcorner \varphi_s)=2Z^{123}+\varphi_s$. The claim follows from $\widetilde{\varphi}_s=\varphi_s-7Z^{123}$, see page 13.
\end{proof}

Next we calculate the necessary components in $p$. 
\begin{lema} \label{L81}
Assume that $\gamma_1=\kappa_s(0,\sigma)$ with $\sigma \in \Lambda^{-}(\H,\bbR^3)$ and that $\gamma_2=\kappa_s(X \lrcorner \omega,0)$ with 
$X \in \H$. We have 
\begin{itemize}
\item[(i)] 
$ p(\gamma_1, \gamma_1)=\frac{1}{2}\sum_{a,b} g(\sigma_a, \sigma_b)  (Z^a \otimes Z_b+Z^b \otimes Z_a)+\frac{1}{2}\vert \sigma \vert^2\id_{\H}$
\item[(ii)] 
$
   p(\gamma_2, \gamma_2) =  2  |X|^2 \, \Id_\V + 10|X|^2 \, \Id_\H - 10X\otimes X 
$
\item[(iii)] $p(\widetilde{\varphi}_s,\widetilde{\varphi}_s)=38\id_{\V}+3\id_{\H}$
\item[(iv)] $p(\gamma_1,\widetilde{\varphi}_s)=\bfs(\sigma)$.
\end{itemize}
\end{lema}
\begin{proof}
(i) Recall that $\gamma_1=\iota_s(0,0,\sigma,0)=\sum_a Z^a \wedge \sigma_a$. Since $\sigma$ is horizontal we have 
$Z_a \lrcorner \gamma_1=\sigma_a$ and $x \lrcorner \gamma_1=-\sum_a Z^a \wedge (x \lrcorner \sigma_a)$ when 
$x \in \H$. The claim follows by a routine computation essentially based on the identity $\vert x \lrcorner \sigma \vert^2=\frac{1}{2}\vert x \vert^2\vert \sigma \vert^2$ with $x \in \H$ which is due to having $\sigma \in \Lambda^{-}(\H,\bbR^3)$.\\
(ii) writing $\alpha=X \lrcorner \omega \in \Lambda^1(\H,\bbR^3)$ we have 
\begin{equation*}
Z_a \lrcorner \gamma_2=Z^b \wedge \alpha_c-Z^c \wedge \alpha_b, \ \ x \lrcorner \gamma_2=\mathfrak{S}_{abc}\alpha_c(x)Z^{ab}-3x \lrcorner X \lrcorner \vol_{\H}
\end{equation*}  
with cyclic permutations on $abc$ and where $x\in\H$. Taking scalar products and using orthogonality w.r.t. $g_s$ of the factors 
in $\Lambda^2M=\Lambda^2\V\oplus (\Lambda^1\V \wedge \Lambda^1\H) \oplus \Lambda^2\H$ shows that  
\begin{equation*}
\begin{split}
& g_s(Z_a \lrcorner \gamma_2,Z_b \lrcorner \gamma_2)=2\vert X \vert^2 \delta_{ab}, \ \ g_s(Z_a \lrcorner \gamma_2,x \lrcorner \gamma_2)=0\\
& g_s(x \lrcorner \gamma_2, x \lrcorner \gamma_2)=(\sum_a \alpha_a \otimes \alpha_a)(x,x)+9\vert x \lrcorner X \lrcorner \vol_{\H}\vert^2.
\end{split}
\end{equation*}
The claim follows from the algebraic identities 
$$(\sum_a \alpha_a \otimes \alpha_a)(x,x)=\vert x \lrcorner X \lrcorner \vol_{\H}\vert^2=
\vert x\vert^2\vert X\vert^2-g(x,X)^2.$$
(iii)$\&$(iv) are proved by direct calculation, in an entirely analogous way as in (i) and (ii). Details are omitted.
\end{proof}

\medskip

Returning to the computation of $P(\varepsilon(f),\varepsilon(f),\varepsilon(f))$ we recall the following. In \cite[Remark 2.3]{NS}, we have showed that the trilinear map $P$
is totally symmetric on $ \Lambda^3_{27}(\varphi_s)$ , i.e. it is an element of $\Sym^3\Lambda^3_{27}$.
We let $\eta:=t_1f\widetilde{\varphi}_s+t_2\gamma_1$ and record that the symmetric endomorphisms $p(\eta,\eta)$ 
and $p(\gamma_2, \gamma_2)$ belong to $(\Lambda^1\H \otimes \H) \oplus (\Lambda^1\V \otimes \V)$ by Lemma \ref{L81}.
Hence both are orthogonal to $\bfi^{-1}\gamma_2$ which lives in 
$(\Lambda^1\V  \otimes \H)\oplus  (\Lambda^1\H \otimes \V )$. The symmetry of $P$ thus entails 
\begin{equation} \label{S81}
\begin{split}
P(\varepsilon(f),\varepsilon(f),\varepsilon(f))=& \la  p(\varepsilon(f), \varepsilon(f)), \bfi^{-1}\varepsilon(f) \ra  = \la  p(\eta, \eta), \bfi^{-1}\eta \ra + 3  \la  p(\gamma_2, \gamma_2), \bfi^{-1}\eta \ra\\
=&P(\eta,\eta,\eta)+3P(\gamma_2,\gamma_2,\eta).
\end{split}
\end{equation}
Further on, the remaining two summands in $P(\varepsilon(f),\varepsilon(f),\varepsilon(f))$ are determined as follows.
\begin{lema} \label{L8-3}
For $\eta$ and  $\gamma_2$ as above we have
\begin{itemize}
\item[(i)] 
$ 
\la  p(\eta, \eta), \bfi^{-1}\eta \ra =-210(t_1f)^3 +3(t_1f)t_2^2 \vert \di_{H}^{-} \bbI \dH f \vert^2
$ \vspace{1.5mm}
\item[(ii)] 
$
\la p(\gamma_2, \gamma_2), \bfi^{-1}\eta \ra  = 33t_1 f \, | \dH f|^2 -5t_2\sum_a g(\dH f \wedge I_a \dH f, \di_{\H}^{-}  I_a \dH f).
$
\end{itemize}
\end{lema}
\begin{proof}
We essentially apply Lemmas \ref{L80} and \ref{L81} with $\sigma=\di_{\H}^{-}\bbI \dH\!f$ and $X=\grad f$.\\ 
(i) since the tensor $P$ is totally symmetric, expansion yields 
\begin{equation*}
P(\eta,\eta,\eta)=(t_1f)^3P(\widetilde{\varphi}_s,\widetilde{\varphi}_s,\widetilde{\varphi}_s)+
3(t_1f)^2t_2P(\widetilde{\varphi}_s,\widetilde{\varphi}_s,\gamma_1)+3(t_1f)t_2^2P(\gamma_1,\gamma_1,\widetilde{\varphi}_s)+
t_2^3P(\gamma_1,\gamma_1,\gamma_1).
\end{equation*}
By combining Lemmas \ref{L80} and \ref{L81} we see that 
\begin{equation*}
\begin{split}
&P(\widetilde{\varphi}_s,\widetilde{\varphi}_s,\gamma_1)=P(\gamma_1,\gamma_1,\gamma_1)=0,\ \ \ P(\widetilde{\varphi}_s,\widetilde{\varphi}_s,\widetilde{\varphi}_s)=-210, \ P(\gamma_1,\gamma_1,\widetilde{\varphi}_s)=\vert \sigma \vert^2
\end{split}
\end{equation*} 
and the claim follows.\\
(ii) using again Lemma \ref{L80} and Lemma \ref{L81} for the explicit expression for $p(\gamma_2,\gamma_2)$ we find 
\begin{equation*}
\begin{split}
&P(\gamma_2,\gamma_2,\widetilde{\varphi}_s)=\la p(\gamma_2, \gamma_2), \bfi^{-1}\widetilde{\varphi}\ra=33\vert X \vert^2\\
&P(\gamma_2,\gamma_2,\gamma_1)=\la p(\gamma_2, \gamma_2), \bfi^{-1}\gamma_1 \ra=-5\bfs(\sigma)(X,X)
\end{split}
\end{equation*}
since $\bfs(\sigma)$ only acts on $\H$ and is trace free. As $\bfs(\sigma)(X,X)=\sum_a \la \di_{\H}^{-}I_a\dH\!f, \dH\!f \wedge I_a \dH\!f\ra$ directly from the definitions, the claim is proved by gathering terms.
\end{proof}
\medskip

\subsection{Integral invariants} \label{int-i}
The algebraic computation in Lemma \ref{L8-3} singles out the three types of integral quantities which need to be computed in order to fully 
determine the obstruction to integrability map $\mathbb{K}$.
\begin{lema}\label{int} Assuming that $f\in \ker(\Delta_{\H}-\nu) \cap C^{\infty}_{b}M$ the following hold
\begin{itemize}
\item[(i)] 
$\int_M  f \vert \dH f \vert^2  \,\vol  = \frac{\nu}{2}\int_M  f^3\,\vol $ \vspace{1.5mm}
\item[(ii)] 
$\int_M \sum_a g(\dH\!f \wedge I_a\dH\!f, \dH\!I_a\dH\!  f)\vol= 0  $ \vspace{1mm}
\item[(iii)] 
$\int_M f | \dH\!\bbI\dH\! f|^2 \, \vol=   \tfrac{(\nu-8)\nu}{2} \int_M f^3  \, \vol $.
\end{itemize}
\end{lema}
\begin{proof}
(i) we have $\int_M f \vert \di\! f \vert^2\vol=\tfrac{1}{2} \int_M \la  \di\!f^2,\di\!  f \ra \,\vol  =\frac{1}{2} \int_M  f^2 \Delta^g  f  \,\vol=
\frac{\nu}{2} \int_M f^3  \,\vol  $. \\
(ii) consider the horizontal vector field 
$X:=\grad\, f$ and observe that 
\begin{equation*}
\begin{split}
\sum_a  g(\dH\!f \wedge I_a\dH\!f, \dH\!I_a\dH\!  f)=&-  \sum_a \dH\!I_a\dH\!  f (X,I_aX)=-  \sum_a \d I_a\dH\! f (X,I_aX)\\
 & =\sum_a \nabla^g_{I_aX}(I_a\di\!  f )X-\nabla^g_X(I_a\di\!  f )I_aX\\
 &=\sum_a \nabla^g_{I_aX}(\di\!  f )I_aX+3 g(\nabla^g_XX,X)
\end{split}
\end{equation*}
since 
$\nabla^g_X I_a$ vanishes on $\Omega^1\H$. At the same time $\{\vert X\vert^{-1}X,\vert X\vert^{-1}I_aX, a=1,2,3\}$ is an orthonormal frame in $\H$, away from the zero set of $X$, hence 
$$\sum_a \nabla^g_{I_aX}(\di\!  f )I_aX + \la  \nabla^g_XX,X \ra  =-\vert X \vert^2\di^{\star}\!\di\! f $$ on $M$. We conclude that 
\bea
&\int_M \sum_a g(\dH\!f \wedge I_a\dH\!f, \dH\!I_a\dH\! f)\vol_g=\int_M ( 2g(\nabla^g_XX,X) - \nu \vert X\vert^2  f)\vol_g = 0
\eea
by taking into account that $\int_M g( \nabla^g_XX,X ) \,  \vol =\frac{1}{2} \, \int_M g(\d \vert X\vert^2, \di\!f)\,  \vol =\frac{\nu}{2}\int_M f\vert X \vert^2 \, \vol $ and 
part (i).

\noindent
(iii) the integral under scrutiny splits as  
\begin{equation*}
\begin{split}
\int_M
 f  |  \dH\!\bbI\dH\!f |^2  \,  \vol&=\int_M  \sum_{a} \langle \dH(fI_a\dH\!f),\dH(I_a\dH \!  f) \rangle  \,  \vol\\
&-\int_M   \sum_{a}  \langle \dH\!f \wedge I_a\dH\!f, \dH\!I_a \dH\!  f \rangle \, \vol .
\end{split}
\end{equation*}
The first summand is computed from 
\bea
\int_M\langle \dH(fI_a\dH\!f),\dH(I_a \dH\!  f)\rangle \, \vol &=& \int_M \langle f I_a\dH\!f, \di_{\H}^{\star}\dH(I_a\d_\H \! \,   f) \ra \, \vol  \\
&=& (\nu-8)
\int_M f \vert \d\!f\vert^2  \,  \vol  = \tfrac{(\nu-8)\nu}{2} 
\int_Mf^3 \,  \vol 
\eea
after taking into account that $\di_{\H}^{\star}\dH(I_a \di\!  f)=\Delta_{\H}(I_a \dH\! f)=(\nu-8)I_a \dH\! f$ (see section \ref{HoEe} for similar arguments) and part (i).
The claim follows now from part (ii).
\end{proof}

\medskip

\begin{teo}\label{T81}  For any $f \in \ker(\Delta_b-24)$ we have
$$
\int_M P(\varepsilon(f),\varepsilon(f),\varepsilon(f)) \, \vol=c\int_M f^3 \, \vol \ \ \ \mbox{with} \ c \in \bbR, c \neq 0.
$$
\end{teo}
\begin{proof}
Recall that $\di_{\H}^{-}\bbI\dH\!f=\dH\bbI\dH\!f+\frac{\nu}{2}f\omega$ by Lemma \ref{batch1},(ii) where $\nu=24$. Taking this into account, Lemma \ref{int} leads to 
\begin{equation*}
\begin{split}
&\int_M \sum_a \la  \dH\!f \wedge I_a\dH\!f, \di_{H}^{-}\!I_a\dH\!  f \ra \,\vol=\frac{3\nu}{2}\int_Mf\vert \dH\!f\vert^2\vol=\frac{3\nu^2}{4}
\int_M f^3\vol\\
& \int_M f | \di_{H}^{-}\!\bbI\dH\! f|^2 \, \vol=\int_M f | \di_{H} \!\bbI\dH\! f|^2 \, \vol-\frac{3\nu^2}{2}\int_Mf^3\vol=-\nu(\nu+4)\int_M f^3\vol.
\end{split}
\end{equation*}
Plugging these into Lemma \ref{L8-3} leads to 
\begin{equation*}
\begin{split}
& \int_M P(\eta,\eta,\eta)\vol=-3t_1(70t_1^2+t_2^2\nu(\nu+4))\int_M f^3\vol, \ \ \int_M P(\gamma_2,\gamma_2,\eta)\vol=\frac{3\nu}{4}(22t_1-5t_2\nu)\int_M f^3\vol.
\end{split}
\end{equation*}
By \eqref{S81} it follows that  $\int_M P(\varepsilon(f),\varepsilon(f),\varepsilon(f))\vol=c\int_M f^3 \, \vol$ for the explicit constant 
$$ c=-3t_1(70t_1^2+t_2^2\nu(\nu+4))+\tfrac{9\nu}{4}(22t_1-5t_2\nu).
$$ From the numerical values $\nu=24, t_1=\tfrac{6}{\sqrt{5}}, t_2=\tfrac{\sqrt{5}}{2}$ we get $c=-\tfrac{33,264}{\sqrt{5}}<0$ and the claim is fully proved.
\end{proof}
$\\$
{\bf{Proof of part (iii) in Theorem \ref{main1}.}} Since $P$ is a totally symmetric cubic form we have 
$P(\varepsilon(f),\varepsilon(f),\varepsilon(h))=\frac{1}{3}\frac{\di}{\di t}_{\vert t=0}P(\varepsilon(f+th),\varepsilon(f+th),
\varepsilon(f+th)).$
By Theorem \ref{T81} it follows that 
$$\mathbf{K}(\varepsilon(f))\varepsilon(h)=\frac{c}{3}\frac{\di}{\di t}_{\vert t=0}\int_M
(f+th)^3\vol=c\int_Mf^2h\vol.$$ Thus $\mathbf{K}^{-1}(0)$ is given as stated in Theorem \ref{main1},(iii).
\section{The basic Lichnerowicz Laplacian} \label{TLL}
\subsection{The comparaison formula} \label{CF}
We work with the canonical variation $g_s$ of a $3$-Sasaki structure $(M^7,g,\xi)$.  In this section we let $s=\s5$ and we systematically suppress any reference to this parameter in relation to the Levi-Civita connection $\nabla$ of $g_s$ and its curvature tensor which is defined according to $R(X,Y)=\nabla_{[X,Y]}-[\nabla_X,\nabla_Y]$. Recall \cite{Besse} that the Lichnerowicz Laplacian of $g_s$ is explicitly defined via the Weitzenb\"ock type formula 
$$ \Delta_L^{g_s}=\nabla^{\star}\nabla-2\mathring{R}+2E$$
where the curvature action $\mathring{R}(h)(X,Y):=\sum_i g_s(R(X,E_i)Y,hE_i)$ for $h\in \Sym^2_0(M,g_s)$
and $\{E_i\}$ is some local orthonormal basis in $TM$. Here $E=54/5$ is the Einstein constant of $g_s$.

The base orbifold $N=M/\mathcal{F}$ is in general not smooth; nevertheless it has a well defined local geometry; we denote with $\pi:(M,g_s) \to (N,g_N)$ 
the orbifold Riemannian submersion and with $R_N$ the Riemann curvature tensor of the orbifold metric $g_N$. From the structure equations of the frame of Killing vector fields $Z^a$ in \eqref{dza} it follows that the curvature action $\mathring{R}$ preserves the subbundle $\Sym_0^2\H$ and satisfies 
$$ \mathring{R}h=(\pi^{\star}\mathring{R}_N)h+\tfrac{3}{5}h.
$$ 
for $h \in \Sym_0^2\H$. We define the basic Lichnerowicz Laplacian 
$$ \Delta_L^b:\Gamma_b(\Sym^2_0\H) \to \Gamma_b(\Sym^2_0\H), \ \ \Delta_L^b:=(\Delta_L^{g_s}+4s^2)_{\Sym^2_0\H}
$$
where the subscript indicates orthogonal projection, w.r.t. $g_s$, onto the space. Below we show $\Delta_L^b$ is indeed the lift of the Lichnerowicz Laplacian of the local base $(N^4,g_N)$.
\begin{lema} \label{bo}
Assuming that $q\in \Gamma_b(\Sym^2_0\H)$ we have 
\begin{equation*}
\Delta^{b}_Lq=\pi^{\star}(\Delta_L^{g_N}Q)
\end{equation*}
where the locally defined tensor $Q\in \Sym^2_0(N,g_N)$ satisfies $q=\pi^{\star}Q$.
\end{lema} 
\begin{proof}
We compare the connection Laplacians of $g_s$ respectively $g_N$. Recall that basic vector fields $X \in \Gamma_b(\H)$ satisfy 
$[X,Z_a]=0$ thus $\nabla_XZ_a=-sI_aX$. It follows that 
\begin{equation} \label{LC}
\begin{split}
\nabla_Xq=&\, \pi^{\star}(D_XQ)+s\sum_a (Z^a \otimes q(I_aX)+(q(I_aX))^{\flat} \otimes Z_a)\\
\nabla_{Z_a}q=&\,s(q\circ I_a-I_a \circ q)
\end{split}
\end{equation}
where $X$ is basic and $D$ is the Levi-Civita connection of $g_N$. Choose a local orthonormal basis 
$\{e_i\}$ in $\Gamma_b(\H)$; by a slight abuse of notation we identify $e_i$ and its projection onto $N$. Direct computation shows that the horizontal 
piece in $\sum_i \nabla_{e_i,e_i}^2q$ is given by  
$$ \sum_i\pi^{\star}(\nabla^2_{e_i,e_i}Q)+s\sum_{i,a} \nabla_{e_i}Z^a \otimes q(I_ae_i)+(q(I_ae_i))^{\flat} \otimes \nabla_{e_i}Z_a=
 \sum_i\pi^{\star}(D^2_{e_i,e_i}Q)-6s^2\pi^{\star}Q
$$
since $q$ is symmetric. At the same time $\nabla_{Z_a}I_a=0$
as routinely implied by the structure equations \eqref{str-o}, hence $ \sum_a \nabla^2_{Z_a,Z_a}q=-2s^2(3q+\sum_aI_aqI_a)
$ after differentiating in the second equation of \eqref{LC}. Since the map $\bfs :\Lambda^{-}(\H,\bbR^3) \to \Sym^2_0\H$ from \eqref{s-iso} is an isomorphism and endomorphisms in $\Lambda^{-}\H$ and $\Lambda^{+}\H$ commute it is straightforward to check that $\sum_aI_aqI_a=q$. Putting these facts together leads to 
$$(\nabla^{\star}\nabla q)_{\Sym^2_0\H}= \pi^{\star}(D^{\star}DQ+14s^2Q)$$ 
and the claim follows after taking into account the comparaison formula for the operators $\mathring{R}$ given above, together with the values for the Einstein constants of $g_s$ and $g_N$ which are $\frac{54}{5}$ and $12$.
\end{proof}
For Einstein Sasaki structures, where the canonical foliation has $1$-dimensional leaves this type of comparaison formula has been proved in \cite{coev2}, see proof of Lemma 2.6; see also \cite[sectn.4.1]{WangW2} for the more general setup of Einstein metrics fibered by circles. 
Lemma \ref{bo} prompts out the following interpretation for the space $\mathbf{H}_4^{-}$.  
\begin{pro} \label{car01}
The bundle isomorphism $\bfs$ induces an injection 
$$ \bfs : \mathbf{H}^{-}_4 \to \ker(\Delta_L^b-16) \cap \TT_b(\H).
$$
\end{pro}
\begin{proof}
Projection of eigentensors for $\Delta_L^{g_s}$ onto $\Sym^2_0\H$ and using Lemma \ref{bo} together with the definition of the basic Lichnerowicz Laplacian $\Delta_L^b$ guarantees the inclusion of eigenspaces $\ker(\Delta^{g_s}_L-\frac{76}{5}) \cap 
\TT_b(\H) \subseteq \ker(\Delta^{b}_L-16) \cap 
\TT_b(\H)$. The claim follows from Proposition \ref{harm1}.
\end{proof}
One can examine up to which extent this is an isomorphism; as this issue is not directly relevant here it is left for further research. 
\subsection{The Aloff-Wallach space} \label{AW}
%
%
We revisit here the Aloff-Wallach space $N(1,1)$ equipped with its proper nearly $\G_2$ structure $\varphi_{1/\sqrt{5}}$
as a very simple example for the general theory developed in this paper. The $3$-Sasaki structure on $M=N(1,1) \stackrel{\pi}{\to} N=
\overline{\mathbb{C}P}^2$ is regular, where $N$ is equipped with the Fubini-Study metric $g_{FS}$ with Einstein 
constant $12$ and canonical complex structure $J_{FS} \in \Lambda^{-}N$. By Lichnerowicz-Matsushima's theorem, the first non-zero eigenvalue of the scalar Laplacian on $(N,g_{FS})$ equals $24$ and the map given by 
$K \in \mathfrak{aut}(N,g_{FS}) \mapsto f_K \in \ker(\Delta^{g_{FS}}-24)$ is a linear isomorphism. The Killing potential $f_K$ is determined from $K \lrcorner \omega_{FS}=J_{FS}\di\!f_K$ and 
$\int_Nf_K \vol=0$.

The space of infinitesimal $\G_2$ deformations of $\varphi_{\s5}$ 
was computed by representation theory in \cite{AlS} and its rigidity was proved in \cite{NS}.  Applying  thms. \ref{main1} and \ref{main2} we obtain new short geometric 
proofs for these results. As a new result, we provide the full description of the space of unstable directions. 

\begin{teo} \label{N11}
Consider the Aloff-Wallach space $(N(1,1),\varphi_{\s5})$. The following hold 
\begin{itemize}
\item[(i)] the space of infinitesimal $\G_2$ deformations of $\varphi_{\s5}$ is isomorphic to $\su(3)$ via the map 
$$K \in \mathfrak{aut}(X,g_{FS}) \cong \su(3) \mapsto \varepsilon(f_K \circ \pi) \in \mathcal{E}(\varphi_{\s5})$$
\item[(ii)] the space of unstable directions for  $g_{\s5}$ is spanned by $\hh=4\,\id_{\V}-3\,\id_{\H}$
\item[(iii)] the nearly $\G_2$ structure $\varphi_{\s5}$ is rigid.
\end{itemize} 
\end{teo}
\begin{proof}
(i) follows directly from Theorem \ref{main1},(i).\\
(ii) As $\ker (\Delta_b - \nu)=0$ for $\nu < 24$ by Lichnerowicz-Matsushima, the space of unstable directions for $g_{\s5}$ is isomorphic to $\bbR \oplus \bfH^{-}_4$ by Theorem \ref{main2}. Since $g_{FS}$ on $\mathbb{CP}^2$ is linearly stable on $\TT$ tensors by \cite{Koiso1980}[Theorem 1.4] (see also \cite[Theorem 4.3]{CaoHe}) it follows that we must have $\ker(\Delta_{L}^{g_{FS}}-16) \cap \TT(g_{FS})=0$. Proposition \ref{car01} together with $\TT_b(\H)=\pi^{\star}\TT(g_{FS})$ thus ensures the vanishing of $\mathbf{H}_4^{-}$.
\\
(iii) First recall that the space of $\su(3)$-invariant, cubic polynomials on the Lie algebra $\su(3)$ is $1$-dimensional and generated by $\sigma_3(A)=\tr(A^3J_0)$ where $J_0$ is the standard complex structure on $\bbR^6$. Next, the map $K \in \mathfrak{aut}(N,g_{FS}) \stackrel{\mu_3}{\mapsto} \int_N f^3_K\vol$ defines an invariant cubic polynomial on the Lie algebra $\mathfrak{aut}(N,g_{FS})$. After identifying $\mathfrak{aut}(N,g_{FS})$ and $\su(3)$ via $A \mapsto K_A$ we thus find $C \in \bbR$ with $\mu_3(X_A)=C\sigma_3(A)$ for all $A \in \su(3)$. According 
to \cite[Lemma 9]{KS}, see also \cite{HMW} for a different argument using the Duistermaat-Heckman localisation formula, there exist Killing fields $K$ such that $\int_N f_K^3\vol \neq 0$, thus $C \neq 0$. Therefore 
the zero locus of $\mu_3$ is isomorphic to the zero locus of $\sigma_3$. If $A$ belongs to the latter zero locus, after polarising $\sigma_3$ it is easy to see that $A^2=\frac{\tr(A^2)}{6}\id$. So if $A \neq 0$ the symmetric matrix $AJ_0$ has eigenvalues $\pm t$ with $t>0$ and multiplicites $m_{\pm}$; those must be even since 
$AJ_0=J_0A$ and satisfy $m_{+}+m_{-}=6$ and $m_{+}=m_{-}$ due to $\tr(AJ_0)=0$. This is impossible in dimension $6$ hence $A=0$.

By Theorem \ref{main1}, (iii) it thus follows that $\mathbf{K}^{-1}(0)=\{0\}$ that is all 
non trivial infinitesimal $\G_2$ deformations are obstructed to second order hence 
the nearly $\G_2$ structure $\varphi_{\s5}$ is rigid.
\end{proof}

\end{document}